\documentclass[11pt]{article}
\usepackage{amsmath,amssymb,amscd,amsthm,verbatim}
\usepackage{enumerate}
\usepackage[all,2cell]{xy}
\usepackage[vcentermath]{youngtab}
\usepackage{url}
\usepackage[top=1in,bottom=1in,left=.91in,right=.91in]{geometry}
\usepackage[pdftex]{hyperref}
\usepackage{color}

\usepackage{calc}

\title{FI-modules and stability for representations of symmetric groups}

\author{Thomas Church, Jordan S. Ellenberg and Benson Farb \thanks{The authors 
    gratefully acknowledge support from the National Science
    Foundation.  The second author's work was partially supported by a Romnes Faculty Fellowship.}}

\theoremstyle{plain}
\newtheorem{theorem}{Theorem}[subsection]
\newtheorem{thm}[theorem]{Theorem}
\newtheorem{proposition}[theorem]{Proposition}
\newtheorem{prop}[theorem]{Proposition}
\newtheorem{lemma}[theorem]{Lemma}

\newtheorem{corollary}[theorem]{Corollary}

\newtheorem{question}[theorem]{Question}

\newtheorem{alphcorollary}{Corollary}[subsection]

\newtheorem{alphprob}[alphcorollary]{Problem}

\newtheorem{introtheorem}{Theorem}[section]
\newtheorem{introprob}[introtheorem]{Problem}
\newtheorem{introcorollary}[introtheorem]{Corollary}

\theoremstyle{definition}
\newtheorem{remark}[theorem]{Remark}
\newtheorem{definition}[theorem]{Definition}

\newtheorem{xample}[theorem]{Example}

\newtheorem{alphremark}[alphcorollary]{Remark}
\newtheorem{alphxample}[alphcorollary]{Example}
\newtheorem{introdefinition}[introtheorem]{Definition}
\newtheorem{alphintroremark}{Remark}[section]

\newcommand{\nc}{\newcommand}
\nc{\dmo}{\DeclareMathOperator}

\nc{\I}{\mathcal{I}}
\nc{\K}{\mathcal{K}}
\nc{\U}{\mathcal{U}}
\renewcommand{\L}{\mathcal{L}}
\nc{\Q}{\mathbb{Q}}
\nc{\R}{\mathbb{R}}
\nc{\Z}{\mathbb{Z}}
\nc{\C}{\mathbb{C}}
\nc{\N}{\mathbb{N}}
\nc{\F}{\mathbb{F}}
\nc{\cR}{\mathcal{R}}
\nc{\cM}{\mathcal{M}}
\nc{\cC}{\mathcal{C}}
\nc{\Schur}{\mathbb{S}}
\dmo{\GL}{GL}
\dmo{\Mat}{Mat}
\dmo{\PSL}{PSL}
\nc{\gin}{i}
\nc{\ga}{\Gamma}
\dmo{\Out}{Out}
\dmo{\Aut}{Aut}
\dmo{\Stab}{Stab}
\dmo{\wt}{weight}

\dmo\im{im}
\dmo{\Rep}{Rep}
\dmo\id{id}
\dmo\SL{SL}
\dmo\Sp{Sp}
\dmo\Mod{Mod}
\dmo\PMod{PMod}
\dmo\fd{fd}
\dmo\IA{IA}
\dmo\IAut{IA}
\nc{\IAn}{\IA_n}
\dmo\Sym{Sym}
\dmo\Ind{Ind}
\dmo\Res{Res}
\dmo\tr{tr}
\dmo\gr{gr}
\dmo\Free{Free}
\dmo\spn{span}
\dmo\End{End}
\dmo\Conf{Conf}
\dmo\op{op}
\dmo\coker{coker}

\dmo\Homeo{Homeo}
\dmo\Teich{Teich}

\dmo\rk{rank}
\dmo\ab{ab}
\dmo\rank{rank}
\dmo\Emb{Emb}
\dmo\Map{Map}

\dmo\Tor{\mathcal{T}}
\dmo\Torsion{Tor}

\def\fp{\mathfrak{p}}

\nc{\bwedge}{\textstyle{\bigwedge}}

\dmo\FI{FI}
\dmo\FB{FB}
\dmo\FIMod{FI-Mod}
\dmo\FBMod{FB-Mod}
\dmo\dMod{-Mod}
\def\kMod{k\dMod}

\dmo\alg{alg}

\dmo{\grd}{gr-}
\def\grkMod{\grd\!k\dMod}
\def\grFIMod{\grd\FIMod}

\nc\kT{k[T]}
\def\FIsharp{\FI\sharp}
\def\FIsharpMod{\FIsharp\dMod}
\dmo\coFI{co-FI}
\dmo\Top{Top}
\dmo\hTop{hTop}
\dmo\FIGroups{FI-Groups}
\dmo\Groups{Groups}

\nc\y[1]{{\tiny\yng(#1)}}
\nc\x{\hspace{0.1em}}

\newcommand{\m}{\mathbf{m}}
\newcommand{\n}{\mathbf{n}}
\newcommand{\npone}{\mathbf{n+1}}
\newcommand{\na}{\overline{\mathbf{a}}}

\newcommand{\kk}{\mathbf{k}}
\newcommand{\dd}{\mathbf{d}}

\newcommand{\ra}{\rightarrow}

\newcommand{\PP}{\mathbf{P}}
\renewcommand{\O}{\mathcal{O}}
\newcommand{\set}[1]{\{#1\}}
\newcommand{\beq}{\begin{displaymath}}
\newcommand{\eeq}{\end{displaymath}}
\newcommand{\beqn}{\begin{equation}}
\newcommand{\eeqn}{\end{equation}}

\newcommand{\tensor}{\otimes}

\DeclareMathOperator{\Hom}{Hom}
\DeclareMathOperator{\car}{char}

\renewcommand{\epsilon}{\varepsilon}
\nc{\coloneq}{\mathrel{\mathop:}\mkern-1.2mu=}
\newcommand*\cocolon{\nobreak \mskip6mu plus1mu \mathpunct{}\nonscript \mkern-\thinmuskip {:}\mskip2mu \relax}
\nc{\margin}[1]{\marginpar{\scriptsize #1}}
\nc{\para}[1]{\medskip\noindent\textbf{#1.}}
\nc{\hide}[1]{#1}

\nc{\into}{\hookrightarrow}
\nc{\onto}{\twoheadrightarrow}
\DeclareMathOperator*{\colim}{colim}
\DeclareMathOperator*{\hocolim}{hocolim}
\nc{\abs}[1]{\left\lvert#1\right\rvert}

\DeclareMathOperator{\stabD}{stab-deg}
\DeclareMathOperator{\injD}{inj-deg}
\DeclareMathOperator{\surjD}{surj-deg}
\DeclareMathOperator{\inv}{inv}
\DeclareMathOperator{\coinv}{coinv}
\DeclareMathOperator{\BGL}{BGL}
\nc{\stabT}{\preceq}
\nc{\uell}{\underline{\ell}}
\nc{\disjoint}{\sqcup}

\nc{\remove}[1]{#1} \nc{\antiremove}[1]{}

\nc{\arXiv}[1]{\href{http://arxiv.org/abs/#1}{arXiv:#1}}
\nc{\myemail}[1]{\href{mailto:#1}{\nolinkurl{#1}}}

\begin{document}

\maketitle
\begin{abstract}
In this paper we introduce and develop the theory of FI-modules.  We apply this theory to obtain new theorems about: 
\begin{itemize}
\item the cohomology of the configuration space of $n$ distinct ordered points on an arbitrary (connected, oriented) manifold;
\item the diagonal coinvariant algebra on $r$ sets of $n$ variables; 
\item the cohomology and tautological ring of the moduli space of $n$-pointed curves;
\item the space of polynomials on rank varieties of $n \times n$ matrices;
\item the subalgebra of the cohomology of the genus $n$ Torelli group generated by $H^1$;
\end{itemize}
and more.  The symmetric group $S_n$ acts on each of these vector spaces.   In most cases almost nothing is known about the characters of these representations, or even their dimensions.  We prove that in each fixed degree the character is given, for $n$ large enough, by a polynomial in the cycle-counting functions that is independent of $n$.  In particular, the dimension is eventually a polynomial in $n$.  In this framework, representation stability (in the sense of Church--Farb) for a sequence of $S_n$-representations is converted to a finite generation property for a single FI-module.

\end{abstract}

\tableofcontents

\section{Introduction}

\addtocounter{subsection}{1}

\bigskip

In this paper we develop a framework in which we can deduce strong constraints 
on naturally occurring sequences of $S_n$-representations using only elementary structural symmetries. These structural properties are encoded by objects we call \emph{FI-modules}. 

Let $\FI$ be the category whose objects are finite sets and whose morphisms are injections.  The category $\FI$ is equivalent to the category with objects $\n\coloneq\{1,\ldots,n\}$ for $n\in \N$ and morphisms from $\m$ to $\n$ being the injections $\m\into \n$.
\begin{introdefinition}[{\bf FI-module}]
\label{def:fimod}
An \emph{FI-module} over a commutative ring $k$ is a functor $V$ from $\FI$ to the category of $k$-modules.  We denote the $k$-module $V(\n)$ by $V_n$.
\end{introdefinition}

Since $\End_{\FI}(\n)=S_n$, any FI-module $V$ determines a sequence of $S_n$-representations $V_n$ with linear maps between them respecting the group actions.   One theme of this paper is the conceptual power of encoding this large amount of (potentially complicated) data into a single object $V$.

Many of the familiar notions from the theory of modules, such as submodule and quotient module, carry over to FI-modules.  In particular, there is a natural notion of \emph{finite generation} for FI-modules.

\begin{introdefinition}[{\bf Finite generation}]
An FI-module $V$ is \emph{finitely generated} if there is a finite set $S$ of elements in $\coprod_iV_i$ so that no proper sub-FI-module of $V$ contains $S$; see Definition~\ref{def:finitegen}.
\end{introdefinition}

It is straightforward to show that finite generation is preserved by quotients, extensions, tensor products, etc. Moreover, finite generation also passes to submodules when $k$ contains $\Q$,  so this is a robust property.   
\begin{introtheorem}[{\bf Noetherian property}]
\label{th:noetherian}
Let $k$ be a Noetherian ring containing $\Q$. 
The category of FI-modules over $k$ is \emph{Noetherian}: that is, any sub-FI-module of a finitely generated FI-module is finitely generated.
\end{introtheorem}
A theorem of Snowden~\cite[Theorem 2.3]{Snowden} can be used to give another proof of Theorem~\ref{th:noetherian}.  In a sequel to the present paper, the authors, in joint work with Rohit Nagpal, removed the restriction in Theorem~\ref{th:noetherian} that $\Q\subset k$, proving that the category of FI-modules over \emph{any} Noetherian ring is Noetherian \cite[Theorem A]{CEFN}.  

\para{Examples of FI-modules} Finitely generated FI-modules are ubiquitous. To illustrate this we present in Table~\ref{table:examples} a variety of examples of FI-modules that arise in topology, algebra, combinatorics and algebraic geometry. In the course of this paper we will prove that each entry in this list is a finitely generated FI-module.   The exact definitions of each of these objects will be given later in the paper. Any parameter here not equal to $n$ should be considered fixed and nonnegative. 

\begin{table}[h]
\centering
\begin{tabular}{ll}
\underline{FI-module $V=\{V_n\}$} \qquad\quad& \underline{Description} \\

&\\
$1. \ H^i(\Conf_n(M);\Q)$ & $\Conf_n(M)=$ configuration space of $n$ distinct ordered points \\
& on a connected, oriented manifold $M$ (\S \ref{section:configspaces})\\
&\\
2.\ $R^{(r)}_J(n)$
& $J=(j_1,\ldots j_r)$, $R^{(r)}(n)=\bigoplus_{J}R^{(r)}_J(n)$= $r$-diagonal \\
&   coinvariant 
algebra on $r$ sets of  $n$ variables (\S \ref{section:coinvariants}) \\

&\\
3.\ $H^i(\cM_{g,n};\Q)$ & $\cM_{g,n}=$ moduli space of $n$-pointed genus $g\geq 2$ curves\ \ (\S \ref{subsection:mgn})\\
&\\
4.\ $\cR^i(\cM_{g,n})$ & $i^{\text{th}}$ graded piece of tautological ring of $\cM_{g,n}$\    (\S \ref{subsection:mgn})\\
&\\
5. $\O(X_{P,r}(n))_i$ & space of degree $i$ polynomials on $X_{P,r}(n)$, the rank variety\\ 
& of 
$n\times n$  matrices of $P$-rank $\leq r$  (\S \ref{subsection:determinantal})\\
&\\
6.\ $G(A_n/\Q)_i$ & degree $i$ part of the Bhargava--Satriano Galois closure\\ & of $A_n=\Q[x_1,\ldots,x_n]/(x_1,\ldots,x_n)^2$ (\S\ref{section:BS})\\
&\\
7.  $H^i(\I_n;\Q)_{\rm alb}$ & degree $i$ part of the subalgebra of $H^\ast(\I_n;\Q)$ generated by \\
& $H^1(\I_n;\Q)$, where $\I_n=$ genus $n$ Torelli group (\S \ref{subsection:albanese})\\
&\\
8. $H^i(\IA_n;\Q)_{\rm alb}$ & degree $i$ part of the subalgebra of $H^\ast(\IA_n;\Q)$ generated by \\
& $H^1(\IA_n;\Q)$, where $\IA_n=$ Torelli subgroup of $\Aut(F_n)$ (\S \ref{subsection:albanese})\\
&\\
9.\ $\gr(\Gamma_n)_i$ &  $i^{th}$ graded piece of associated graded Lie algebra of many \\
& groups $\Gamma_n$, including $\I_n$, $\IA_n$ and pure braid group $P_n$ (\S\ref{section:gradedlielcc})\\

\end{tabular}
\caption{Examples of finitely generated FI-modules}
\label{table:examples}
\end{table}

%

\para{Character polynomials}
Each of the vector spaces in Table~\ref{table:examples} admits a natural action of the symmetric group $S_n$.  As noted above, any FI-module $V$ provides a linear action of $\End_{\FI}(\n) = S_n$ on $V_n$.   If $V$ is finitely generated, the representations $V_n$ satisfy strong constraints, as we now explain.

For each $i\geq 1$ and any $n\geq 0$, let $X_i\colon S_n\to \N$ be the class function defined by 
\[X_i(\sigma)= \mbox{number of $i$-cycles in the cycle decomposition of $\sigma$}.\]
Polynomials in the variables $X_i$ are called \emph{character polynomials}.  Though perhaps not widely known, the study of character polynomials goes back to work of Frobenius, Murnaghan, Specht,  and Macdonald (see e.g.\  \cite[Example I.7.14]{Mac}).

It is easy to see that for any $n\geq 1$ the vector space of class functions on $S_n$ is spanned by character polynomials, so the character of any representation can be described by such a polynomial. 
For example, if $V\simeq \Q^n$ is the standard permutation representation of $S_n$, the character $\chi_V(\sigma)$ is the number of fixed points of $\sigma$, so $\chi_V=X_1$.  If $W=\bwedge^2V$ then $\chi_W=\displaystyle{\binom{X_1}{2}-X_2}$, since $\sigma\in S_n$ fixes those basis elements $x_i\wedge x_j$  for which the cycle decomposition of $\sigma$ contains the pair of fixed points $(i)(j)$, and negates those for which it contains the $2$-cycle $(i\ j)$.

One notable feature of these two examples is that the same polynomial describes the characters of an entire family of similarly-defined $S_n$-representations, one for each $n\geq 0$.  Of course, the expression of a class function on $S_n$ as a character polynomial is not unique; for example $X_N$ vanishes identically on $S_n$ for all $N > n$.  However, two character polynomials that agree on $S_n$ for \emph{infinitely many} $n$ must be equal, so for a sequence of class functions $\chi_n$ on $S_n$ it makes sense to ask about ``the'' character polynomial, if any, that realizes $\chi_n$.
\begin{introdefinition}[{\bf Eventually polynomial characters}]
A sequence $\chi_n$ of characters of $S_n$ is \emph{eventually polynomial} if there exist integers $r$, $N$ and a character polynomial $P(X_1,\ldots,X_r)$ such that 
\[
\qquad\qquad\chi_n(\sigma)=P(X_1,\ldots ,X_r)(\sigma)\quad \  \ {\rm for \ all}\  n\geq N\text{ and all }\sigma\in S_n.
\]
The \emph{degree} of the character polynomial $P(X_1,\ldots,X_r)$ is defined by setting $\deg(X_i)=i$.
\end{introdefinition}

One of the most striking properties of FI-modules in characteristic~0 is that the characters of any finitely generated FI-module have such a uniform description.

\begin{introtheorem}[{\bf Polynomiality of characters}] 
\label{thm:intro:persinomial}
\label{thm:intro:polydim}
Let $V$ be an FI-module over a field of characteristic~$0$. If $V$ is finitely generated then the sequence of characters $\chi_{V_n}$ of the $S_n$-representations $V_n$ is eventually polynomial.  
In particular, $\dim V_n$ is eventually polynomial.
\end{introtheorem}
The \emph{character polynomial} of a finitely generated FI-module is the polynomial $P(X_1,\ldots,X_r)$ which gives the characters $\chi_{V_n}$. It will always be integer-valued, meaning that $P(X_1,\ldots,X_r)\in \Z$ whenever $X_1,\ldots,X_r\in \Z$.
In situations of interest one can typically produce an explicit upper bound on the degree of the character polynomial by computing the \emph{weight} of $V$ as in Definition~\ref{de:weight}. In particular this gives an upper bound for the number of variables $r$. Moreover, computing the \emph{stability degree} of $V$ (\S\ref{section:stabilitydegree}) gives explicit bounds on the range $n\geq N$ where $\chi_{V_n}$ is given by the character polynomial. This converts the problem of finding all the characters $\chi_{V_n}$ into a concrete finite computation.  Note that the second claim in Theorem~\ref{thm:intro:persinomial} follow from the first since
\[\dim V_n=\chi_{V_n}({\rm id})=P(n,0,\ldots,0).\]

One consequence of Theorem~\ref{thm:intro:persinomial} is that $\chi_{V_n}$ only depends on ``short cycles'', i.e.\ on cycles of length $\leq r$.  This is a highly restrictive condition when $n$ is much larger than $r$. For example, the proportion of permutations in $S_n$ that have \emph{no} cycles of length $r$ or less is bounded away from $0$ as $n$ grows, and $\chi_{V_n}$ is constrained to be \emph{constant} on this positive-density subset of $S_n$.

Except for a few special (e.g.\ $M=\R^d$) and low-complexity  (i.e.\ small $i$, $d$, $g$, $J$, etc.) cases, little is known about the characters of the $S_n$-representations in Table~\ref{table:examples}, or even their dimension.  In many cases, closed form computations seem  out of reach.  By contrast, the following result gives an answer, albeit a non-explicit one, in every case.  Since each sequence (1)--(9) comes from a finitely generated FI-module, Theorem~\ref{thm:intro:persinomial} applies to each sequence, giving the following.

\begin{introcorollary}
\label{cor:polydim:app}
The characters $\chi_{V_n}$ of each of the sequences (1)--(9) of $S_n$-representations in Table~\ref{table:examples} are eventually polynomial.  In particular $\dim(V_n)$ is eventually polynomial.
\end{introcorollary}
\pagebreak

Apart from a few special cases, we do not know how to specify the polynomials produced by Corollary~\ref{cor:polydim:app}, although we can give explicit upper bounds for their degree.   A primary obstacle is that our theorems on finite generation depend on the Noetherian property of FI-modules proved in Theorem~\ref{th:noetherian}, and such properties cannot in general be made effective.

As a contrasting example, the dimension of $H^2(\overline{\cM}_{g,n};\Q)$ grows exponentially with $n$, where $\overline{\cM}_{g,n}$ is the Deligne-Mumford compactification of the moduli space of $n$-pointed genus $g$ curves.   Although the cohomology groups $H^2(\overline{\cM}_{g,n};\Q)$ do form an FI-module, this FI-module is not finitely generated.

\para{$\FIsharp$-modules}
It is often the case that FI-modules arising in nature carry an even more rigid structure.  An 
\emph{$\FIsharp$-module} is a functor from the category of \emph{partial injections} of finite sets to the category of $k$-modules (see \S\ref{ss:fisharp}).
In contrast with the category of FI-modules, the category of $\FIsharp$-modules is close to being 
semisimple (see Theorem~\ref{th:charFIsharp} for a precise statement).
\begin{introtheorem}
\label{thm:intro:FIsharp}
Let $V$ be an $\FIsharp$-module over any field $k$.  The following are equivalent:
\begin{enumerate}
\item $\dim(V_n)$ is bounded above by a polynomial in $n$.
\item $\dim(V_n)$ is exactly equal to a polynomial in $n$ for all $n\geq 0$.
\end{enumerate}
\end{introtheorem}
The power of Theorem~\ref{thm:intro:FIsharp} comes from the fact that in practice it is quite easy to prove that $\dim(V_n)$ is bounded above by a polynomial.  The rigidity behind this theorem holds even for $\FIsharp$-modules over $\Z$ or other rings. 
When $k$ is a field of characteristic~0, we strengthen Theorem~\ref{thm:intro:persinomial} to show that the character of $V_n$ is given by a single character polynomial for all $n\geq 0$.

We now focus in greater detail on two of the most striking applications of our results.  Many other applications are given in Sections~\ref{section:FIalgebrasandcoinv}, \ref{section:configspaces} and \ref{section:moduli}.

\para{Cohomology of configuration spaces}
 In \S\ref{section:configspaces} we prove a number of new theorems about configuration spaces on manifolds.
Let $\Conf_n(M)$ denote the configuration space of ordered $n$-tuples of distinct points in a space $M$: \[\Conf_n(M)\coloneq \big\{(p_1,\ldots,p_n)\in M^n\,\big|\, p_i\neq p_j\big\}\]
Configuration spaces and their cohomology are of wide interest in topology and in algebraic geometry; 
for a sampling, see Fulton-MacPherson~\cite{FuM}, McDuff~\cite{McD}, or Segal~\cite{Se}.

An injection $f\colon \m\into\n$ induces a map $\Conf_n(M)\to \Conf_m(M)$ sending
$(p_1,\ldots,p_n)$ to $(p_{f(1)},\ldots,p_{f(m)})$.
This defines a {contravariant} functor $\Conf(M)$ from $\FI$ to the category of topological spaces. 
Thus for any fixed $i\geq 0$ and any ring $k$, we obtain an FI-module $H^i(\Conf(M);k)$.  
Using work of Totaro~\cite{To}
we prove that when $M$ is a connected, oriented compact manifold of dimension $\geq 2$, the FI-module $H^i(\Conf(M);\Q)$ is finitely generated for each $i\geq 0$.  Moreover, when $\dim M\geq 3$ we bound the weight and stability degree of this FI-module to prove the following.   
\begin{introtheorem}
\label{thm:confncharpoly}
If $\dim M\geq 3$, there is a character polynomial $P_{M,i}$ of degree $\leq i$ so that
\[\chi_{H^i(\Conf_n(M);\Q)}(\sigma)=P_{M,i}(\sigma)\quad\text{for all }n\geq 2i\text{ and all }\sigma\in S_n.\]
In particular, the Betti number $b_i(\Conf_n(M))$ agrees with a polynomial of degree $i$ for all $n\geq 2i$.
\end{introtheorem}
If $M$ is the interior of a compact manifold with nonempty boundary, we prove that $H^i(\Conf(M);k)$ is in fact an $\FIsharp$-module for any ring $k$. This implies that the character polynomial $P_{M,i}$ from Theorem~\ref{thm:confncharpoly} agrees with the character of $H^i(\Conf_n(M);\Q)$ for \emph{all} $n\geq 0$. It also implies 
sharp constraints on the integral and mod-$p$ cohomology of $\Conf_n(M)$.\pagebreak

\begin{introtheorem}
\label{thm:intro:openbetti}
If $M$ is the interior of a compact manifold with nonempty boundary, each of the following invariants of $\Conf_n(M)$ is given by a polynomial in $n$ for all $n\geq 0$ (of degree $i$ if $\dim M\geq 3$, and of degree $2i$ if $\dim M=2$):
\begin{enumerate}
\item the $i$-th rational Betti number $b_i(\Conf_n(M))$, i.e.\ $\dim_\Q\Conf_n(M;\Q)$;
\item the $i$-th mod-$p$ Betti number of $\Conf_n(M)$, i.e.\ $\dim_{\F_p}\Conf_n(M;\F_p)$;
\item the minimum number of generators of $H^i(\Conf_n(M);\Z)$;
\item the minimum number of generators of the $p$-torsion part of $H^i(\Conf_n(M);\Z)$.
\end{enumerate}
\end{introtheorem}
We believe that each of these results is new.
Our theory also yields a new proof of \cite[Theorem~1]{Ch}, which was used by Church~\cite{Ch} to give the first proof of rational homological stability for unordered configuration spaces of arbitrary manifolds. 
Our proof here is in a sense parallel to that of \cite{Ch}, but the new framework simplifies the mechanics of the proof considerably, and allows us to sharpen the bounds on the stable range.
When $\dim(M)>2$ we improve Church's stable range from \cite[Theorem 5]{Ch} on homological stability for configurations of ``colored points'' in $M$.

As a simple illustration of the above results, the character of the $S_n$-representation $H^2(\Conf_n(\R^2);\Q)$ is given for all $n\geq 0$  by the single character polynomial
\[\chi_{H^2(\Conf_n(\R^2);\Q)}=2\binom{X_1}{3}+3\binom{X_1}{4}+\binom{X_1}{2}X_2-\binom{X_2}{2}-X_3-X_4.\]

\para{Diagonal coinvariant algebras}
In \S\ref{section:coinvariants} we obtain new results about a well-studied object in algebraic combinatorics: the multivariate diagonal coinvariant algebra.  The story begins in classical invariant theory.    Let $k$ be a field of characteristic~0, and fix $r\geq 1$.  For each $n\geq 0$ we consider the algebra of polynomials 
\[k[{\bf X}^{(r)}(n)]\coloneq k[x_1^{(1)},\ldots ,x_n^{(1)},\ldots ,x_1^{(r)},\ldots ,x_n^{(r)}]\] in $r$ collections of $n$ variables.  The symmetric group $S_n$ acts on this algebra diagonally:
\[\sigma\cdot x_j^{(i)} \coloneq x_{\sigma(j)}^{(i)}\]
 Chevalley and Weyl computed the $S_n$-invariants under this action, the so-called \emph{multisymmetric polynomials} (see \cite[II.A.3]{Weyl}).   Let $I_n$ be the the ideal in $k[{\bf X}^{(r)}(n)]$ generated by the
multisymmetric polynomials with vanishing constant term.    The \emph{$r$-diagonal coinvariant algebra} is the $k$-algebra 
\[R^{(r)}(n)\coloneq k[{\bf X}^{(r)}(n)]/I_n.\]

Each coinvariant algebra $R^{(r)}(n)$ is known to be a finite-dimensional $S_n$-representation.    These representations have been objects of intense study in algebraic combinatorics.   Borel proved that $R^{(1)}(n)$ is isomorphic as an $S_n$-representation to the cohomology $H^*(\GL_n\C/B;k)$ of the complete flag variety $\GL_n\C/B$. Furthermore Chevalley \cite[Theorem B]{Chevalley} proved that $R^{(1)}(n)$ is isomorphic to the regular representation of $S_n$, so $\dim R^{(1)}(n)=n!$.   Haiman~\cite{Hai} gave a geometric interpretation for  $R^{(2)}(n)$ and used it to prove the ``$(n+1)^{n-1}$ Conjecture'', which stated that
\beq
\dim(R^{(2)}(n))=(n+1)^{n-1}.
\eeq
For $r>2$ the dimension of $R^{(r)}(n)$ is not known.  

The polynomial algebra $k[{\bf X}^{(r)}(n)]$ naturally has an $r$-fold multi-grading, where a monomial has multi-grading $J=(j_1,\ldots,j_r)$ if its total degree in the variables $x_1^{(k)},\ldots,x_n^{(k)}$ is $j_k$. This multi-grading is $S_n$-invariant, and descends to an $S_n$-invariant multi-grading
\[R^{(r)}(n)=\bigoplus_JR^{(r)}_J(n)\]
on the $r$-diagonal coinvariant algebra $R^{(r)}(n)$. 
It is a well-known problem to describe these graded pieces as $S_n$-representations.

\begin{introprob}
 \label{problem:megaprob}
For each $r\geq 2, n\geq 1$, and $J=(j_1,\ldots ,j_r)$, compute the character of each $R^{(r)}_J(n)$ as an  
$S_n$-representation, at least for $n$ sufficiently large.   In particular, find a formula for $\dim(R^{(r)}_J(n))$.  
 \end{introprob}

For $r=1$ this problem was solved
independently by Stanley, Lusztig, and Kraskiewicz--Weyman (see \cite[\S7.1]{CF}).  
For $r=2$ Haiman~\cite{HaimanInventiones} gave a formula for these characters in terms of 
Macdonald polynomials and a ``rather mysterious operator'' (see \cite{HHLRU} for a discussion).  
 In the lowest degree cases the computation is elementary.  For example, it is easy to check that:
\begin{equation}
\label{eq:coinvariantexamples}
\begin{array}{rlrll}
\dim R^{(1)}_1(n)=&n-1\qquad\quad&\chi_{R^{(1)}_1(n)}=&X_1-1&\text{for }n\geq 1\\
\dim R^{(1)}_2(n)=&\!\!\displaystyle{\binom{n}{2}}-1&\chi_{R^{(1)}_2(n)}=&\!\!\displaystyle{\binom{X_1}{2}}+X_2-1\quad&\text{for }n\geq 2\\
\dim R^{(2)}_{11}(n)=&2\displaystyle{\binom{n}{2}}-n&\chi_{R^{(2)}_{11}(n)}=&\displaystyle{2\binom{X_1}{2}-X_1}&\text{for }n\geq 2
\end{array}
\end{equation}

Note that in these low-degree cases, once $n$ is sufficiently large the dimension of $R_J^{(r)}(n)$ is polynomial in $n$, as is its character.  For  $r>2$ it seems that almost nothing is known about Problem~\ref{problem:megaprob}, except in low-degree cases (see \cite[\S4]{Be} for a discussion of these cases).  However, the descriptions in \eqref{eq:coinvariantexamples} are just the simplest examples of a completely general phenomenon. 
\begin{introtheorem}[{\bf Characters of $R^{(r)}_J(n)$ are eventually polynomial in $n$}]
\label{theorem:intro:polynomial:answer}
For any fixed $r\geq 1$ and $J=(j_1,\ldots ,j_r)$, the characters $\chi_{R^{(r)}_J(n)}$ are eventually polynomial in $n$ of degree at most $|J|$. 
In particular there exists a polynomial $P_J^{(r)}(n)$ of degree at most $|J|$ such that \[\dim(R^{(r)}_J(n)) = P^{(r)}_J(n)\text{ for all }n\gg 0.\]
\end{introtheorem}

We prove that  (the dual of) $R_J^{(r)}$ is a finitely generated FI-module, and deduce Theorem~\ref{theorem:intro:polynomial:answer} from Theorem~\ref{thm:intro:persinomial}.
It is clear \emph{a priori} that $\dim(R^{(r)}_J(n))$ grows no \emph{faster} than $O(n^{|J|})$, but the fact that this dimension eventually coincides exactly with a polynomial is new, as is the polynomial behavior of the character of $R^{(r)}_J(n)$. We emphasize that beyond this bound on the degree, Theorem~\ref{theorem:intro:polynomial:answer} gives no information on the polynomials $P^{(r)}_{J}$(n).

\begin{introprob}[{\bf Character polynomials for diagonal coinvariant algebras}]
Describe the polynomials $P^{(r)}_{J}(n)$ whose existence is guaranteed by Theorem~\ref{theorem:intro:polynomial:answer}.
\end{introprob}\pagebreak

\para{Murnaghan's theorem} Given a field of characteristic~0 and a partition $\lambda=(\lambda_1,\ldots ,\lambda_\ell)$,  for any $n\geq
\abs{\lambda}+\lambda_1$ we define $V(\lambda)_n$ to be the irreducible representation of $S_n$ corresponding to the partition 
\[\lambda[n]\coloneq (n-\abs{\lambda},\lambda_1,\ldots,\lambda_\ell).\] 

Murnaghan's theorem states that, for any partitions $\lambda$ and $\mu$, there are coefficients $g^{\nu}_{\mu,\lambda}$ such that the tensor product $V(\lambda)_n\tensor V(\mu)_n$ decomposes into $S_n$-irreducibles as
\beq
V(\lambda)_n \tensor V(\mu)_n = \bigoplus_{\nu} g_{\lambda,\mu}^\nu V(\nu)_n
\eeq
for all sufficiently large $n$. The first complete proof of this theorem was given by Littlewood~\cite{Li} in 1957, but the coefficients $g^{\nu}_{\lambda,\mu}$ remain unknown in general.  We will show in \S\ref{section:murnaghan} that Murnaghan's theorem is an easy consequence of the fact that the tensor product of finitely generated FI-modules is finitely generated.  From this point of view, Murnaghan's theorem is not merely an assertion about a list of numbers, but becomes a structural statement about a  single mathematical object: the FI-module $V(\lambda)\otimes V(\mu)$.

\para{Connection with representation stability}  In  \cite{CF}, Church and Farb introduced the theory of \emph{representation stability}. Stability theorems in topology and algebra typically assert that in a given sequence of vector spaces with linear maps
\beq
\cdots \ra V_n \ra V_{n+1} \ra V_{n+2} \ra \cdots
\eeq
the maps $V_n\to V_{n+1}$ are isomorphisms for $n$ large enough.
The goal of representation stability is to provide a framework for generalizing these results to situations when each vector space $V_n$ has an action of the symmetric group $S_n$ (or other natural families of groups).  Representation stability provides a formal way of saying that the ``names'' of the $S_n$--representations $V_n$ stabilize, and a language to describe this stabilization rigorously; see \S\ref{section:repstab} below for the precise definition.  Representation stability has been proved in many cases, and gives new conjectures in others; see \cite{CF} and \cite{Ch}, as well as \cite{J1,WiRepstab}.  The theory of FI-modules converts representation stability into  a finite generation property.

\begin{introtheorem}[{\bf Finite generation vs.\ representation stability}]
\label{theorem:intro:FIequiv}
\label{th:FIequiv}
An FI-module $V$ over a field of characteristic~0 is finitely generated if and only if the sequence $\{V_n\}$ of $S_n$-representations is uniformly representation stable in the sense of \cite{CF} and each $V_n$ is finite-dimensional. In particular, for any finitely generated FI-module $V$, we have for sufficiently large $n$ a decomposition:
\[V_n\simeq \bigoplus c_\lambda V(\lambda)_n\]
where the coefficients $c_\lambda$ do not depend on $n$.
\end{introtheorem}

In the language of \cite{CF}, the new result here is that ``surjectivity'' implies ``uniform multiplicity stability'' for FI-modules.  This turns out to be very useful, because in practice finite generation is much easier to prove than representation stability. A key ingredient in this theorem is the ``monotonicity'' proved by the first author in \cite[Theorem 2.8]{Ch}, and as a byproduct of the proof we obtain that finitely generated FI-modules are monotone in this sense. Once again, the stability degree defined in Section~\ref{section:stabilitydegree} allows us, in many cases of interest, to replace ``sufficiently large $n$'' with an explicit range.\pagebreak

Our new point of view also simplifies the description of many representation-stable sequences.  As a simple example, in \cite{CF} we showed that
\beq
H^2(\Conf_n(\R^2);\Q)=V(1)_n^{\oplus 2} \oplus V(1,1)_n^{\oplus 2}  \oplus V(2)_n^{\oplus 2} \oplus V(2,1)_n^{\oplus 2} \oplus V(3)_n \oplus V(3,1)_n
\eeq
for all $n \geq 7$, and separate descriptions were necessary for smaller $n$.  In the language of the present paper we can simply write
\begin{equation}
\label{eq:h2conf1}
H^2(\Conf(\R^2);\Q) = M(\y{2,1}) \oplus M(\y{3,1})
\end{equation}
to simultaneously describe $H^2(\Conf_n(\R^2);\Q)$ for \emph{all} $n\geq 0$ (see \S\ref{sec:free} for notation, and \S\ref{section:configspaces} for the proof). One appealing feature here is that even the ``unstable'' portions of the sequence 
$H^2(\Conf_n(\R^2);\Q)$, meaning the qualitatively different descriptions that are necessary for $n<7$, are already encoded in the right-hand side of \eqref{eq:h2conf1}.
Based on computer calculations, John Wiltshire-Gordon (personal communication) has formulated a precise conjecture 
for the decomposition of
$H^i(\Conf(\R^2);\Q)$ as in \eqref{eq:h2conf1} for all $i\geq 0$.  Our results give a decomposition 
as in \eqref{eq:h2conf1} with $\R^2$ replaced by any open manifold, but we do not in general know the right-hand side explicitly.

\para{Relation with other work} We record here some relationships between the material in this paper and the work of others, both before and after this paper was first posted.

\begin{itemize}
\item Modules over EI-categories, namely those where all endomorphisms are isomorphisms, have been previously studied in the context of transformation groups (see e.g.\ \cite{Lue}). In particular, the analogue of Theorem~\ref{th:noetherian} with $\FI$ replaced by a \emph{finite} EI-category was proved by L\"{u}ck in \cite[Lemma 16.10b]{Lue}; however, he has explained to us that his methods cannot be extended to infinite categories such as $\FI$. Building on the present paper, Gan--Li~\cite{GanLi} have generalized Theorem~\ref{th:noetherian} to a wide class of infinite EI-categories, and Snowden--Sam~\cite{SnowdenSamGrobner} have extended the generalization \cite[Theorem A]{CEFN} to an even wider class of combinatorial categories.

\item The classification of $\FIsharp$-modules in Theorem~\ref{th:charFIsharp} below parallels the main theorem of Pirashvili in \cite{Pi}, which establishes a similar classification for modules over a larger category (that of finite pointed sets). Pirashvili's work was later extended by S{\l}omi\'{n}ska \cite{Slominska} and Helmstutler \cite{Helmstutler}, who give general axiomatizations of ``EI-categories with factorizations'' for which such a ``Dold--Kan'' classification theorem holds. In particular, using the ``factorization'' $\FIsharp=\FI\circ \FB\circ \FI^{\op}$, the equivalence of $\FIsharp$-modules and $\FB$-modules in Theorem~\ref{th:charFIsharp} can be deduced from \cite[Theorem~1.5(ii.4)]{Slominska} or \cite[Example~4.7 and Corollary~7.3]{Helmstutler}. We are grateful to Nicholas Kuhn for informing us of their work.

\item The category of FI-modules, along with related abelian categories, has been considered by researchers in the field of polynomial functors. See for example the recent work of Djament and Vespa~\cite{DV} which studies the stable behavior of the cohomology groups $H^i(S_n,V_n)$ for FI-modules $V$, or the earlier results of Pirashvili~\cite{Pi}.

\item The recent work of  Snowden and Sam on ``twisted commutative algebras'' has some overlap with our own, because FI-modules can be viewed as modules for the ``exponential'' twisted commutative algebra (see \cite{Aguiar}; for an introduction to TCAs, see \cite{SSsurvey}). In particular, as mentioned above, Snowden proves in \cite[Theorem 2.3]{Snowden} a Noetherian property for modules over a broad class of TCAs in characteristic $0$, which could be used to give a different proof of Theorem~\ref{th:noetherian}.
Applying this perspective to $\FIsharp$-modules yields examples of (divided power) $D$-modules both in characteristic~0 and in positive characteristic, as we hope to explain in a future paper.
More recently, Sam--Snowden have given in \cite{SS} a more detailed analysis of the algebraic structure of the category of FI-modules in characteristic~0.

\item Objects in Deligne's category $\Rep(S_t)$, where $t$ is a complex number, are closely related to sequences of $S_n$-representations whose characters are  polynomial (see Deligne~\cite{De}, Knop~\cite{Kn}, or Etingof's lecture~\cite{Et}).  Is there a sense in which a finite-type object of $\Rep(S_t)$ can be specialized to a finitely generated FI-module?  An interesting example is provided by the recent work of Ren and Schedler~\cite{RS} on spaces of invariant differential operators on symplectic manifolds.  Their results are consistent with the proposition that their sequence $\mbox{Inv}_n(V)$ carries the structure of a finitely generated FI-module.  Does it? 

\item A recent paper of Putman~\cite{Pu} proposes an alternative representation stability condition called ``central stability'' and proves that it holds for the mod-$p$ cohomology of congruence groups. We  prove with Rohit Nagpal in \cite{CEFN} that in the language of FI-modules, central stability is equivalent to finite presentation. We also extend Theorem~\ref{th:noetherian} 
from characteristic 0 to  any Noetherian ring in \cite{CEFN}. Therefore any finitely generated FI-module is finitely presented, so finite generation for FI-modules, representation stability for $S_n$-representations (in characteristic 0), and central stability are all equivalent.

\item In \cite{CEF2}, we apply Theorem~\ref{th:noetherian} to the \'{e}tale cohomology of hyperplane arrangements to prove asymptotic stabilization for certain statistics for polynomials over finite fields. In \cite{CP}, Church--Putman extend the notion of finite generation from FI-modules to nonabelian FI-groups (see \S\ref{subsection:albanese}), and use this to construct generating sets for the Johnson filtration of the Torelli group with bounded complexity. In \cite{CE}, Church--Ellenberg prove a quantitative bound on FI-homology which says that in an approximate sense, FI-modules behave as if the homological dimension of FI were 1. The results of this paper on $S_n$-representations have been extended to other Weyl groups by Wilson by constructing and analyzing the category of $\FI_W$-modules~\cite{JennyFIW1,JennyFIW2}. The analogous categories where $S_n$ is replaced by $\GL_n(R)$ for a finite ring $R$ have been studied by Putman--Sam~\cite{PutmanSam}; modules over these categories turn out to have many features in common with FI-modules, including the Noetherian property of Theorem~\ref{th:noetherian} and \cite[Theorem A]{CEFN}.  

\end{itemize}

\begin{alphintroremark}
The arXiv version of this paper is slightly expanded from the published version that will appear in \emph{Duke Math.\ J}. All theorems, propositions, remarks, and so on are numbered consistently across both versions, so that for example Theorem~\ref{thm:persinomial} refers to the same theorem in both versions; those that appear only in the arXiv version, such as Remark~\ref{rem:projective}, are numbered separately.
\end{alphintroremark}

\para{Acknowledgements} We are grateful to Wolfgang L\"{u}ck for conversations regarding his work in \cite{Lue} and its relation with the present paper. We thank Persi Diaconis for many helpful conversations regarding the polynomials considered in Theorem~\ref{thm:persinomial}. We also thank Marcelo Aguiar, Ralph Cohen, Pierre Deligne, Aur\'{e}lien Djament, Pavel Etingof,  James Griffin, Jim Haglund, Rita Jimenez Rolland, Nicholas Kuhn, Aaron Lauve, Daniel Litt, Jacob Lurie, Andrew Putman, Claudiu Raicu, Steven Sam, Andrew Snowden, Bernd Sturmfels, David Treumann, Julianna Tymoczko, Christine Vespa, Jennifer Wilson, and John Wiltshire-Gordon for helpful discussions. We are especially grateful to Rosona Eldred and the Hamburg/Bremen reading group for numerous helpful comments on an earlier version of this paper.  Finally, we would like to thank the referees, whose comments and suggestions have greatly improved the exposition in this paper.  

\section{FI-modules:  basic properties}
\label{section:setup}

In this section we develop the basic theory of FI-modules, including definitions and foundational properties. 

\subsection{FI-objects} 
\label{ss:notation}
Fix a commutative ring $k$. The case when $k$ is a field will be foremost in our minds, and our notation is chosen accordingly. For example, by a \emph{representation over $k$ of a group $G$} we mean a $k$-module $V$ together with an action of $G$ on $V$ by $k$-module automorphisms, i.e.\ a $kG$-module.  The \emph{dual representation} $V^*$ is the $k$-module $\Hom_k(V,k)$ together with its induced $G$-action.

Given two groups $G$ and $H$, if $V$ is a representation of $G$, then we denote by 
$V\boxtimes k$ the same $k$-module $V$ interpreted as a representation of $G\times H$ on which $H$ acts trivially. All tensor products are taken over $k$ unless otherwise specified.


\para{FI-modules}
Recall from Definition~\ref{def:fimod} that $\FI$ denotes the category whose objects are \textbf{f}inite sets and whose morphisms from $S$ to $T$ are \textbf{i}njections $S\into T$. We will later consider three other categories with the same objects: $\FB$ will denote the category whose morphisms from $S$ to $T$ are the bijections $S\xrightarrow{\approx}T$; $\coFI$ will denote the opposite category $\FI^{\text{op}}$, whose morphisms from $S$ to $T$ are the injections $T\into S$; and $\FIsharp$ will denote the category whose morphisms from $S$ to $T$ are the \emph{partial injections} from $S$ to $T$ (see \S\ref{ss:fisharp} for details).

\begin{definition}
The category of FI-modules is the category $\FIMod\coloneq [\FI,\kMod]$ of functors from $\FI$ to the category of $k$-modules, with natural transformations of functors as morphisms.
\end{definition}
If $V$ is an FI-module and $S$ is a finite set, we write $V_S$ for $V(S)$, and write $V_n$ for $V(\n)$. Given a morphism $f\colon S\to T$ or $f\colon\m\to\n$, we often write $f_*\colon V_S\to V_T$ or $f_*\colon V_m\to V_n$ for the map $V(f)$.
For any FI-module $V$, the endomorphisms $\End_{\FI}(\n)\simeq S_n$ act on $V_n$, so we can  consider $V_n$ as an $S_n$-representation. Given a map $V\to W$ of FI-modules, the induced map $V_n\to W_n$ is $S_n$-equivariant.

\begin{remark}
$\FIMod$ inherits the structure of an abelian category from $\kMod$, with all notions such as kernel, cokernel, subobject, quotient object, 
injection, and surjection being defined ``pointwise'' from the corresponding notions in $\kMod$.
This is true for any category of functors from any small category to an abelian category \cite[A.3.3]{We}. In particular, the same holds for $\FB$-modules, $\FIsharp$-modules, etc.; see Remark~\ref{remark:FInotation} below.
 
For example, a map $F\colon V\to W$ of FI-modules is a surjection if and only if the maps $F_S\colon V_S\to W_S$ are surjections for all finite sets $S$. Since every finite set is isomorphic to some $\n$, an equivalent condition is that $V_n\to W_n$ is surjective for all $n\in \N$. We will frequently use this equivalence to verify some property of an FI-module $V$ by verifying it for $V_n$ for all $n\in \N$. For another example, the kernel of this map $F$ is an FI-module $\ker[F]$, which to $S$ assigns the $k$-module $\ker[F]_S\coloneq \ker(F_S\colon V_S\to W_S)$. To a morphism $f\colon S\into T$, it assigns the restriction of $V(f)\colon V_S\to V_T$ to $\ker(F_S)$ (which must have image in $\ker(F_T)$ if $F$ is a map of FI-modules).
\end{remark}

\begin{remark}
\label{rem:nonsemisimple}
In some sense the category of FI-modules might be thought of as a ``limit'' of the category of representations of $S_n$ as $n \ra \infty$.  But care is necessary.  For example, the category $\FIMod$ is not semisimple, even when $k$ is a field of characteristic~$0$.

Consider the FI-module $V$ with $V_S=k$ for all finite sets $S$ and $V_f=\id$ for all injections $f$. Let $W$ be the FI-module with $W_\emptyset=k$ and $W_S=0$ if $S\neq \emptyset$, with $W_f=0$ except for $f\colon \emptyset\to \emptyset$. There is an obvious surjection of FI-modules $F\colon V\onto W$ defined by $F_\emptyset=\id\colon k\to k$ and $F_S=0\colon k\to 0$ for $S\neq \emptyset$. However, this surjection is not split: for $f\colon \mathbf{0}\into \mathbf{1}$ the map $f_*\colon W_0\to W_1$ is zero while $f_*\colon V_0\to V_1$ is an isomorphism, so no nonzero map $s\colon W\to V$ can exist.
\end{remark}

The category of FI-modules is closed under any covariant functorial construction on $k$-modules, by applying functors pointwise. For example, if $V$ and $W$ are FI-modules, then $V\oplus W$ and  $V\otimes W$ are FI-modules. Concretely, the FI-module $V\otimes W$ assigns to a finite set $S$ the $k$-module $V_S\otimes W_S$, and to an inclusion $f\colon S\into T$ the homomorphism $V(f)\otimes W(f)\colon V_S\otimes W_S\to V_T\otimes W_T$. Symmetric products and exterior products of FI-modules are also FI-modules, once we fix a functorial definition of these constructions. The \emph{dual} $V^*$ of an FI-module $V$, on the other hand, is a co-FI-module; to a finite set $S$ the co-FI-module assigns the $k$-module $V(S)^*$, and to an inclusion $f\colon S\into T$ the homomorphism $V(f)^*\colon V(T)^*\to V(S)^*$.

\begin{remark}
\label{remark:FInotation}
We use the same notation for other categories: for example, an \emph{FB-module} is a functor from $\FB$ to the category of $k$-modules; an \emph{FI-group} is a functor from $\FI$ to the category of groups; an \emph{$\FIsharp$-space} is a functor from $\FIsharp$ to the category of topological spaces, and so on. In all these cases, we take natural transformations as our morphisms, so $\FBMod\coloneq [\FB,\kMod]$ is the category of FB-modules, $\FIGroups\coloneq [\FI,\Groups]$ is the category of FI-groups, and so on.
\end{remark}

\begin{remark}
By a \emph{graded $k$-module} we always mean an $\N$-graded $k$-module. In keeping with the approach of this paper, we can think of a graded $k$-module as a functor $A\colon \N\to \kMod$, where $\N$ is the discrete category with objects $\N$, so that $A_i=A(i)$ is the part of $A$ in grading $i$. A morphism $A\to B$ of graded $k$-modules is a natural transformation, i.e.\ a collection of homomorphisms $A_i\to B_i$ (so we only consider morphisms of ``degree 0''). To avoid phrases such as ``FI-(graded $k$-module)'', we will say instead ``graded FI-module'', and so on. That is, a graded FI-module $W$ is a functor $W\colon \FI\to \grkMod$, and $\grFIMod=[\FI,\grkMod]$ is the category of graded FI-modules. If $W$ is a graded FI-module, restricting to a fixed $i\in \N$ yields an FI-module $W^i$.
\end{remark}

\begin{remark}
\label{remark:FIalgebras}
A \emph{graded $k$-algebra} is a graded $k$-module $A$ endowed with a unital multiplication $A\otimes A\to A$. We require that this multiplication be unital, but it need not be associative. An ideal $I\subset A$ is a graded submodule $I\subset A$ such that under multipliciation $A\otimes I$ and $I\otimes A$ map into $I$; in this case, the quotient $A/I$ is a graded $k$-algebra. In particular, a \emph{graded FI-algebra} $V$ is a functor from $\FI$ to the category of graded $k$-algebras, and an \emph{FI-ideal} $I\subset V$ yields for each finite set $S$ an ideal $I_S\subset V_S$.
\end{remark}

\subsection{Free FI-modules}
\label{sec:free} In this section we define certain families of FI-modules that  can be thought of as the FI-modules ``freely generated'' by an $S_a$-representation.

\begin{definition}[{\bf FB-modules}] 
Let $\FB$ denote the category of \textbf{f}inite sets and \textbf{b}ijections. An FB-module $W$ is a functor $W\colon \FB\to \kMod$.
\end{definition}
The category $\FBMod=[\FB,\kMod]$ of FB-modules is also known as the category of \emph{vector species} (see e.g.\ \cite{Aguiar}). 
Since $\End_{\FB}(\n)\simeq S_n$, an FB-module $W$ determines a collection of $S_n$-representations $W_n$ for all $n\in \N$, just as an FI-module did. But in contrast with FI-modules, an FB-module $W$ is \emph{determined} by the $S_n$-representations $W_n$, with no additional data such as maps between them. Indeed, we can consider any $S_a$-representation $W_a$ as an FB-module by setting $W_n=0$ for $n\neq a$, and extending to all finite sets $T$ by choosing isomorphisms $T\simeq \n$. With this convention we have an isomorphism of FB-modules $W\simeq \bigoplus_{n=0}^\infty W_n$.

The obvious inclusion $\FB\into \FI$ induces a forgetful functor $\pi\colon \FIMod\to \FBMod$, which remembers the $S_n$-actions on $V_n$ but forgets the maps $V_m\to V_n$ for $m<n$.

\begin{definition}[{\bf The FI-module {\boldmath$M(W)$}}]
\label{def:MW} We define the functor $M(-)\colon \FBMod \to \FIMod$ as  the left adjoint of $\pi\colon \FIMod\to \FBMod$. Explicitly, if $W$ is an FB-module, then by \cite[(Eq.\ X.3.10)]{MacL} the FI-module $M(W)$ satisfies
\begin{equation}
\label{eq:MWdecomp}
M(W)_S=\colim_{\substack{(T\in \FB,\\f\colon T\into S)}} W_T=\bigoplus_{T\subseteq S} W_T,
\end{equation} with the map $f_*\colon M(W)_S\to M(W)_{S'}$ induced by $f\colon S\into S'$ being the sum of $(f|_T)_*\colon W_T\to W_{f(T)}$.
\end{definition}
The unit of this adjunction is the inclusion of FB-modules $W\into \pi(M(W))$ which sends $W_S\into \bigoplus_{T\subseteq S}W_T=M(W)_S$. The co-unit provides a surjection of FI-modules $\bigoplus_{n\geq 0} M(V_n)\onto V$.

From \eqref{eq:MWdecomp} we see that $M(-)$ is an exact functor. Also, as an $S_n$-representation, we can identify
\begin{equation}
\label{eq:MWinduced}
M(W)_n\simeq \bigoplus_{a\leq n} \Ind_{S_a\times S_{n-a}}^{S_n}W_a\boxtimes k.
\end{equation}
In particular, when $k$ is a field, the dimension of the vector space $M(W)_n$ is:
\[\dim M(W)_n=\sum_{a\geq 0}\dim(W_a)\cdot\binom{n}{a}\]
This shows that for any $W$, the dimension of $M(W)_n$ is given for all $n\geq 0$ by a single polynomial in $n$; we will see in Theorem~\ref{thm:polyFIsharp} that a similar statement 
holds for the character of $M(W)_n$.

Consider the regular $S_m$-representation $k[S_m]$ as an FB-module.  The resulting functor assigns to a finite set $T$ the $k$-module with basis indexed by the bijections $\m\to T$. Applying the functor $M(-)$, we obtain the following important family of FI-modules.

\begin{definition}[\textbf{The FI-module \boldmath$M(m)$}]
\label{def:Mm}
For any $m\geq 0$, we denote by $M(m)$ the ``representable'' FI-module $k[\Hom_{\FI}(\m,-)]\simeq M(k[S_m])$. For any finite set $S$, the $k$-module $M(m)_S$ has basis indexed by the injections $\m\into S$; in other words, by ordered sequences of $m$ distinct elements of $S$. A map of FI-modules $F\colon M(m)\to V$ is determined by the choice of an element $v=F(\id_\m)\in V_m$, via the adjunction defining $M(-)$.
\end{definition}

\begin{xample}
Since there is precisely one injection $\emptyset =\mathbf{0}\into S$ for every finite set $S$, the FI-module $M(0)$ is the constant functor with value $k$; as an $S_n$-representation $M(0)_n\simeq k$ is the $1$-dimensional trivial representation. Since injections $\mathbf{1}\into S$ correpond simply to elements of $S$, the FI-module $M(1)$ assigns to a finite set $S$ the free $k$-module with basis $S$; as an $S_n$-representation $M(1)_n\simeq k^n$ is the usual permutation representation.
\end{xample}

The permutation group $S_m\simeq \End_{\FI}(\m)$ acts on $M(m)$ by FI-module automorphisms (as is true of any representable functor). This action is quite concrete: it is simply the standard action of $S_m$ on ordered sequences $(s_1,\ldots,s_m)$ by reordering the elements of a sequence.

\para{Irreducible representations of $S_n$ in characteristic~0} When $k$ is a field of characteristic $0$, every $S_n$-representation decomposes as a direct sum of irreducible representations. In characteristic 0 every irreducible representation of $S_n$ is
defined over $\Q$. As a result this decomposition does not depend on the field $k$; moreover, every representation of $S_n$ is self-dual.

Recall that the irreducible representations of $S_n$ over any field $k$ of characteristic $0$ are
classified by the partitions $\lambda$ of $n$. A \emph{partition} of
$n$ is a sequence $\lambda=(\lambda_1\geq \cdots \geq \lambda_\ell>0)$
with $\lambda_1+\cdots+\lambda_\ell=n$; we write $\abs{\lambda}=n$ or
$\lambda\vdash n$.  We denote by $V_\lambda$ the irreducible $S_n$-representation corresponding to the partition $\lambda\vdash n$.  The representation $V_\lambda$ can be obtained as the image $k[S_n] \cdot
c_\lambda$ of a certain idempotent $c_\lambda$ (the ``Young symmetrizer'') in the group algebra $k[S_n]$.  

\begin{definition}[{\bf The irreducible representation $V(\lambda)_n$}]
\label{def:Vlambdan}
Given a partition $\lambda$, for any $n\geq
\abs{\lambda}+\lambda_1$ we define the \emph{padded partition}
\[\lambda[n]\coloneq (n-\abs{\lambda},\lambda_1,\ldots,\lambda_\ell).\] 
For $n\geq \abs{\lambda}+\lambda_1$, we define $V(\lambda)_n$ to be the irreducible $S_n$--representation
\[V(\lambda)_n\coloneq V_{\lambda[n]}.\]
\end{definition} Since every partition of $n$ can be written as $\lambda[n]$ for a unique partition $\lambda$,  every irreducible $S_n$-representation is isomorphic to $V(\lambda)_n$ for a unique partition $\lambda$.  We sometimes replace $\lambda$ by its corresponding Young diagram.

In this notation the trivial representation of $S_n$ is $V(0)_n$, and the standard $(n-1)$-dimensional irreducible representation is $V(1)_n=V(\y{1})_n$ for all $n \geq 2$.  Many of the usual linear algebra operations behave well with respect to this notation. For example, the identities \[\bwedge^3 V(\y{1})_n\simeq V\Big(\x\y{1,1,1}\x\Big)_n\qquad \text{and}\qquad V(\y{1})_n\otimes V(\y{1})_n \simeq V\big(\x\y{1,1}\x\big)_n\oplus V(\y{2})_n\oplus V(\y{1})_n\oplus V(0)_n\] hold whenever both sides are defined, namely whenever $n \geq 4$. For a similar decomposition of any tensor product in general, see Theorem~\ref{thm:Murnaghan} in \S\ref{section:SchurMurnaghan} below.

\begin{definition}[{\bf The FI-module {\boldmath$M(\lambda)$}}]
When $k$ is a field of characteristic~0, given a partition $\lambda$ we write $M(\lambda)$ for the FI-module $M(\lambda)\coloneq M(V_\lambda)$.
\end{definition}
Many natural combinatorial constructions correspond to FI-modules of this form. For example, since $\y{3}$ is the trivial representation of $S_3$, a basis for $M(\y{3})_S$ is given by the 3-element subsets of $S$; in other words, the FI-module $M(\y{3})$ is the linearization of the functor $S\mapsto \binom{S}{3}$. Similarly, the FI-module $M\big(\x\y{1,1}\x\big)$ is the linearization of the functor sending $S$ to the collection of oriented edges $x\to y$ between distinct elements of $S$ (oriented in the sense that $y\to x$ is the negative of $x\to y$).

\remove{
\begin{alphremark}[{\bf Projective FI-modules}]
\label{rem:projective}
A \emph{projective FI-module} is a projective object in the abelian category $\FIMod$.  
It follows from general considerations (see \cite[Exercise 2.3.8]{We}) that the category $\FIMod$ has enough projectives. Indeed if  $P_n$ is a projective $k[S_n]$-module, then $M(P_n)$  is projective, since $M(-)$ is the left adjoint to an exact functor and thus  preserves projectives \cite[Proposition 2.3.10]{We}.
In fact, it is not hard to show that every projective FI-module is a direct sum $\bigoplus_{n\geq 0} M(P_n)$ where $P_n$ is a projective $k[S_n]$-module; see e.g.\ \cite[Corollary 9.40]{Lue}.
\end{alphremark}
}

\subsection{Generators for an FI-module}
\label{ss:generators}

\begin{definition}[{\bf Span}]
If $V$ is an FI-module and $\Sigma$ is a subset of the disjoint union $\coprod V_n$, the \emph{span} $\spn_V(\Sigma)$ is the minimal sub-FI-module of $V$ containing each element of $\Sigma$.  We say that $\spn_V(\Sigma)$ is the \emph{sub-FI-module of $V$ generated by $\Sigma$}. We sometimes write $\spn(\Sigma)$ when there is no ambiguity.  We write $\spn(V_{\leq m})$ for $\spn\big(\coprod_{k\leq m} V_k\big)$ and $\spn(V_{< m})$ for $\spn\big(\coprod_{k< m} V_k\big)$.
\end{definition}

\begin{lemma}
\label{lemma:Mspan}
Given an FI-module $V$, a choice of element $v\in V_m$ determines a map $M(m)\to V$. The image of this map is $\spn_V(v)$. More generally, if $\Sigma$ is the disjoint union of $\Sigma_n\subset V_n$, the image of the natural map $\bigoplus_{n\geq 0} M(n)^{\oplus \Sigma_n}\to V$ is $\spn_V(\Sigma)$. 
\end{lemma}
\begin{proof}
Let $W$ be the image of the map $M(m)\to V$ sending  $\id_\m\in M(m)_m$ to $v\in V_m$. Certainly $v\in W_m$, so $\spn_V(v)\subset W$. A basis for $M(m)_S$ consists of the morphisms $f\colon \m\into S$, and $f\in M(m)_S$ is sent to $f_*(v)\in V_S$ (as it must be for $M(m)\to V$ to be a map of FI-modules). Thus $W_S$  is the submodule of $V_S$ spanned by the elements $f_*(v)$, where $f$ ranges over injections $f\colon \m\into S$. Since any sub-FI-module of $V$ containing $v$ must contain all these elements $f_*(v)$, we have $W\subset \spn_V(v)$ as well. This completes the proof for $\spn_V(v)$; the argument for general $\Sigma$ is identical.
\end{proof}

Taking $\Sigma_n=V_n$ in Lemma~\ref{lemma:Mspan} gives a surjection $\bigoplus_{n\geq 0}M(n)^{\oplus V_n}\onto V$, which is a slightly less efficient version of the surjection $\bigoplus_{n\geq 0}M(V_n)\onto V$ from Definition~\ref{def:MW}; this amounts to including every element of $V$ in our generating set.\pagebreak

\begin{definition}[{\bf Generation in degree {\boldmath $\leq m$}}]
We say that an FI-module $V$ is \emph{generated in degree $\leq m$} if $V$ is generated by elements of $V_k$ for $k\leq m$, i.e.\ if $\spn(V_{\leq m})=V$.
\end{definition}

\begin{definition}[{\bf Finite generation}]
\label{def:finitegen}
We say that an FI-module $V$ is \emph{finitely generated} if there is a finite set of elements $v_1,\ldots,v_k$ with $v_i\in V_{m_i}$ which \emph{generates} $V$, meaning that $\spn(v_1,\ldots,v_k)=V$.
\end{definition}

We will say that $V$ is \emph{finitely generated in degree $\leq m$} if $V$ is generated in degree $\leq m$ and $V$ is finitely generated; in other words, there exists a finite generating set $v_1,\ldots,v_k$ with $v_i\in V_{m_i}$ for which $m_i\leq m$ for all $i$. Lemma~\ref{lemma:Mspan} implies the following characterization of finitely generated FI-modules in terms of the free FI-modules $M(m)$. 
\begin{proposition}[{\bf Finite generation in terms of \boldmath$M(m)$}]
\label{prop:charfg}
An FI-module $V$ is finitely generated if and only if it admits a surjection $\bigoplus_i M(m_i)\onto V$ for some finite sequence of integers $\{m_i\}$. \end{proposition}
One consequence of Proposition~\ref{prop:charfg} is that if $V$ is a finitely generated FI-module  then $V_S$ is a finitely generated $k$-module for any finite set $S$. Proposition~\ref{prop:charfg} will allow us to reduce questions of finite generation to the corresponding question for the particular FI-modules $M(m)$, as in the proof of the following proposition.

\begin{prop}[{\bf Tensor products of f.g.\ FI-modules}]  
\label{pr:tensor}
If $V$ and $W$ are finitely generated FI-modules, so is $V \tensor W$. If $V$ is generated in degree $\leq m_1$ and $W$ is generated in degree $\leq m_2$, then $V \tensor W$ is generated in degree $\leq m_1+m_2$.  
\end{prop}
 
\begin{proof}
By Proposition~\ref{prop:charfg}, it suffices to show that $U\coloneq M(m_1) \tensor M(m_2)$ is finitely generated in degree $\leq m_1 + m_2$. Each $k$-module $U_n$ is certainly finitely generated (of rank $\binom{n}{m_1}\cdot\binom{n}{m_2}$), so it suffices to prove that $U=\spn(U_{\leq m_1+m_2})$. A basis for $U_S$ is given by pairs $(f_1\colon \m_1\into S, f_2\colon \m_2\into S)$. The basis element $(f_1,f_2)$ lies in the image of $U_T$, where $T=\im(f_1)\cup \im(f_2)$. Since $\abs{T}$ is at most $m_1+m_2$, this shows that $\spn(U_{\leq m_1+m_2})$ is all of $U$, as desired.
\end{proof}

\para{\boldmath$H_0(V)$ and generators for \boldmath$V$}
We conclude this section with another perspective on generators for an FI-module $V$ that will be used in later sections. Given an FB-module $W$, we can consider $W$ as an FI-module by declaring $f_*\colon W_S\to W_T$ to be 0 whenever $f\colon S\into T$ is not bijective. This defines an inclusion of categories $\FBMod\into \FIMod$. 
\begin{definition}[{\bf The functor ${H_0}$}]
\label{def:H0}
We define the functor $H_0\colon \FIMod\to \FBMod$ as the left adjoint of this inclusion $\FBMod\into \FIMod$. Explicitly, given an FI-module $V$, the FB-module $H_0(V)$ satisfies 
\begin{equation}
\label{eq:defH0}
H_0(V)_S\coloneq V_S/\spn(V_{<\abs{S}})_S;
\end{equation} given a bijection  $f\colon S\to S'$, the map $f_*\colon H_0(V)_S\to H_0(V)_{S'}$ is that induced by $f_*\colon V_S\to V_{S'}$.
\end{definition}
\begin{remark}
\label{remark:H0}
An FI-module $V$ is generated in degree $\leq m$ if and only if $H_0(V)_n$ vanishes for $n>m$. Similarly, $V$ is finitely generated if and only if the  $k$-module $\bigoplus_{n\geq 0} H_0(V)_n$ is finitely generated.

If $W_a$ is a $k[S_a]$-representation, the FI-module $M(W_a)$ is generated by $M(W_a)_a\simeq W_a$, so $H_0(M(W_a))_n$ vanishes for all $n\neq a$, and we also have $H_0(M(W_a))_a\simeq W_a$. More generally, for any FB-module $W$ we have a natural isomorphism $H_0(M(W))\simeq W$.
\end{remark}

\begin{remark}
The functor $H_0\colon \FIMod\to \FBMod$ is right-exact but not exact.
In the paper \cite{CE} we investigate the \emph{higher FI-homology}  $H_i(V)$ of an FI-module $V$, namely the left-derived functors of $H_0$. In particular, we show that the vanishing of higher FI-homology is connected with the existence of inductive presentations for $V$, and apply this to homological stability for FI-groups.
\end{remark}

\section{FI-modules: representation stability and character polynomials}

In this section we introduce various quantitative measures of how quickly an FI-module $V$ stabilizes.  We use these to give, in characteristic $0$, an equivalence between finite generation for the FI-module~$V$ and  representation stability (in the sense of Church--Farb)  for  the sequence of $S_n$-representations $V_n$. 

We give a number of strong constraints on $S_n$-representations arising in FI-modules, and prove that the characters $\chi_{V_n}$ are eventually given by a character polynomial. Finally, we prove stability for Schur functors, and we demonstrate how a classical theorem of Murnaghan follows easily from the general theory of FI-modules.

\subsection{Stability degree}
\label{section:stabilitydegree}
In this subsection we introduce the notion of a \emph{stability degree} for an FI-module over an arbitrary ring $k$. In characteristic 0, the stability degree provides a counterpart to the ``stable range'' for representation stability in \cite{CF}, as we prove later in Proposition~\ref{pr:stabdegstabrange}.

\para{Coinvariants of FI-modules} 
If $V_n$ is a representation of $S_n$, recall that the coinvariant quotient $(V_n)_{S_n}$ is the $k$-module $V_n\otimes_{k[S_n]}k$; this is the largest $S_n$-equivariant quotient of $V_n$ on which $S_n$ acts trivially. The functor taking an $S_n$-representation $V_n$ to the coinvariants $(V_n)_{S_n}$ is automatically right exact, and is also left exact when $\abs{S_n}=n!$ is invertible in $k$ (by averaging over $S_n$). 

Given an FI-module $V$, we can apply this for all $n\geq 0$ simultaneously, yielding $k$-modules $(V_n)_{S_n}$ for each $n\geq 0$. Moreover, all the various maps $f_*\colon V_n\to V_m$ induce just a single map $(V_n)_{S_n}\to (V_m)_{S_m}$ for each $n\leq m$. This data basically amounts to a graded $k[T]$-module, which we  define below as $\Phi_0(V)$.

\begin{definition}
A \emph{graded $k[T]$-module} $U$ consists of a collection of $k$-modules $U_i$ for each $i\in \N$, endowed with a map $T\colon U_i\to U_{i+1}$ for each $i\in \N$ (i.e.\ $T$ acts by a map of grading 1).
\end{definition}

For each $a\geq 0$, fix once and for all some set $\na$ with cardinality $a$. The specific set $\na$ is irrelevant; the authors suggest the choice $\na\coloneq \{-1,\ldots,-a\}$, to minimize mental collisions with the finite sets the reader is likely to have in mind. Fix also a functorial definition of the disjoint union $S\disjoint T$ of sets.

\begin{definition}[{\bf The graded $k[T]$-module $\Phi_a(V)$}]
\label{def:Phia}
Given $a\geq 0$ and an FI-module $V$, we define the graded $k[T]$-module $\Phi_a(V)$ as follows.
For $n\in \N$, let $\Phi_a(V)_n$ be the coinvariant quotient
\[\Phi_a (V)_n\coloneq (V_{\na\disjoint \n})_{S_n}.\]
To define $T\colon \Phi_a(V)_n\to \Phi_a(V)_{n+1}$, choose any injection $f\colon \n\into \npone$. This determines an injection $\id\disjoint f\colon \na\disjoint \n\into \na\disjoint \npone$, and thus a map $(\id\disjoint f)_*\colon V_{\na\disjoint \n}\to V_{\na\disjoint \npone}$. The map $T\colon \Phi_a(V)_n\to \Phi_a(V)_{n+1}$ is defined to be the induced map $(V_{\na\disjoint \n})_{S_n}\to (V_{\na\disjoint \npone})_{S_{n+1}}$. All such injections $\id\disjoint f$ are equivalent under post-composition by $S_{n+1}$, so the map $T$ is well-defined, independent of the choice of $f$.
\end{definition}
A morphism $V\to W$ of FI-modules induces a homomorphism $\Phi_a(V)\to \Phi_a(W)$ of graded $k[T]$-modules, and $\Phi_a$ is a functor from FI-modules to graded $k[T]$-modules. The functor $\Phi_a$ is always right exact. When $k$ contains $\Q$ the functor $\Phi_a$ is  exact, for the same reason that $W\mapsto (W)_{S_m}$ is exact.

\begin{definition}[{\bf Stability degree}]
\label{def:stabilitydegree}
The \emph{stability degree} $\stabD(V)$ of an FI-module $V$ is the smallest $s\geq 0$ such that for all $a\geq 0$, the map $T\colon \Phi_a(V)_n\to \Phi_a(V)_{n+1}$ is an isomorphism for all $n \geq s$. 
(If no such $s$ exists, we take $\stabD(V)=\infty$.)
\end{definition}
We can rewrite this definition to fit our usual indexing: if $f_n\colon \n\into \npone$ is the inclusion, for all $a\geq 0$ the maps $(f_n)_*\colon (V_n)_{S_{n-a}}\to (V_{n+1})_{S_{n+1-a}}$ are isomorphisms for all $n\geq \stabD(V)+a$.
When FI-modules are thought of as modules for the exponential twisted commutative algebra as in \cite{Aguiar}, the functor $\Phi_0$ is the \emph{bosonic Fock functor} that plays a crucial role in that theory.

\begin{remark}
The permutations of $\na$ act on the functor $\Phi_a$, so $\Phi_a(V)_n$ is an $S_a$-representation; explicitly, $\Phi_a(V)_n\simeq (\Res^{S_{a+n}}_{S_a\times S_n}V_{a+n})_{S_n}$ as an $S_a$-representation. The map $T\colon \Phi_a(V)_n\to \Phi_a(V)_{n+1}$ is $S_a$-equivariant, so it is equivalent to ask in Definition~\ref{def:stabilitydegree} that $T$ be an isomorphism of $k$-modules or of $S_a$-representations. We make use of this equivalence in Proposition~\ref{pr:stabdegstabrange} and Corollary~\ref{co:vakilwood}~below.
\end{remark}

It turns out to be useful to refine the notion of stability degree slightly.  This will be especially important in Section~\ref{section:configspaces} when we study the behavior of FI-modules in spectral sequences.

\begin{definition}
\label{def:injectivitydegree}
We say that the \emph{injectivity degree} $\injD(V)$ (resp.\ \emph{surjectivity degree} $\surjD(V)$) of an FI-module $V$ is the smallest $s\geq 0$ such that for all $a\geq 0$, the map $\Phi_a(V)_n\to \Phi_a(V)_{n+1}$ induced by multiplication by $T$ is injective (resp.\ surjective) for all $n\geq s$. 
By definition, $\stabD(V)=\max(\injD(V),\surjD(V))$.
\end{definition}

\begin{lemma}
Let $V$ be an FI-module. Any quotient $V\onto W$ satisfies $\surjD(W)\leq \surjD(V)$. When $k$ contains $\Q$, any submodule $W\subset V$ satisfies $\injD(W)\leq \injD(V)$.
\label{lemma:injsurjdescend}
\end{lemma}
\begin{proof}
Given $V\onto W$, the maps $\Phi_a(V)_n\onto \Phi_a(W)_n$ are surjective since $\Phi_a$ is right-exact. Thus if $T\colon \Phi_a(V)_n\to \Phi_a(V)_{n+1}$ is surjective, the same is true of $T\colon \Phi_a(W)_n\to \Phi_a(W)_{n+1}$. The proof for $W\subset V$ is identical, using that $\Phi_a$ is left-exact when $k$ contains $\Q$.
\end{proof}

The following computation will be needed in Section~\ref{section:configspaces}; it also provides an example where the injectivity degree and surjectivity degree differ quite drastically.

\begin{prop}
\label{pr:stabdegprojective}
For any FB-module $W$, the FI-module $M(W)$ has $\injD(M(W))=0$; if $W_i=0$ for $i>m$,  then $\surjD(M(W))\leq m$. The FI-module $M(m)$ has $\injD(M(m))=0$ and $\surjD(M(m))=\stabD(M(m))=m$.
\end{prop}

\begin{proof}
We will directly compute the graded $k[T]$-module $\Phi_a M(W)$.
For any FB-module $W$, by \eqref{eq:MWdecomp} we have \[M(W)_{\na\disjoint\n}=\bigoplus_{T\subset\, \na\disjoint \n} W_T.\]
Since $\Phi_a M(W)_n$ is by definition the $S_n$-coinvariants of $M(W)_{\na\disjoint\n}$, it splits as a direct sum over the orbits of $S_n$ acting on $\{T\subset \na\disjoint \n\}$. Such an orbit is determined by the intersection $T\cap \na$ (which can be any subset $U\subset \na$) together with the cardinality $\abs{T\cap \n}$ (which can be any integer $0\leq k\leq n$).

As a representative for the orbit corresponding to $(U,k)$ we can take $T=U\disjoint \kk$. The stabilizer in $S_n$ of this set consists of permutations of $\n$ preserving the subset $\kk$, which we can identify with $S_k\times S_{n-k}$.
This summand thus contributes $(W_{U\disjoint \kk})_{S_k\times S_{n-k}}$ to the coinvariants. Since the subgroup $S_{n-k}<S_k\times S_{n-k}$ acts trivially on $U\disjoint \kk$ and thus on $W_{U\disjoint \kk}$, we have $(W_{U\disjoint \kk})_{S_k\times S_{n-k}}=(W_{U\disjoint \kk})_{S_k}$. This last can be identified as a $k$-module with $\Phi_{\abs{U}}(W)_k$. (Recall from Definition~\ref{def:H0} that every FB-module can be regarded as an FI-module, so it is meaningful to talk about $\Phi_a(W)$.) In conclusion, we have an isomorphism of $k$-modules
\begin{equation}
\label{eq:PhiaMW}
\Phi_a M(W)_n\simeq \bigoplus_{\substack{U\subset \na\\0\leq k\leq n}}\Phi_{\abs{U}}(W)_k\end{equation}

The map $T\colon \Phi_a M(W)_n\to \Phi_a M(W)_{n+1}$ is induced by $\id\disjoint f\colon \na\disjoint \n\into \na\disjoint \npone$ for any $f\colon \n\into \npone$. Such an injection preserves the subset $T\cap \na$ and the cardinality $\abs{T\cap \n}=\abs{f(T)\cap \npone}$. In other words, $T$ preserves the decomposition \eqref{eq:PhiaMW}, and takes the factor $\Phi_{\abs{U}}(W)_k$ of $\Phi_a M(W)_n$ to the factor $\Phi_{\abs{U}}(W)_k$ of  $\Phi_a M(W)_{n+1}$ via the identity. This demonstrates that $\injD(M(W))$ has injectivity degree 0, since $T\colon \Phi_a M(W)_n\to \Phi_a M(W)_{n+1}$ is always injective.

If $W_i=0$ for $i>m$, certainly $\Phi_{\abs{U}}(W)_i=0$ for $i>m$ (indeed for $i>m-\abs{U}$). When $n\geq m$, the condition $0\leq k\leq n$ in \eqref{eq:PhiaMW} is therefore vacuous, so this description of $\Phi_aM(W)_n$ is independent of $n$ for $n\geq m$. Since on each factor the map $T$ is an isomorphism, we conclude that $T\colon \Phi_a M(W)_n\to \Phi_a M(W)_{n+1}$ is an isomorphism for $n\geq m$, so $M(W)$ has surjectivity degree $\leq m$.

Finally, for $M(m)=M(k[S_m])$, we can compute that $\Phi_0(k[S_m])_m\simeq (k[S_m])_{S_m}\simeq k\neq 0$. This summand contributes to $\Phi_a M(m)_n$ only when $n\geq m$, so $\surjD(M(m))=m$.
\end{proof}

\begin{prop}
\label{pr:surjdeg}
If $V$ is generated in degree $\leq d$ then $\surjD(V)\leq d$.
\end{prop}
\begin{proof}
If $V$ is generated in degree $\leq d$, Lemma~\ref{lemma:Mspan} shows that $V$ is a quotient of $\bigoplus_{m\leq d} M(V_m)$.  The latter has surjectivity degree $\leq d$ by Proposition~\ref{pr:stabdegprojective}, so $\surjD(V)\leq d$ by Lemma~\ref{lemma:injsurjdescend}.
\end{proof}
The converse of Proposition~\ref{pr:surjdeg} is not true. For fixed $d\geq 0$,  any FI-module $V$ with $(V_{d+1})_{S_{d+1}}=0$ and $V_n=0$ for $n>d+1$ will have $\surjD(V)\leq d$. Indeed, these assumptions guarantee that $\Phi_a(V)_{n+1}=0$ for all $n\geq d$ and all $a\geq 0$, so surjectivity of $\Phi_a(V)_n\to \Phi_a(V)_{n+1}$ is automatic. But if $V_{d+1}\neq 0$, this FI-module is not generated in degree $\leq d$.

\subsection{Weight and stability degree in characteristic 0}
\label{ss:weight}
In this section we define and characterize certain structural properties of FI-modules over a field of characteristic~0 that will be used in \S\ref{section:repstab} below.

\begin{definition}[{\bf Weight of an FI-module}]
\label{de:weight}
Let $V$ be an FI-module over a field of characteristic~0. The \emph{weight} $\wt(V)$ of $V$ is the maximum of $\abs{\lambda}$ over all irreducible constituents $V(\lambda)_n$ occurring in the $S_n$-representations $V_n$. (If $V=0$ we set $\wt(V)=0$; if $\abs{\lambda}$ is unbounded then $\wt(V)=\infty$.) 
\end{definition}
If $W$ is a subquotient of $V$, then $\wt(W)\leq \wt(V)$. We also have the following proposition, which follows from a well-known property of Kronecker coefficients: if $V(\nu)_n$ occurs in the tensor product $V(\lambda)_n\otimes V(\mu)_n$, then $\abs{\nu}\leq \abs{\lambda}+\abs{\mu}$. 
\begin{prop}
\label{pr:tensorwt}
If $V$ and $W$ are FI-modules over a field of characteristic~0, the tensor product $V\otimes W$ satisfies $\wt(V\otimes W)\leq \wt(V)+\wt(W)$.
\end{prop}

In our analysis of the notion of weight, we will make frequent use of the classical branching rule for $S_n$-representations (see e.g.\ \cite{FH}). The same result holds if $\Q$ is replaced by any field of characteristic 0.
\begin{lemma}[{\bf Branching rule for $S_n$-representations}]
\label{lemma:branching}
Let  $\lambda\vdash n$ be a partition, and $V_\lambda$ the corresponding  irreducible $S_n$-representation over $\Q$.
\begin{enumerate}
\item[\rm{(i)}] As $S_{n+k}$-representations, we have the decomposition
\[\Ind_{S_n\times S_k}^{S_{n+k}} V_\lambda\boxtimes \Q\simeq \bigoplus_{\mu} V_\mu\]
over those partitions $\mu\vdash n+k$ obtained from $\lambda$ by adding one box to $k$ different columns.
\item[\rm{(ii)}] As $S_{n-k}$-representations, we have the decomposition \[(\Res^{S_n}_{S_{n-k}\times S_k}V_\lambda)_{S_k}\simeq\bigoplus_\nu V_\nu\] over those partitions $\nu\vdash n-k$  obtained from $\lambda$ by removing one box from $k$ different columns.\end{enumerate}
\end{lemma}

\begin{prop}
\label{pr:Mmuweight}
For any partition $\mu\vdash m$, the FI-module $M(\mu)$ over a field of characteristic~0 has $\wt(M(\mu))=m$. 
\end{prop}
\begin{proof}
The identification \eqref{eq:MWinduced} shows that $M(\mu)_n=M(V_\mu)_n$ is isomorphic  as an $S_n$-representation to $\Ind_{S_m\times S_{n-m}}^{S_n} V_\mu\boxtimes k$. Thus by Lemma~\ref{lemma:branching}(i),
those $\nu\vdash n$ for which $V_\nu$ occurs in $M(\mu)_n$ are those obtained from $\mu$ by adding one box to $n-m$ different columns. Since $\nu_1$ is the \emph{number} of columns of $\nu$, we must have $\nu_1\geq n-m$. When writing $\nu=\lambda[n]$ we have $\abs{\lambda}=n-\nu_1$, so we can rephrase this as saying that $\abs{\lambda}\leq m$ for any constituent $V(\lambda)_n$ of $M(\mu)_n$. This shows that $\wt(M(\mu))\leq m$. For any $n\geq m+\mu_1$, adding one box to the first $n-m$ columns yields the partition $\mu[n]$, so $V(\mu)_n$ itself occurs in $M(\mu)_n$; we conclude that $\wt(M(\mu))=m$.
\end{proof}
Proposition~\ref{pr:Mmuweight} shows that $\wt(M(V_m))\leq m$ for any $S_m$-representation $V_m$. 
Since any FI-module $V$ generated in degree $\leq d$ is a quotient of $\bigoplus_{m\leq d} M(V_m)$ by Lemma~\ref{lemma:Mspan}, we obtain the following proposition as a corollary.
\begin{prop}
\label{pr:fgweight}
Let $V$ be an FI-module over a field of characteristic~0. If $V$ is generated in degree $\leq d$, then $\wt(V)\leq d$.
\end{prop}


\para{Stability degree in characteristic 0} The notion of stability degree over a field $k$ of characteristic 0 is richer than in the general case. For example, Proposition~\ref{pr:stabdegprojective} states that the FI-module $M(\lambda)=M(V_\lambda)$ has $\stabD(M(\lambda))\leq\abs{\lambda}$, but in fact this can be improved. Recall that $\lambda_1$ is the length of the first row of the partition $\lambda$.
\begin{prop}
\label{pr:stabdegMlambda}
For any partition $\lambda$ the FI-module $M(\lambda)$ over a field of characteristic~0 has 
$\stabD(M(\lambda))=\lambda_1$.
\end{prop}
To prove Proposition~\ref{pr:stabdegMlambda}, we will need the following elementary consequences of the branching rule.
If $V_n$ is an $S_n$-representation, we will write $(V_n)_{S_k}$ for the $S_{n-k}$-representation $(\Res^{S_n}_{S_{n-k}\times S_k}V_\lambda)_{S_k}$.
\begin{lemma}
\label{lem:coinvariants}
Let $\mu$ be a partition of $n$, and write $\mu=\lambda[n]$ as in Definition~\ref{def:Vlambdan}. Given $a\leq n$, set $k\coloneq n-a$, and consider the $S_a$-representation $(V_\mu)_{S_k}=(V(\lambda)_n)_{S_{n-a}}$ over $\Q$.
\begin{enumerate}
\item[\rm{(i)}] $(V_\mu)_{S_k}=0\iff k>\mu_1$; equivalently, $(V(\lambda)_n)_{S_{n-a}}=0 \iff a<\abs{\lambda}$.
\item[\rm{(ii)}] If $k=\mu_1$ (i.e.\ if $a=\abs{\lambda}$), we have $(V_\mu)_{S_k}=(V(\lambda)_n)_{S_{n-a}}\simeq V_\lambda$.
\item[\rm{(iii)}] For fixed $\lambda$ and $a\geq \abs{\lambda}$, the $S_a$-representation $(V(\lambda)_n)_{S_{n-a}}$ is independent of $n$ once $n\geq a+\lambda_1$; in fact 
\begin{equation}
\label{eq:coinvariants}
(V(\lambda)_n)_{S_{n-a}}\simeq \Ind_{S_{\abs{\lambda}}\times S_{a-\abs{\lambda}}}^{S_a} V_\lambda\boxtimes \Q\qquad\text{ for all }n\geq a+\lambda_1.
\end{equation}
\item[\rm{(iv)}] If $V$ is an FI-module with $\wt(V)\leq a$, any subquotient $W_n$ of the $S_n$-representation $V_n$ satisfies $W_n=0\iff (W_n)_{S_{n-a}}=0$.
\end{enumerate}
\end{lemma}
\begin{proof}
Lemma~\ref{lemma:branching}(ii) states that $(V_\mu)_{S_k}$ is the sum of $V_\nu$ over partitions $\nu\vdash a$ obtained from $\mu$ by removing boxes from $k$ different columns. Since $\mu$ has only $\mu_1$ columns, this is impossible when $k>\mu_1$. This demonstrates (i), from which (iv) follows immediately. When $k=\mu_1$, this can be done only by removing one box from each of the $\mu_1$ columns of $\mu$; by definition, to say that $\mu=\lambda[n]$ means that the resulting partition is $\lambda$, demonstrating (ii). 

To prove (iii), let $r=n-a$ and $c=a-\abs{\lambda}$; our assumption is that $r\geq \lambda_1$.
By Lemma~\ref{lemma:branching}(i), the right side of \eqref{eq:coinvariants} consists of those $V_\nu$ for which $\nu\vdash a$ is obtained from $\lambda$ by adding one box to $c$ different columns. Observe that these columns must lie within the first $\lambda_1+c$ columns, or the result would not be a partition. 

By definition, $\lambda[n]$ is obtained from $\lambda$ by adding one box to each of the first $r+c$ columns. By Lemma~\ref{lemma:branching}(ii) the left side of \eqref{eq:coinvariants} consists of those $V_\nu$ for which $\nu$ is obtained from $\lambda[n]$ by removing boxes from $r$ different columns; in other words, those $\nu$ obtained from $\lambda$ by adding one box to $c$ different columns \emph{within} the first $r+c$ columns. When $r\geq \lambda_1$, the observation above shows this last condition is vacuous, and the collection of $V_\nu$ occurring in the left and right side of \eqref{eq:coinvariants} coincide.
\end{proof}

\begin{proof}[Proof of Proposition~\ref{pr:stabdegMlambda}]
From our earlier computation of $\Phi_a M(V_\lambda)$ in \eqref{eq:PhiaMW} we have: 
\begin{equation}
\label{eq:PhiaMVlambda}
\Phi_a M(V_\lambda)_n\simeq\bigoplus_{\substack{k\leq n\\U\subset \na}}\Phi_{\abs{U}}(V_\lambda)_k=\bigoplus_{\substack{k\leq n\\\abs{U}+k=\abs{\lambda}}}(V_\lambda)_{S_k}
\end{equation}
We saw in the proof of Proposition~\ref{pr:stabdegprojective} that $T\colon \Phi_a M(V_\lambda)_n\to \Phi_a M(V_\lambda)_{n+1}$ is the identity on each factor of this decomposition, so our goal is to show that no new factors occur in $\Phi_a M(V_\lambda)_n$ for $n>\lambda_1$.

By Lemma~\ref{lem:coinvariants}(i), $(V_\lambda)_{S_k} = 0$ when $k>\lambda_1$. Thus we can add the condition $k\leq \lambda_1$ to \eqref{eq:PhiaMVlambda}; the condition $k\leq n$ is then vacuous for $n\geq \lambda_1$, showing that no new factors occur after $n=\lambda_1$. By Lemma~\ref{lem:coinvariants}(ii), $(V_\lambda)_{S_k}\neq 0$ when $k=\lambda_1$. Therefore the factors  with $k=\lambda_1$, which occur only for $n\geq \lambda_1$, are nonzero. We conclude that $T\colon \Phi_a M(V_\lambda)_n\to \Phi_a M(V_\lambda)_{n+1}$ is an isomorphism for $n\geq \lambda_1$ but is not surjective for $n=\lambda_1-1$, so $M(V_\lambda)$ has stability degree $\lambda_1$. 
\end{proof}

The stability degree bounds the width of the irreducible constituents of the $S_n$-representations $V_n$.
\begin{prop}
\label{pr:narrow}
Let $V$ be an FI-module over a field of characteristic~0. For all $n\geq 0$, every irreducible constituent $V(\lambda)_n$ of the $S_n$-representation $V_n$ satisfies $\lambda_1\leq \stabD(V)$.
\end{prop}
\begin{proof}
Let $s\coloneq \stabD(V)$. Consider the filtration $F^iV\coloneq \spn(V_{\leq i})$, which satisfies $V=\bigcup F^i V$. We will prove the inequality $\lambda_1\leq s$ for the irreducible constituents of each quotient $F^mV/F^{m-1}V$, and this implies the same inequality for $V$.

We first prove that for all $m\geq 0$, every irreducible constituent $V_\mu$ of $H_0(V)_m$ satisfies $\mu_1\leq s$. 
Fix $m\geq 0$, and let $\overline{V}\coloneq V/F^{m-1}V$. 
We have $\overline{V}_m=V_m/\spn(V_{<m})_m=H_0(V)_m$ as $S_m$-representations, and of course $\overline{V}_n=0$ for $n<m$. Being a quotient of $V$, we have $\surjD(\overline{V})\leq s$ by Lemma~\ref{lemma:injsurjdescend}, so the map $T\colon \Phi_{m-s-1}(\overline{V})_s\to  \Phi_{m-s-1}(\overline{V})_{s+1}$ is surjective. The domain $\Phi_{m-s-1}(\overline{V})_s$ is a quotient of $\overline{V}_{m-1}=0$ and thus vanishes, so $\Phi_{m-s-1}(\overline{V})_{s+1}=0$ as well. The latter is isomorphic to $(\overline{V}_m)_{S_{s+1}}$, so Lemma~\ref{lem:coinvariants}(i) states that $s+1>\mu_1$ for every irreducible constituent $V_\mu$ of $\overline{V}_m\simeq H_0(V)_m$.

We next show that every irreducible constituent $V(\lambda)_n$ of $M(V_\mu)_n$ satisfies $\lambda_1\leq \mu_1$. As we saw in the proof of Proposition~\ref{pr:Mmuweight}, $M(V_\mu)_n$ consists of those $V_\nu$ for which $\nu\vdash n$ is obtained from $\mu$ by adding boxes to $n-m$ different columns. The length of the \emph{second} row of $\nu$ is thus bounded by the length of the first row of $\mu$. When we write $V_\nu=V(\lambda)_n$ we have $\lambda_1=\nu_2$, which verifies the claim.

Consider the quotient $W\coloneq F^mV/F^{m-1}V$; to complete the proof we must show that every irreducible constituent $V(\lambda)_n$ of $W_n$ satisfies $\lambda_1\leq s$. By definition $F_mV$ is generated in degree $\leq m$, so the same is true of $W$. By Lemma~\ref{lemma:Mspan} there exists a surjection $\bigoplus_{n\leq m} M(W_n)\onto W$. As before we have $W_n=0$ for $n<m$ and  $W_m\simeq \overline{V}_m\simeq H_0(V)_m$, so this simplifies to $M(H_0(V)_m)\onto W$. As an $S_m$-representation $H_0(V)_m$ is a sum of irreducibles $V_\mu$, which according to the second paragraph all satisfy $\mu_1\leq s$, so the third paragraph implies that every irreducible constituent $V(\lambda)_n$ of $M(H_0(V)_m)_n$ satisfies $\lambda_1\leq s$. As an $S_n$-representation $W_n$ is a quotient of $M(H_0(V)_m)_n$, so the desired property holds for $W_n$ as well. This verifies the claim for $F^mV/F^{m-1}V$, and completes the proof.
\end{proof}
\pagebreak

Finally, we prove the Noetherian property for FI-modules over rings containing $\Q$. 

\begin{proof}[Proof of Theorem~\ref{th:noetherian}]
Let $k$ be a Noetherian ring, let $V$ be a finitely generated FI-module over $k$ generated in degree $\leq a$, and let $W$ be a sub-FI-module of $V$; our goal is to prove that $W$ is finitely generated. Proposition~\ref{prop:charfg} implies that $V_n$ is a finitely generated $k$-module for each $n$. Since $k$ is Noetherian, its submodule $W_n$ is finitely generated as well.

The functor $\Phi_a$ of Definition~\ref{def:Phia} is exact over rings containing $\Q$, so $\Phi_a(W)\subset \Phi_a(V)$. The graded $k[T]$-module $\Phi_a(V)$ is finitely generated by Proposition~\ref{pr:surjdeg}. Since $k[T]$ is a Noetherian ring, its submodule $\Phi_a(W)$ is  finitely generated as a $k[T]$-module (equivalently, as a graded $k[T]$-module).

\para{The sub-FI-module $\widetilde{W}$} Choose a finite set of generators $x_1, \ldots, x_r$ of $\Phi_a(W)$ as a graded $k[T]$-module, with  $x_i\in \Phi_a(W)_{n_i}$. Since $\Phi_a(W)_{n}\simeq (W_{a+n})_{S_{n}}$ is a quotient of $W_{a+n}$, we can choose lifts  $w_i\in W_{a+n_i}$ projecting to $x_i$.
Let $\widetilde{W}$ be the finitely generated sub-FI-module of $W$ generated by the lifts $w_1,\ldots,w_r$, together with finite generating sets for $W_0$, \ldots, $W_{a-1}$.

 Since $\widetilde{W}\subseteq W$ we have $\Phi_a(\widetilde{W})\subseteq \Phi_a(W)$.  Since $\Phi_a(\widetilde{W})$ contains the generating set $\{x_i\}$ we have $\Phi_a(\widetilde{W})\supseteq \Phi_a(W)$. Thus $\Phi_a(\widetilde{W})=\Phi_a(W)$.  Since $\Phi_a$ is exact, this implies that $\Phi_a(W/\widetilde{W})=0$.

\para{The difference $W/\widetilde{W}$} For any $n$ we can decompose $(W/\widetilde{W})_n$ as a $k[S_n]$-module into isotypic components $(W/\widetilde{W})_n\simeq \bigoplus_{\lambda}N_\lambda\tensor_\Q V(\lambda)_n$, where $V(\lambda)_n$ is the irreducible $\Q[S_n]$-module and $N_\lambda$ is a $k$-module. We claim that only those $\lambda$ with $\abs{\lambda}\leq a$ appear in this decomposition. This is essentially nothing more than the claim $\wt(W/\widetilde{W})\leq a$, except that weight is only defined for FI-modules over a \emph{field} of characteristic 0.

However, by restricting to $\Q\subset k$, we can consider any FI-module $U$ over $k$ as an FI-module $U^\Q$ over $\Q$.
To say that $V$ is generated in degree $\leq a$ means the natural map $\bigoplus_{m\leq a}M(V_m)\to V$ is surjective. It certainly remains surjective when considered as a map of $\Q$-vector spaces, so since $M(V_m)^\Q=M(V^\Q_m)$ this provides a surjection $\bigoplus_{m\leq a}M(V^\Q_m)\onto V^\Q$. Therefore $V^\Q$ is generated in degree $\leq a$ as well, $\wt(V^\Q)\leq a$ by Proposition~\ref{pr:fgweight} states that $\wt(V^\Q)\leq a$. Since $(W/\widetilde{W})^\Q$ is a subquotient of $V^\Q$ we have $\wt((W/\widetilde{W})^\Q)\leq a$, as claimed.

For any $n\geq a$, we have
\[\Phi_a(W/\widetilde{W})_{n-a}=(W/\widetilde{W})_n\otimes_{k[S_{n-a}]} k\simeq \bigoplus N_\lambda\tensor_\Q \big(V(\lambda)_n\tensor_{\Q[S_{n-a}]} \Q\big).\] 
Lemma~\ref{lem:coinvariants}(i) states that $V(\lambda)_n\tensor_{\Q[S_{n-a}]} \Q=(V(\lambda)_n)_{S_{n-a}}$ is nonzero when $\abs{\lambda}\leq a$. But we proved above that $\Phi_a(W/\widetilde{W})=0$. This is only possible if $N_\lambda=0$ for all $\lambda$, or in other words if $(W/\widetilde{W})_n=0$. 

This shows that $(W/\widetilde{W})_n=0$ for all $n\geq a$, and  $\widetilde{W}_n=W_n$ by definition for  $n<a$, so $W=\widetilde{W}$. This demonstrates that $W$ itself is finitely generated, and thus completes the proof of the theorem.
\end{proof}

\subsection{Representation stability and character polynomials}
\label{section:repstab}

Our main goal in this subsection is to prove Proposition~\ref{pr:stabdegstabrange}, relating the weight and stability degree of FI-modules with uniform representation stability in the sense of Church--Farb~\cite{CF}. We then use this to prove Theorem~\ref{th:FIequiv}, which states the equivalence of finite generation and representation stability. 

We first quickly recall the definition of representation stability; see \cite{CF} for a detailed treatment. 
As mentioned in the introduction, an FI-module $V$ provides a sequence $\{V_n\}$ of $S_n$-representations for all $n\in \N$. Moreover, if $f_{n,n+k}\colon \n\into \mathbf{n+k}$ denotes the standard inclusion, the injections $f_{n,n+1}\colon \n\into \npone$ induce $S_n$-equivariant maps $\phi_n\colon V_n\to V_{n+1}$. Such a sequence  $\{V_n,\phi_n\}$ of $S_n$-representations and $S_n$-equivariant maps was called a \emph{consistent sequence} in \cite{CF}.

\begin{remark}
Not every consistent sequence can arise from an FI-module.
For any $\sigma\in S_{n+k}$ with $\sigma|_\n=\id$, we have the identity $f_{n,n+k}=\sigma \circ f_{n,n+k}$ in $\Hom_{\FI}(\n,\mathbf{n+k})$. Therefore on any consistent sequence arising from an FI-module,
\begin{equation}
\label{eq:consistentseqcondition}
\text{all }\sigma\in S_{n+k}\text{ with }\sigma|_\n=\id\text{ must act trivially on }\im(V_n\to\cdots\to V_{n+k})\subset V_{n+k}.
\end{equation}
For example, the consistent sequence of regular representations $k\to k \to k[S_2] \to k[S_3]\to \cdots$ induced by the standard inclusions $S_m\into S_{m+1}$ does not satisfy \eqref{eq:consistentseqcondition}, and thus cannot arise from an FI-module. Conversely, it is not difficult to check that this condition is also sufficient:  any consistent sequence satisfying \eqref{eq:consistentseqcondition} can be ``promoted'' to an FI-module. (In fact, it suffices that \eqref{eq:consistentseqcondition} holds when $k=2$.)
\end{remark}

\begin{definition}[{\bf Representation stability \cite{CF}}] 
\label{definition:repstab}
Let $\{V_n,\phi_n\}$ be a consistent sequence of $S_n$--representations over a field of characteristic~0.
The sequence $\{V_n,\phi_n\}$ is \emph{uniformly representation stable with stable range $n\geq N$} if each of the following conditions holds.
\begin{enumerate}[{\bf I.}]
\item \textbf{Injectivity:} The map $\phi_n\colon V_n\to
  V_{n+1}$ is injective for all $n\geq N$.
\item \textbf{Surjectivity:} The span of the $S_{n+1}$--orbit of
  $\phi_n(V_n)$ equals all of $V_{n+1}$ for all $n\geq N$.
\item \textbf{Multiplicities:} Decompose $V_n$ into irreducible
  representations as
  \[V_n=\bigoplus_\lambda c_{\lambda,n}V(\lambda)_n\] with
  multiplicities $0\leq c_{\lambda,n}\leq \infty$. For each $\lambda$,
  the multiplicities $c_{\lambda,n}$ are  independent of
  $n$ for $n\geq N$.
\end{enumerate}
\end{definition}
We say that $V$ is uniformly representation stable if Definition~\ref{definition:repstab} holds for some $N$.
The following proposition shows that a stability degree for an FI-module $V$ guarantees that it is uniformly representation stable with a stable range that can be specified quite precisely.

\begin{prop}[{\bf Stability degree and\ representation stability}]
\label{pr:stabdegstabrange}
Let $V$ be an FI-module over a field of characteristic~0. The sequence of $S_n$-representations $\{V_n,\phi_n\}$ is uniformly representation stable with stable range ${n\geq \wt(V)+\stabD(V)}$.
\end{prop}
\begin{proof}
Set $s\coloneq \stabD(V)$ and $d\coloneq \wt(V)$.
Let $K_n$ be the kernel of the $S_n$-equivariant map $\phi_n\colon V_n\to V_{n+1}$, and let
 $C_{n+1}$ be the cokernel of the $S_{n+1}$-equivariant map $\phi'_n\coloneq \Ind_{S_n}^{S_{n+1}}\phi_n \colon \Ind_{S_n}^{S_{n+1}} V_n\to V_{n+1}$ induced by $\phi_n$.
 
To prove Conditions I and II of Definition~\ref{definition:repstab}, we must show that $K_n = 0$ and $C_{n+1}=0$ for all $n \geq s+d$.
Since $K_n$ is a subrepresentation of $V_n$ and $C_{n+1}$ is a quotient of $V_{n+1}$, Lemma~\ref{lem:coinvariants}(iv) implies for any $n\geq d$ that $K_n=0\iff (K_n)_{S_{n-d}}=0$ and $C_{n+1}=0\iff (C_{n+1})_{S_{n+1-d}}=0$. Since taking $S_a$-coinvariants in characteristic 0 is exact, $(K_n)_{S_{n-d}}$ is the kernel of $(\phi_n)_{S_{n-d}}\colon (V_n)_{S_{n-d}}\to (V_{n+1})_{S_{n-d}}$, and $(C_{n+1})_{S_{n+1-d}}$ is the cokernel of $(\phi'_n)_{S_{n+1-d}}\colon (\Ind^{S_{n+1}}_{S_n} V_n)_{S_{n+1-d}} \ra (V_{n+1})_{S_{n+1-d}}$.

The map $T\colon \Phi_d(V)_{n-d}\to \Phi_d(V)_{n+1-d}$ of Definition~\ref{def:stabilitydegree}  can be factored in two ways:
\[T\colon (V_n)_{S_{n-d}} \xrightarrow{(\phi_n)_{S_{n-d}}} (V_{n+1})_{S_{n-d}} \ra (V_{n+1})_{S_{n+1-d}}.\]
\[T\colon (V_n)_{S_{n-d}} \ra (\Ind^{S_{n+1}}_{S_n} V_n)_{S_{n+1-d}} \xrightarrow{(\phi'_n)_{S_{n+1-d}}} (V_{n+1})_{S_{n+1-d}}\]
By definition of stability degree, this composition $T$ is an isomorphism when $n-d \geq s$. It follows that $(\phi_n)_{S_{n-d}}$ is injective and $(\phi'_n)_{S_{n+1-d}}$ is surjective when $n-d\geq s$.  We conclude for $n\geq s+d$ that $(K_n)_{S_{n-d}} = 0$ and $(C_{n+1})_{S_{n+1-d}}=0$; therefore $K_n = 0$ and $C_{n+1}=0$, as desired.

It remains to prove that the multiplicity $c_{\lambda,n}$ of $V(\lambda)_n$ in $V_n$ is constant when $n \geq s + d$. We prove this by induction on $\abs{\lambda}$; no base case will be necessary. If $\abs{\lambda}>d$, by definition of weight $c_{\lambda,n}=0$ for all $n$, so it suffices to prove this for $\abs{\lambda}\leq d$. For fixed $m\leq d$, assume by induction that for all $\mu$ with $\abs{\mu}<m$ the multiplicity $c_{\mu,n}$ of $V(\mu)_n$ in $V_n$ is constant for $n\geq s+d$.

By definition of stability degree, that the isomorphism class of the $S_m$-representation $(V_n)_{S_{n-m}}$ is independent of $n$ for  $n\geq s+m$.
By Proposition~\ref{pr:narrow}, each $V(\lambda)_n$ which occurs in $V_n$ satisfies $\lambda_1\leq s$. The results of Lemma~\ref{lem:coinvariants} apply once $n\geq \lambda_1+m$, so they hold in our range $n\geq s+d$. 

By Lemma~\ref{lem:coinvariants}(i), only those $V(\lambda)_n$ with $\abs{\lambda}\leq m$ contribute to $(V_n)_{S_{n-m}}$, so we can write
\begin{equation}
\label{eq:repstabcoinvariants}
(V_n)_{S_{n-m}}=\bigoplus_{\abs{\mu}<m}c_{\mu,n}(V(\mu)_n)_{S_{n-m}}\oplus \bigoplus_{\abs{\lambda}=m}c_{\lambda,n}(V(\lambda)_n)_{S_{n-m}}
\end{equation}
For each $\mu$ appearing in the first summand, Lemma~\ref{lem:coinvariants}(iii) states that the $S_m$-representation $(V(\mu)_n)_{S_{n-m}}$ is independent of $n$ in our range. By induction the multiplicities $c_{\mu,n}$ are also constant for $n\geq s+d$, so the isomorphism class of the first summand of \eqref{eq:repstabcoinvariants} is constant for $n\geq s+d$. Since the same is true of $(V_n)_{S_{n-m}}$ itself, the isomorphism class of the second summand of \eqref{eq:repstabcoinvariants} must also be constant for $n\geq s+d$. By Lemma~\ref{lem:coinvariants}(ii), the second summand of \eqref{eq:repstabcoinvariants} is simply $\bigoplus_{\abs{\lambda}=m}c_{\lambda,n}V_\lambda$. Therefore for each $\lambda$ with $\abs{\lambda}=m$, the multiplicity $c_{\lambda,n}$ is constant for $n\geq s+d$, as desired.
\end{proof}

We can now complete the proof of Theorem~\ref{th:FIequiv}, which states that representation stability is equivalent to finite generation for FI-modules over a field of characteristic~0. We also observe along the way that finitely generated FI-modules are monotone in the sense of \cite[Definition~1.2]{Ch}.
\begin{proof}[Proof of Theorem~\ref{th:FIequiv}]
Assume that $V$ is finitely generated, say in degree $\leq g$. We have $\surjD(V)\leq g$ by Proposition~\ref{pr:surjdeg}, and we would like to bound $\injD(V)$ as well. By Lemma~\ref{lemma:Mspan} there is a free FI-module $M\coloneq \bigoplus_{i=0}^g M(i)^{\oplus b_i}$ with a surjection $M\onto V$; let $K$ be the kernel of this map. By Theorem~\ref{th:noetherian}, the submodule $K$ of $M$ is finitely generated, say in degree $\leq r$.

Since $\Phi_a$ is exact over rings containing $\Q$, we have $\Phi_a(V)_n\simeq \Phi_a(M)_n / \Phi_a(K)_n$. Proposition~\ref{pr:stabdegprojective} states that $\injD(M)=0$, so the maps $T\colon \Phi_a(M)_n\to \Phi_a(M)_{n+1}$ are injective for all $n\geq 0$.  Proposition~\ref{pr:surjdeg} implies $\surjD(K)\leq r$, so for $n\geq r$ the maps $T\colon \Phi_a(K)_n\to \Phi_a(K)_{n+1}$ are surjective. We conclude that for $n\geq r$ the maps $T\colon \Phi_a(V)_n\to \Phi_a(V)_{n+1}$ are injective, or in other words that $\injD(V)\leq r$. Therefore $\stabD(V)\leq N\coloneq \max(g,r)$, so Proposition~\ref{pr:stabdegstabrange} implies that $\{V_n,\phi_n\}$ is uniformly representation stable in degrees $\geq N+g$.

Another application of Proposition~\ref{pr:stabdegstabrange} shows that $\set{K_n}$ is uniformly representation stable in some appropriate range, and Church~\cite[Theorem 2.8]{Ch} states that $\set{\bigoplus_{i=0}^g M(i)^{\oplus b_i}_n}$ is monotone.  It then follows from \cite[Proposition 2.3]{Ch} that $\set{V_n}$ is monotone in some stable range.

It remains to show that uniform representation stability implies finite generation.  Assume that $\{V_n,\phi_n\}$ is uniformly representation stable for $n\geq N$. Definition~\ref{definition:repstab}.II  states that for $n\geq N$ the $S_{n+1}$-span of $\im(\phi_n\colon V_n\to V_{n+1})$ is all of $V_{n+1}$. Equivalently, $V_{n+1}$ is spanned by the images of all FI-maps $f_*\colon V_n\to V_{n+1}$ for $f\colon \n\into \npone$. By induction, this implies that $V=\spn(V_{\leq N})$. Since each $V_n$ is finite-dimensional by assumption, this shows that $V$ is finitely generated.\end{proof}

\para{Character polynomials} We finish this section by proving a refined version of Theorem~\ref{thm:intro:persinomial}. Recall that for each $i\geq 1$ and any $n\geq 0$, the class function $X_i\colon S_n\to \N$ is defined by
\[X_i(\sigma)\coloneq \text{\ number of $i$-cycles in $\sigma$.}\]
For example, $X_1(\sigma)$ is the number of fixed points of the permutation $\sigma$. Polynomials in the variables $X_i$ are called \emph{character polynomials}. Class functions form a ring under pointwise product, so any character polynomial $P\in \Q[X_1,X_2,\ldots]$ also defines a class function $P\colon S_n\to \Q$ for all $n\geq 0$. 
The \emph{degree} of a character polynomial is defined by setting $\deg(X_i)=i$.

\begin{thm}[{\bf Polynomiality of characters}]
\label{thm:persinomial}
Let $V$ be a finitely generated FI-module over a field of characteristic~0. There is a unique polynomial $P_V \in \Q[X_1, X_2, \ldots]$ with $\deg P_V\leq \wt(V)$ such that
for all $n\geq \stabD(V)+\wt(V)$ and all $\sigma\in S_n$, 
\[\chi_{V_n}(\sigma)=P_V(\sigma).\]
\end{thm}
\begin{proof}
Classically, the interest in character polynomials is driven by the following fact: for each partition $\lambda$, there exists a polynomial $P_\lambda\in\Q[X_1, X_2, \ldots]$ of degree $\abs{\lambda}$ such that $\chi_{V(\lambda)_n}(\sigma) = P_\lambda(\sigma)$ for all $n \geq \abs{\lambda} + \lambda_1$ and all $\sigma\in S_n$ \cite[Example I.7.14]{Mac}.
This seems to have been known, at least implicitly, as far back as Murnaghan; Macdonald traces it back to work of Frobenius in 1904. For a more recent reference, see Garsia--Goupil~\cite{GG}.


Proposition~\ref{pr:stabdegstabrange} implies that for every finitely generated FI-module $V$, there exist constants $c_\lambda$ so that $
V_n \simeq \bigoplus_\lambda c_\lambda V(\lambda)_n$ for all $n \geq  \stabD(V)+\wt(V)$.
If $\lambda$ is a partition with $c_\lambda \neq 0$, the definition of weight implies that $\abs{\lambda} \leq  \wt(V)$ and Proposition~\ref{pr:narrow} states that that $\lambda_1 \leq \stabD(V)$.
The polynomial $P_V\coloneq \sum_\lambda c_\lambda P_\lambda$ thus satisfies the desired claim.
\end{proof}


\subsection{Murnaghan's theorem and stability of Schur functors}
\label{section:SchurMurnaghan}

\para{The FI-module $V(\lambda)$} Given a partition $\lambda$, in Definition~\ref{def:Vlambdan} we defined the padded partition $\lambda[n]$ and the irreducible $S_n$-representation $V(\lambda)_n\coloneq V_{\lambda[n]}$. The following proposition says that all these $S_n$-representations arise from a single finitely-generated FI-module $V(\lambda)$ with $V(\lambda)_{\n}\simeq V(\lambda)_n$. 
\begin{prop}[{\bf The FI-module $V(\lambda)$}]
\label{pr:Vlambda}
Let $k$ be a field of characteristic~0. For any partition $\lambda$, there is an FI-module $V(\lambda)$ satisfying
\begin{equation}
\label{eq:Vlambdan}
V(\lambda)_\n\simeq \begin{cases}V_{\lambda[n]}&\text{if }n\geq \abs{\lambda}+\lambda_1\\0&\text{otherwise}\end{cases}
\end{equation}
that is finitely generated in degree $\abs{\lambda}+\lambda_1$. It has $\stabD(V(\lambda))=\lambda_1$ and $\wt(V(\lambda))=\abs{\lambda}$.
\end{prop}
\begin{proof}
We will define $V(\lambda)$ as a sub-FI-module of $M(\lambda)$. Let $d\coloneq \abs{\lambda}+\lambda_1$. Since every finite set is isomorphic to $\n$ for some $n\geq 0$, it suffices to define the subspace $V(\lambda)_\n\subset M(\lambda)_\n$ for $n\geq d$.

Lemma~\ref{lemma:branching}(i) states that $V_\mu$ occurs in $M(\lambda)_\n$ if $\mu\vdash n$ is obtained from $\lambda$ by adding one box to $n-\abs{\lambda}$ different columns. When $n-\abs{\lambda}\geq \lambda_1$, adding boxes to the \emph{first} $n-\abs{\lambda}$ columns yields the partition $\mu=\lambda[n]$. (For future reference, we note that $\lambda[n]$ has $\abs{\lambda}$ boxes below the first row, while all other partitions $\mu\neq \lambda[n]$ that occur have $<\abs{\lambda}$ boxes below the first row.)
The condition \eqref{eq:Vlambdan} that $V(\lambda)_\n\simeq V_{\lambda[n]}$ for $n\geq d$ thus uniquely defines the subspace $V(\lambda)_\n\subset M(\lambda)_\n$. 

We must now verify that these subspaces define a sub-FI-module $V(\lambda)$ of $M(\lambda)$. That is, we must show that for every  $f\in \Hom_{\FI}(\m,\n)$ the map $f_*\colon M(\lambda)_m\to M(\lambda)_n$ satisfies
$f_*(V(\lambda)_m)\subset V(\lambda)_n$. As we will see, this follows just from the fact that $f_*$ is $S_m$-equivariant. 

According to Pieri's rule \cite[Exercise 4.44]{FH}, the $S_m$-irreducible representations $V_{\nu}$ which occur in the restriction $\Res_{S_m}^{S_n}V_{\mu}$ are exactly those for which $\nu\vdash m$ can be obtained from $\mu\vdash n$ by removing $n-m$ boxes. We noted above that every $V_\mu$ with $\mu\neq \lambda[n]$ occurring in $M(\lambda)_n$ has $< \abs{\lambda}$ boxes below the first row. Any partition $\nu$ obtained by removing boxes from such a $\mu\neq \lambda[n]$ must also have $<\abs{\lambda}$ boxes below the first row. In particular, $V(\lambda)_m$ cannot occur in $\Res^{S_n}_{S_m}V_\mu$ for any such $\mu\neq \lambda[n]$.

We conclude that 
any $S_m$-equivariant map from $V(\lambda)_m$ to $M(\lambda)_n$ has image contained in $V(\lambda)_n$. In particular this applies to $f_*$, which shows that $V(\lambda)$ is a sub-FI-module of $M(\lambda)$.

We now show that $\spn(V_d)=V$. For any $n\geq d$, choose any injection $f\colon \dd\into \n$. The map $f_*\colon V(\lambda)_d\to V(\lambda)_n$ is injective, since $f_*\colon M(\lambda)_d\to M(\lambda)_n$ is. Its image $f_*(V(\lambda)_d)\subset V(\lambda)_n$ is thus nonzero, showing that $\spn(V_d)_n\neq 0$. Since $V(\lambda)_n$ is irreducible, it follows that $\spn(V_d)_n=V_n$, as desired. This argument (a special case of the ``monotonicity'' proved for $M(\lambda)$ in Church~\cite[Theorem 2.8]{Ch}) shows that $V$ is generated in degree $d=\abs{\lambda}+\lambda_1$.

That $\wt(V(\lambda))=\abs{\lambda}$ is immediate from \eqref{eq:Vlambdan}.
Since $V(\lambda)$ is a sub-FI-module of $M(\lambda)$, Lemma~\ref{lemma:injsurjdescend} implies that $V(\lambda)$ has injectivity degree 0. 
Let $\mu=(\lambda_1,\lambda_1,\ldots,\lambda_\ell)=\lambda[d]$.
Since $V(\lambda)$ is generated by $V(\lambda)_{d}\simeq V_{\mu}$, the FI-module $V(\lambda)$ is a quotient of $M(\mu)$. By Proposition~\ref{pr:stabdegMlambda} we know $\surjD(M(\mu))=\mu_1=\lambda_1$, so by Lemma~\ref{lemma:injsurjdescend} we have $\surjD(V(\lambda))\leq \lambda_1$. Lemma~\ref{lem:coinvariants}(ii) implies that $\Phi_{\abs{\lambda}}V(\lambda)_{\lambda_1}\simeq V_{\lambda}\neq 0$, while $\Phi_{\abs{\lambda}}V(\lambda)_{\lambda_1-1}=0$, so this bound is sharp.
\end{proof}

\para{Murnaghan's theorem}\label{section:murnaghan}
In 1938 Murnaghan stated the following theorem;  the first complete proof of this theorem was given in 1957 by Littlewood~\cite{Li}.
\begin{theorem}[{\bf Murnaghan's Theorem}]
\label{thm:Murnaghan}For each pair of partitions $\lambda$ and $\mu$ there exists a finite set $S$ of partitions $\nu$ and a set of nonnegative integers $g_{\lambda,\mu}^\nu$ such that 
for all sufficiently large $n$:
\begin{equation}
\label{eq:Murnaghan}
V(\lambda)_n \tensor V(\mu)_n = \bigoplus_{\nu \in S} g_{\lambda,\mu}^\nu V(\nu)_n.
\end{equation}
\end{theorem}

As we now demonstrate, Murnaghan's Theorem follows rather easily from the theory that we have already built up as a structural statement about a single object, the FI-module $V(\lambda)\otimes V(\mu)$.

\begin{proof}
 Since $V(\lambda)$ and $V(\mu)$ are finitely generated, Proposition~\ref{pr:tensor} implies that $V(\lambda)\otimes V(\mu)$ is finitely generated.  Theorem~\ref{th:FIequiv} then implies that the sequence of $S_n$-representations
\[
(V(\lambda) \tensor V(\mu))_n = V(\lambda)_n \tensor V(\mu)_n
\]
is uniformly representation stable, which implies \eqref{eq:Murnaghan}.
\end{proof}

\para{Stability for Schur functors}
Let $k$ be a field of characteristic~0. Given a partition $\lambda$, the \emph{Schur functor} $\Schur_\lambda$ is a functor from $k$-vector spaces to $k$-vector spaces whose properties are the subject of substantial interest in combinatorics and representation theory.  See \cite[Lecture 6]{FH} for the basic properties of Schur functors.  
Since $\Schur_\lambda$ is a covariant functor, for any FI-module $V$ we can consider the FI-module $\Schur_\lambda V$, which satisfies $(\Schur_\lambda V)_S\coloneq \Schur_\lambda(V_S)$.

\begin{prop}[{\bf Finite generation of Schur functors}]
\label{pr:schurfunctor}
Let $V$ be an FI-module over a field of characteristic 0. Assume that $V$ is finitely generated (resp.\ generated in degree $\leq m$, resp.\ of weight $\leq m$). Then for any partition $\lambda$, the FI-module $\Schur_\lambda(V)$ is finitely generated (resp.\ generated in degree $\leq m\cdot \abs{\lambda}$, resp.\ of weight $\leq m\cdot \abs{\lambda}$).
\end{prop}
\begin{proof}
Let $d=\abs{\lambda}$, and let $c_\lambda\in k[S_d]$ be the Young symmetrizer associated to the partition $\lambda$. The permutation group $S_d$ acts on the FI-module $V^{\otimes d}$ by permuting the factors, so $c_\lambda$ defines an idempotent endomorphism of $V^{\otimes d}$ whose image is isomorphic to $\Schur_\lambda(V)$. 
Propositions~\ref{pr:tensor} and \ref{pr:tensorwt} imply that the FI-module $V^{\otimes d}$ has the desired properties, so the same is true of its quotient $\Schur_\lambda(V)$.
\end{proof}

Applying Theorem~\ref{th:FIequiv} and Proposition~\ref{pr:schurfunctor} to the FI-module $V=M(1)$ over $\Q$ yields the following corollary, which resolves a basic issue left open in \cite{CF} (cf.\ Theorem 3.1 of \cite{CF}). The relative simplicity of the proof given here illustrates the power of the FI-module point of view.
\begin{corollary}[{\bf Stability of Schur functors}]
The sequence of $S_n$--representations $\{\Schur_\lambda(\Q^n)\}$ is monotone in the sense of \cite{Ch} and is uniformly representation stable.
\end{corollary}

Applying Proposition~\ref{pr:schurfunctor} to the finitely generated FI-module $V(\mu)$ from Proposition~\ref{pr:Vlambda} shows that the FI-module $\Schur_\lambda(V(\mu))$ is finitely generated. Theorem~\ref{th:FIequiv} then implies the following.

\begin{prop}[{\bf Schur functors of irreducibles}]  Let $\lambda, \mu$ be partitions.  There exists a finite set $S$ of partitions $\nu$ and a set of nonnegative integers $\beta_{\lambda,\mu}^\nu$ such that 
for all sufficiently large $n$:
\beq
\Schur_\lambda(V(\mu)_n) = \bigoplus_{\nu \in S}  \beta_{\lambda,\mu}^\nu V(\nu)_n
\eeq
\end{prop}

\section{\texorpdfstring{$\FIsharp$-modules}{FI\#-modules} and graded FI-modules}
In this section we describe and analyze two types of additional structure frequently carried by the FI-modules that arise in applications.

\subsection{\texorpdfstring{$\FIsharp$-modules}{FI\#-modules}}
\label{ss:fisharp}

Many of the sequences of $S_n$-representations we encounter in applications simultaneously carry both an FI-module and a co-FI-module structure, which are compatible in a certain precise sense. Together these give such sequences an $\FIsharp$-module structure, which is extremely rigid.

\begin{definition}[{\bf \boldmath$\FIsharp$-modules}]  Let $\FIsharp$ be the category whose objects are finite sets, with morphisms being \emph{partially-defined injections} $S\supset A\xrightarrow{\phi}B\subset T$. That is, $\Hom_{\FIsharp}(S,T)$ consists of triples $(A,B,\phi)$ with $A\subseteq S$, $B\subseteq T$ and $\phi\colon A \ra B$ an isomorphism, with the composition of
\[S\supset \phi^{-1}(B)\xrightarrow{\phi}B\subset T\quad\text{and}\quad T\supset C\xrightarrow{\psi}\psi(C)\subset U\quad\text{being}\quad
S\supset \phi^{-1}(B\cap C)\xrightarrow{\psi\circ \phi}\psi(B\cap C)\subset U.
\] 
An \emph{$\FIsharp$-module} $V$ is a functor from $\FIsharp$ to the category of $k$-modules, and $\FIsharpMod=[\FIsharp,\kMod]$ is the category of $\FIsharp$-modules.
 \end{definition}

\begin{xample}
\label{ex:standardFIsharp}
The most basic example is the $\FIsharp$-module $M(1)$ taking a finite set $S$ to the free $k$-module $M(1)_S=\langle e_s|s\in S\rangle$ with basis $S$. A morphism $f\colon S\supset A\xrightarrow{\phi}B\subset T$ induces the map $f_*\colon M(1)_S\to M(1)_T$ defined by $f_*(e_s)=e_{\phi(s)}$ if $s\in A$, and $f_*(e_s)=0$ otherwise.
\end{xample}

\begin{remark}
\label{remark:FIsharpremarks}
The category $\FI$ embeds in $\FIsharp$ by taking only those morphisms of the form $S=S\xrightarrow{\approx} B\subset T$, so every $\FIsharp$-module can be considered as an FI-module. The category $\coFI$ also embeds in $\FIsharp$ by taking only those morphisms of the form $S\supset A\xrightarrow{\approx} T=T$, so every $\FIsharp$-module can be considered as a co-FI-module as well. The relations in $\FIsharp$ impose conditions on how these FI-module and co-FI-module structures interact.

In particular, the endomorphisms $\End_{\FIsharp}(\n)$ form the so-called \emph{rook algebra} of rank $n$, so for any $\FIsharp$-module $V$ the $k$-module $V_n$ is a representation of the rook algebra. For the $\FIsharp$-module $M(1)$ in Example~\ref{ex:standardFIsharp}, $M(1)_n$ is the standard representation of the rook algebra on $k^n$. The basic properties of the rook algebra and its representation theory were determined by Munn and Solomon~\cite{So}.

The category $\FIsharp$ is naturally isomorphic to $\FIsharp^{\op}$, by simply taking $S\supset A\xrightarrow{\phi}B\subset T\in \Hom_{\FIsharp}(S,T)$ to $T\supset B\xrightarrow{\phi^{-1}}A\subset S\in \Hom_{\FIsharp^{\op}}(S,T)$. Therefore the dual of an $\FIsharp$-module is naturally an $\FIsharp$-module.
\end{remark}

\begin{xample}[{\bf \boldmath$M(W)$ is an \boldmath$\FIsharp$-module}]
\label{ex:MWFIsharp}
Given an FB-module $W$, we defined the FI-module $M(W)$ in Definition~\ref{def:MW}. It turns out this extends to  a natural $\FIsharp$-module structure on $M(W)$, as follows. To a finite set $S$, the $\FIsharp$-module $M(W)$ assigns
\[M(W)_S\coloneq \bigoplus_{T\subseteq S}W_T.\] To a morphism $f\colon S\supset A\xrightarrow{\phi}B\subset S'$, it assigns the map $f_*\colon M(W)_S\to M(W)_{S'}$ which on the factor $W_T$ is $(\phi|_T)_*\colon W_T\to W_{\phi(T)}$ if $T\subseteq A$, and is $0$ if $T\not\subset A$.
\end{xample}
When restricted to FI-morphisms (those with $A=S$), the condition $T\subseteq A$ is always satisfied, and so this definition agrees with Definition~\ref{def:MW}. This means that when considered as FI-modules, the two definitions of $M(W)$ agree, so there is no conflict of notation.
For example, the $\FIsharp$-module $M(1)$ of Example~\ref{ex:standardFIsharp} becomes  the FI-module $M(1)$ defined in Definition~\ref{def:Mm}.

\para{The classification of $\FIsharp$-modules}
The following classification theorem shows that \emph{every} $\FIsharp$-module is of the form $M(W)$ for some FB-module $W$.
Considering an $\FIsharp$-module as an FI-module provides a forgetful functor $\FIsharpMod\to \FIMod$. By slight abuse of notation, we denote the composition $\FIsharpMod\to \FIMod\xrightarrow{H_0} \FBMod$ also by $H_0$. To reiterate, if $V$ is an $\FIsharp$-module, then $H_0(V)$ is \emph{always} defined by considering $V$ as an FI-module and applying Definition~\ref{def:H0}.
\begin{thm}[{\bf Classification of \boldmath$\FIsharp$-modules}]
\label{th:charFIsharp}
The category of $\FIsharp$-modules is equivalent to the category of FB-modules, via the equivalence of categories
\[M(-)\colon \FBMod\leftrightarrows \FIsharpMod\cocolon H_0(-).\]
Thus every $\FIsharp$-module $V$ is of the form $V=\bigoplus_{i=0}^\infty M(W_i)$,
where $W_i$ is the $S_i$-representation $H_0(V)_i$. 
\end{thm}

\begin{remark}
Pirashvili established in \cite[Theorem~3.1]{Pi} an equivalence between the category of $\Gamma$-modules and FS-modules; here $\Gamma$ is the category of finite based sets and all based maps, and FS is the category of finite sets and surjections. The classification of $\FIsharp$-modules in Theorem~\ref{th:charFIsharp} is closely related to Pirashvili's result, although technically neither follows from the other. 

Since FB is a full subcategory of FS,  an FS-module can be thought of as an FB-module with additional morphisms. Moreover, the full subcategory of $\Gamma$ whose morphisms are  based maps that are injective away from the basepoint can be identified with $\FIsharp$, so a $\Gamma$-module can be thought of as an $\FIsharp$-module with certain additional morphisms. Given Theorem~\ref{th:charFIsharp}, one could think of Pirashvili's theorem saying that the necessary \emph{additional} information to promote an $\FIsharp$-module to a $\Gamma$-module is the same as the information needed to promote an FB-module to an FS-module (and indeed, in spirit this is how Pirashvili proves his result).
\end{remark}

\begin{proof}[Proof of Theorem~\ref{th:charFIsharp}] We have already seen the canonical isomorphism $H_0(M(W))\simeq W$ for any FB-module $W$ in Remark~\ref{remark:H0}. It remains to prove that for any $\FIsharp$-module, we have a canonical isomorphism $V\simeq M(H_0(V))$.

Fix $n\geq 0$. We will analyze the category of $\FIsharp$-modules $V$ satisfying the condition
\begin{equation}
\label{eq:FIsharpcond}
V_S=0\quad\text{ for all }S\text{ with }\abs{S}<n.
\end{equation}
The assumption \eqref{eq:FIsharpcond} implies that the quotient \eqref{eq:defH0} defining $H_0(V)_n$ is vacuous, so $H_0(V)_n\simeq V_n$ as $S_n$-representations.
Our goal will be to prove that we have a natural isomorphism
\begin{equation}
\label{eq:FIsharpgoal}
V\simeq M(V_n)\oplus V'\simeq M(H_0(V)_n)\oplus V',
\end{equation} where $V'$ is an $\FIsharp$-module for which $V_S=0$ whenever $\abs{S}\leq n$. The theorem then follows by induction on $n$ (beginning with $n=0$, when the  assumption \eqref{eq:FIsharpcond} is vacuous), since $H_0(V)=\bigoplus_n H_0(V)_n$.

\para{Notation} In this proof we use the following notation. Given a subset $C\subseteq S$, we denote by $I_C\in \End_{\FIsharp}(S)$ the morphism $I_C\colon S\supset C\xrightarrow{\id}C\subset S$. Given a morphism $f\colon S\supset A\xrightarrow{\phi}\phi(A)\subset T$ and subset $C\subset A$, we denote by $f|_C$ the morphism $f|_C\coloneq f\circ I_C=S\supset C\xrightarrow{\phi}\phi(C)\subset T$.

\para{The endomorphism $E$} Under the assumption \eqref{eq:FIsharpcond}, we define an endomorphism $E\colon V\to V$ of $\FIsharp$-modules by taking $E_S\colon V_S\to V_S$ to be:
\begin{equation}
\label{eq:FIsharpE}
E_S = \sum_{\substack{C \subset S\\ \abs{C} = n}} (I_C)_*
\end{equation}

We first verify that $E\colon V\to V$ is a map of $\FIsharp$-modules: given a morphism $f\colon S\supset A\xrightarrow{\phi}B\subset T$, we must check that $f_*\circ E_S=E_T\circ f_*\colon E_S\to E_T$. 

Expanding $f_*\circ E_S$ yields a sum of terms $f_*\circ (I_C)_*=(f\circ I_C)_*=(f|_{C\cap A})_*$, each of which factors as $(f|_{C\cap A})_*\colon V_S\to V_{C\cap A}\to V_T$. But if $C\not\subset A$ we have $\abs{C\cap A}<\abs{C}=n$, in which case our assumption guarantees $V_{C\cap A}=0$. Thus $(f|_{C\cap A})_*=0$ when $C\not\subset A$, and $f_*\circ E_S$ reduces to the sum over $C\subset A$ of $(f|_{C\cap A})_*=(f|_C)_*$. Applying a similar analysis to $I_D\circ f=f|_{\phi^{-1}(D\cap B)}$ shows that  \[f_*\circ E_S=\sum_{\substack{C\subset A\\|C|=n}}(f|_C)_*\qquad\text{and}\qquad E_T\circ f_*=\sum_{\substack{D\subset B\\|D|=n}}(f|_{\phi^{-1}(D)})_*.\] 
Since $C$ and $\phi^{-1}(D)$ range over the same sets, these sums coincide, and so $f_*\circ E_S=E_T\circ f_*$ as desired.

For future reference, we point out that if $V_n=0$, then $E\colon V\to V$ is the zero map. Indeed, this assumption implies that for any $C$ with $\abs{C}=n$ we have $V_C=0$. Since the map $(I_C)_*$ factors as $(I_C)_*\colon V_S\to V_C\to V_S$, we have $(I_C)_*=0$. It follows that $E_S=\sum_C (I_C)_*=0$, as claimed.

If $U$ is another $\FIsharp$-module satisfying \eqref{eq:FIsharpcond}, the definition \eqref{eq:FIsharpE} applies equally well to $U$ and defines an endomorphism $E^U\colon U\to U$.  Any map $F\colon U\to V$ of $\FIsharp$-modules respects the $\FIsharp$-morphisms $(I_C)_*$; since $E$ is defined in terms of these maps $(I_C)_*$, we have $E^V\circ F=F\circ E^U$. 

\para{The decomposition \boldmath$V\simeq EV\oplus \ker[E]$}
The endomorphism $E\colon V\to V$ is idempotent, as we now verify. Using the identity $I_C\circ I_{C'}=I_{C\cap C'}$, we compute \[E_S\circ E_S=\sum_{\substack{C,C'\subset S\\\abs{C}=\abs{C'}=n}}(I_{C\cap C'})_*=\sum_{\substack{C\subset S\\\abs{C}=n}}(I_C)_*=E_S.\] In the second equality, we used that $(I_{C\cap C'})_*$ factors through $V_{C\cap C'}$, and thus vanishes whenever $\abs{C\cap C'}<n$, i.e.\ whenever $C\neq C'$.

Since $E^2=E\in \End_{\FIsharpMod}(V)$, the $\FIsharp$-module $V$ splits as a direct sum $V\simeq EV\oplus \ker[E]$. Moreover, since $E^V\circ F=F\circ E^U$  for any map $F\colon U\to V$ of $\FIsharp$-modules satisfying \eqref{eq:FIsharpcond}, this decomposition $V\simeq EV\oplus \ker[E]$ is natural for all $\FIsharp$-modules satisfying \eqref{eq:FIsharpcond}. We will show that this provides the desired decomposition \eqref{eq:FIsharpgoal}.

When  $\abs{S}=n$ the definition of $E_S$ reduces to $E_S=(I_S)_*=\id\colon V_S\to V_S$, so $\ker[E]_S=0$ for such $S$. Since $V_S=0$ when $\abs{S}<n$, certainly  $\ker[E]_S=0$ in this case as well. Therefore $\ker[E]$ satisfies the desired properties of $V'$ from \eqref{eq:FIsharpgoal}, so it remains to exhibit an isomorphism $M(V_n)\simeq EV$.

\para{The map \boldmath$M(V_n)\to V$} 
We now construct a natural map of $\FIsharp$-modules $F\colon M(V_n)\to V$. We emphasize that without our assumption \eqref{eq:FIsharpcond} that $V_S=0$ when $\abs{S}<n$,  such a map need not exist! 

For readability, let $M\coloneq M(V_n)$. From Example~\ref{ex:MWFIsharp} we have that \[M_S=M(V_n)_S=\bigoplus_{\substack{T\subset S\\\abs{T}=n}}V_T.\] Given $T\subset S$, set $g_T\colon T=T\xrightarrow{\id}T\subset S$; then we define the map $F_S\colon M_S\to V_S$ on the factor $V_T$ to be $(g_T)_*\colon V_T\to V_S$. We must verify that $F\colon M\to V$ is a map of $\FIsharp$-modules, i.e.\ for any morphism $f\colon S\supset A\xrightarrow{\phi}\phi(A)\subset S'$ we have $f_*\circ F_S=F_{S'}\circ f_*\colon M_S\to V_{S'}$.

On the factor $V_T$, the map $F_{S'}\circ f_*$ is 0 if $T\not\subset A$, and is $(g_{\phi(T)})_*\circ (f|_T)_*$ if $T\subset A$.
In contrast, the map $f_*\circ F_S$ on the factor $V_T$ is $(f\circ g_T)_*$, where $f\circ g_T\colon T\supset T\cap A\xrightarrow{\phi|_{T\cap A}}\phi(T\cap A)\subset S'$. When $T\subset A$ this is simply $f\circ g_T\colon T=T\xrightarrow{\phi|_T}\phi(T)\subset S'$, which is indeed equal to $g_{\phi(T)}\circ \phi|_T$. In the remaining case, we can factor $(f\circ g_T)_*\colon V_T\to V_{T\cap A}\to V_{S'}$. When $T\not\subset A$ we have $\abs{T\cap A}<\abs{T}=n$, so our assumption implies $V_{T\cap A}=0$. Thus in this case $(f\circ g_T)_*=0$, as desired. This verifies that $F\colon M(V_n)\to V$ is a map of $\FIsharp$-modules. Since $F$ is constructed from the $\FIsharp$-morphisms $(g_T)_*$, it is natural for $\FIsharp$-modules satisfying \eqref{eq:FIsharpcond}.

\para{The isomorphism \boldmath$M(V_n)\simeq EV$}
Since $M$ satisfies \eqref{eq:FIsharpcond}, the map $F\colon M\to V$ must satisfy $E^V\circ F=F\circ E^M$.
We can easily analyze $E^M$: on the factor $V_T$ of $M_S$, the map $(I_C)_*$ is 0 unless $C=T$, when it is $(I_T)_*=\id$. Thus the sum $E_S=\sum_C (I_C)_*$ acts as the identity on each factor, so $E^M=\id\colon M\to M$.

We conclude that $E^V\circ F=F\circ E^M=F$, so the image of $F\colon M\to V$ is contained in $EV$.
We therefore have
\beq
0\to K\to M \xrightarrow{F} EV\to U\to 0
\eeq
where $K\coloneq \ker[F]$ and $U\coloneq \coker[F]$. 
Since $M_n=V_n=EV_n$ we have $K_n=0$ and $U_n=0$, and we observed above that this implies $E^K=0$ and $E^U=0$. However $E^M=\id$ implies $E^K=\id$, and similarly $E^{EV}=\id$ implies $E^U=\id$. We conclude that $K=0$ and $U=0$, so the natural map $F\colon M(V_n)\to EV$ is an isomorphism, as desired.
\end{proof}

From the classification in Theorem~\ref{th:charFIsharp} we see that if an $\FIsharp$-module $V$ is generated in degree $\leq d$ as an $\FI$-module, the same is true of any sub-$\FIsharp$-module of $V$. More generally, this classification has the following consequences, which in particular imply Theorem~\ref{thm:intro:FIsharp}. 
In any abelian category, one says that an object $V$ is {\em finitely generated} if for any directed family $\{W^i\}_{i\in I}$ of subobjects $W^i\subset V$ with $\bigcup W^i=V$, there exists $N\in I$ such that $W^N=V$. (For FI-modules, it is easy to see that this is equivalent to Definition~\ref{def:finitegen}.)
\begin{theorem}
\label{thm:polyFIsharp}
Let $V$ be an $\FIsharp$-module over $k$. The following conditions are equivalent:
\begin{enumerate}
\item[(i)] $V$ is finitely generated as an $\FIsharp$-module.
\item[(ii)] $V$ is finitely generated as an FI-module.
\item[(iii)] $\bigoplus_{n\geq 0} H_0(V)_n$ is finitely generated as a $k$-module.
\item[(iv)] There exists $d\geq 0$ such that $V_n$ is generated as a $k$-module by $O(n^d)$ elements.
\end{enumerate}
When $k$ is a field, these are also equivalent to the following conditions:
\begin{enumerate}
\item[(v)] There exists $d\geq 0$ such that $\dim_k V_n=O(n^d)$.
\item[(vi)] There exists an integer-valued polynomial $P\in \Q[T]$ such that $\dim_k V_n = P(n)$ for all $n\geq 0$.
\end{enumerate}
When $k$ is a field of characteristic 0, these are also equivalent to the following condition:
\begin{enumerate}
\item[(vii)] There exists a  polynomial $P_V\in \Q[X_1,X_2,\ldots]$ s.t.\ $\chi_{V_n}(\sigma)=P_V(\sigma)$ for all $n\geq 0$ and all $\sigma\in S_n$.
\end{enumerate}
\end{theorem}
\begin{proof} It is immediate that $(ii) \implies (i)$, and we have already seen in Remark~\ref{remark:H0} that $(ii)\iff (iii)$. Consider the FB-module $W=H_0(V)$, so that $V\simeq M(W)$.  Theorem~\ref{th:charFIsharp} states that $(i)$ holds if and only if $H_0(V)$ is finitely generated as an FB-module. Since $k[S_n]$ is finite over $k$, this implies that $\bigoplus_{n\geq 0}H_0(V)_n$ is finitely generated as a $k$-module, so $(i)\implies (iii)$.

Let $c_i$ be the smallest number of generators for the $k$-module $H_0(V)_i=W_i$.  Condition $(iii)$ states that there exists $d$ such that $c_i<\infty$ for all $i\geq d$ and $c_i=0$ for all $i>d$.
From \eqref{eq:MWdecomp} we have $M(W)_n\simeq \bigoplus_{T\subset \n} W_T$, so as $k$-modules we have $V_n=M(W)_n\simeq \bigoplus_{i\geq 0}W_i^{\oplus \binom{n}{i}}$. This is generated by $\sum_{i\geq 0} c_i\cdot \binom{n}{i}$ elements, so $(iii)\implies (iv)$. Conversely, assume that $V_n$ is generated by $O(n^d)$ elements, so $V_n$ admits a surjection from $k^{c n^d}$ for some constant $c$.  Suppose furthermore that $W_i$ is nonzero for some $i > 0$, and let $\mathfrak{m}$ be a maximal ideal of $k$ with quotient field $\F\coloneq k/\mathfrak{m}$ such that $W_i \otimes_k \F\neq 0$.  The hypothesized surjection shows that the  $\F$-dimension of $k^{c n^d} \otimes_k \F\simeq \F^{cn^d}$ is at least that of $W_i^{\oplus \binom{n}{i}} \otimes_k \F\simeq (W_i\otimes_k \F)^{\oplus\binom{n}{i}}$ for all $n$, which  implies $i \leq d$.  Therefore $W_i$ is zero for all $i > d$. Since $V_i$ surjects to $W_i$, the $k$-module $W_i$ must be finitely generated for $i \leq d$. This shows that $(iv)\implies (iii)$.

When $k$ is a field, $(iv)\iff (v)$ are equivalent by definition, and $(vi)\implies(v)$ is immediate. Under the assumption of $(iii)$, we can take $P(n)$ to be the polynomial $P(n)=\sum_{i=0}^d \dim W_i\cdot \binom{n}{i}$. Then $\dim_k V_n=P(n)$ for all $n\geq 0$, so $(iii)\implies (vi)$.

When $k$ is a field of characteristic 0, specializing the polynomial $P_V$ of $(vii)$ to $P(n)\coloneq P_V(n,0,0,\ldots)$ shows that $(vii)\implies (vi)$. We will prove that $(iii)\implies (vii)$ by providing an explicit formula for the character polynomial $P_V$ under the assumption that $V\simeq M(W)$. 

Given a partition $\lambda\vdash d$, let $n_i(\lambda)$ be the number of parts of $\lambda$ equal to $i$, so that $\sum i\cdot n_i(\lambda)=\abs{\lambda}$. Define the polynomial $\binom{X_\bullet}{\lambda}\in \Q[X_1,X_2,\ldots]$ by: \[\binom{X_\bullet}{\lambda}\coloneq \binom{X_1}{n_1(\lambda)}\binom{X_2}{n_2(\lambda)}\cdots\binom{X_d}{n_d(\lambda)}\]
For $\lambda\vdash d$, let $\chi_W(\lambda)$ denote the character of $W_d$ on the conjugacy class of $S_d$ whose cycle decomposition is encoded by $\lambda$.
We then define the character polynomial $P_V\in \Q[X_1,X_2,\ldots]$ to be: \[P_V\coloneq \sum_{\lambda}\chi_{W}(\lambda)\binom{X_\bullet}{\lambda}\] 
The assumption $(iii)$ guarantees that this sum is finite.

To verify that $\chi_{V_n}=P_V$, we directly compute the character of $M(W)_n$. Fix $n\geq 0$ and $\sigma\in S_n$. As above, we have $M(W)_n=\bigoplus_{T\subset \n} W_T$. Only those summands with $\sigma(T)=T$ will contribute to the character of $\sigma$. A subset $T$ fixed by $\sigma$ is a union of orbits; if $T$ is the union of $n_i(\lambda)$ $i$-cycles, $\sigma|_T$ has cycle type $\lambda$ and thus contributes $\chi_{W}(\lambda)$ to $\chi_{M(W)_n}(\sigma)$. The number of such subsets is $\binom{X_1(\sigma)}{n_1(\lambda)}\cdots\binom{X_d(\sigma)}{n_d(\lambda)}=\binom{X_\bullet}{\lambda}(\sigma)$. We conclude that  $\chi_{M(W)_n}(\sigma)= \sum\chi_W(\lambda)\binom{X_\bullet}{\lambda}(\sigma)= P_V(\sigma)$ for any $n\geq 0$ and any $\sigma\in S_n$. This shows that $(iii)\implies(vii)$, completing the proof of the theorem.
\end{proof}

Finally, we have the following corollary regarding representation stability for $\FIsharp$-modules.
\begin{corollary}
\label{co:fisharprange}
Let $V$ be an $\FIsharp$-module over a field of characteristic~0. The sequence of $S_n$-representations $\{V_n,\phi_n\}$ is uniformly representation stable with stable range $n\geq 2\cdot\wt(V)$.
\end{corollary}
\begin{proof}
We first show that $\stabD(V)\leq \wt(V)$.  By Theorem~\ref{th:charFIsharp}, we have an isomorphism $V\simeq \bigoplus_\lambda M(\lambda)^{\oplus c_\lambda}$. Applying Propositions~\ref{pr:stabdegMlambda} and \ref{pr:Mmuweight}, each factor $M(\lambda)$ occurring in this sum satisfies $\stabD(M(\lambda))=\lambda_1\leq \abs{\lambda}=\wt(M(\lambda))$, so the same inequality holds for $V$. The corollary now follows from Proposition~\ref{pr:stabdegstabrange}.
\end{proof}

\remove{
\begin{alphremark}[{\bf Finite projective resolutions}]
\label{rem:finiteprojresolution} Over a field of characteristic~0, Remark~\ref{rem:projective} implies that every projective FI-module extends to an $\FIsharp$-module. So Theorem~\ref{thm:polyFIsharp} implies that the characters of any finitely generated projective FI-module are polynomial for all $n\geq 0$. Since characters are additive in exact sequences, the same is true for any finitely generated FI-module admitting a finite projective resolution: the character $\chi_{V_n}$ is given by the character polynomial $P_V$ for \emph{all} $n\geq 0$.

So if this is \emph{not} the case, then $V$ does not have finite projective dimension.  For example, let $W$ be the FI-module from Remark~\ref{rem:nonsemisimple} with $W_0=k$ and $W_n=0$ for each $n > 0$. This has character polynomial $P_V = 0$, but when $n=0$, the character $\chi_{V_0}(\text{id})=\dim V_0=1$ fails to agree with $P_V(0,\ldots,0)=0$, so this FI-module has no finite projective resolution. 
\end{alphremark}
}

\remove{
\para{Tensor products of $\FIsharp$-modules}
When $k$ is a field of characteristic~0,  Remark~\ref{rem:projective} and Theorem~\ref{th:charFIsharp} together imply
 that
the projective FI-modules are precisely those that can be extended to $\FIsharp$-modules. The tensor product of two $\FIsharp$-modules is an $\FIsharp$-module, so in particular the tensor project of two projective FI-modules over a field of characteristic~0 is projective.

Recall the $\FIsharp$-module $M(\lambda)=M(V_\lambda)$ generated by the single irreducible representation $V_\lambda$. Theorem~\ref{th:charFIsharp} states that every $\FIsharp$-module is a direct sum of $\FIsharp$-modules $M(\nu)$. In particular, every tensor product $M(\lambda)\otimes M(\mu)$ decomposes as a direct sum of finitely many $\FIsharp$-modules $M(\nu)$. It is not hard to compute small examples by hand; for instance:
\begin{equation}
\label{eq:dlmn}
\begin{array}{rcl}
M(\y{1}) \tensor M(\y{1}) &= & M(\y{2}) \oplus M\big(\x\y{1,1}\x\big) \oplus M(\y{1}) \\[2pt]
M(\y{1}) \tensor M(\y{2}) &=&  M(\y{3}) \oplus M\big(\x\y{2,1}\x\big) \oplus M(\y{2}) \oplus M\big(\x\y{1,1}\x\big) \\[2pt]
M(\y{2}) \tensor M(\y{2}) &= & M(\y{4}) \oplus M\big(\x\y{3,1}\x\big) \oplus M\big(\x\y{2,2}\x\big)\\&&\!\! {}\oplus M(\y{3}) \oplus M\big(\x\y{2,1}\x\big)^{\oplus 2} \oplus M\Big(\x\y{1,1,1}\x\Big) \oplus M(\y{2}).
\end{array}  
\end{equation}
These computations have natural combinatorial interpretations. For example, $M(1)=M(\y{1})$ takes a finite set $S$ to the vector space with basis $S$, so $M(\y{1})\otimes M(\y{1})$ associates to $S$  the vector space with basis the ordered pairs $(x,y)$ from $S$. We can partition these into those pairs with $x=y$, yielding the summand $M(\y{1})$, and those with $x\neq y$, yielding the summand $M(2)=M(\y{2})\oplus M\big(\x\y{1,1}\x)$.

For any partitions $\lambda$ and $\mu$ we have a direct sum decomposition
\begin{equation}
\label{eq:Kronecker}
M(\lambda) \tensor M(\mu) = \bigoplus_\nu d_{\lambda, \mu}^\nu M(\nu)
\end{equation}
The coefficients $d_{\lambda,\mu}^\nu$ are nonnegative integers, and it can be shown that $d_{\lambda,\mu}^\nu$ is only nonzero when $\max(\abs{\lambda},\abs{\mu})\leq \abs{\nu}\leq \abs{\lambda}+\abs{\mu}$.
It is straightforward to check that when $\abs{\lambda} = \abs{\mu} = \abs{\nu}=n$, the coefficient $d_{\lambda, \mu}^\nu$ is equal to the Kronecker coefficient (the multiplicity of $V_\nu$ in $V_\lambda\otimes V_\mu$).

Moreover, any Schur functor $\Schur_\lambda$ yields an $\FIsharp$-module $\Schur_\lambda(M(\y{1}))$ whose ``leading term'' is $M(\lambda)$, in the sense that $\Schur_\lambda(M(\y{1}))=M(\lambda)\oplus V$ where $V$ is generated in degrees $<\abs{\lambda}$.
It follows that when $\abs{\nu} = \abs{\lambda} + \abs{\mu}$, the coefficient $d_{\lambda, \mu}^\nu$ is equal to the Littlewood-Richardson coefficient $c_{\lambda, \mu}^\nu$ (the coefficient of $\Schur_\nu W$ in $\Schur_\lambda W\otimes \Schur_\mu W$). The honeycomb model used by Knutson--Tao in their proof of the saturation conjecture \cite{KT} gives a geometric interpretation for
the Littlewood--Richardson coefficients $c_{\lambda,\mu}^\nu$ as the number of integer
points in a certain Berenstein--Zelevinsky polytope. It would be very
interesting to find a similar geometric interpretation for the
coefficients $d_{\lambda,\mu}^\nu$.

\begin{alphprob}[{\bf $\FIsharp$-module tensor coefficients}]
Give a geometric interpretation of the structural coefficients $d_{\lambda,\mu}^\nu$ in \eqref{eq:Kronecker}.  Give a method for determining which of these coefficients are nonzero.
\end{alphprob}
}

\subsection{Graded FI-modules and FI-algebras} 
\label{ss:graded}

The FI-modules that arise in the applications of \S\ref{section:FIalgebrasandcoinv}, \S\ref{section:configspaces}, and \S\ref{section:moduli} all come with a natural grading.  When this happens, the FI-module will almost never be finitely generated, even if the graded pieces are.
For example, the FI-module that associates to a finite set $S$ the polynomial algebra ${R_S\coloneq \Q[x_s\,|\,s\in S]}$ is definitely not finitely generated, since $R_S$ is not even finite-dimensional for a fixed set $S$. However, for any fixed $i\geq 0$, the homogeneous degree-$i$ polynomials form a sub-FI-module $R^i$ which is finitely generated in degree $\leq i$. This example motivates the following definition.

Recall that a graded FI-module $V$ is a functor from $\FI$ to $\N$-graded $k$-modules. For each $i\geq 0$, restricting $V_S$ to the piece $V^i_S$ in grading $i$ yields an FI-module  $V^i$, and $V$ is equivalent to the collection of FI-modules $\{V^i\}_{i\in \N}$.
\begin{definition}[{\bf Graded FI-modules of finite type}]
Let $V$ be a graded FI-module. We say that $V$ is \emph{of finite type} if each FI-module $V^i$ is a finitely generated FI-module. If $k$ is a field of characteristic~0, we say that $V$ has \emph{slope $\leq m$} if $\wt(V^i)\leq m\cdot i$ for all $i\in \N$. (For instance, this occurs when $V^i$ is generated in degree $\leq m\cdot i$, via Proposition~\ref{pr:fgweight}.)
\end{definition}
If $V$ is a graded FI-module of finite type, any quotient of $V$ is of finite type.  When $k$ contains $\Q$, Theorem~\ref{th:noetherian} implies that any \emph{subquotient} of $V$ is of finite type. If $V$ has slope $\leq m$, any subquotient of $V$ has slope $\leq m$. (Whenever we speak about slope, we implicitly assume that $k$ is a field of characteristic 0, since the weight of an FI-module is only defined in this case.) Although a graded FI-module over $k$ can be of finite type without having finite slope (if $\wt(V^i)$ grows faster than linearly), we do not know of any interesting examples where this is the case.  

\para{Tensor products of FI-modules of finite type} If $A$ and $B$ are graded $k$-modules, their tensor product $A\tensor B$  is the graded $k$-module defined by
\begin{equation}
\label{eq:gradedtensor}
(A\otimes B)^k=\bigoplus_{i+j=k}A^i\otimes B^j.
\end{equation} Given graded FI-modules $V$ and $W$, their tensor product $V\otimes W$ is the graded FI-module obtained by applying \eqref{eq:gradedtensor} pointwise. The following two propositions show that we can safely take tensor products of graded FI-modules of finite type.
\begin{proposition}
\label{pr:gradedtensor}
Let $U$ and $W$ be graded FI-modules. If $U$ and $W$ are of finite type, the tensor product $U\tensor W$ is of finite type. If $U$ and $W$ have slope $\leq m$, then $U\tensor W$ has slope $\leq m$.
\end{proposition}
\begin{proof}
By Proposition~\ref{pr:tensor}, $U^i\otimes W^j$ is a finitely generated FI-module for each fixed $i$ and $j$. For fixed $k$ the graded piece $(U\otimes W)^k$ is the direct sum of $k+1$ such summands, and thus is finitely generated. The second claim follows from the inequality $\wt(U^i\otimes W^j)\leq \wt(U^i)+\wt(W^j)\leq mi+mj=mk$ from Proposition~\ref{pr:tensorwt}.
\end{proof}

\para{FI-algebras}
Recall from Remark~\ref{remark:FIalgebras} that a \emph{graded FI-algebra} $A$ is a functor from FI to the category of graded $k$-algebras. We say that a graded FI-algebra is of finite type if the underlying graded FI-module is of finite type.

Given a $k$-module $W$, we denote by $k\{W\}$ the free non-associative algebra on $W$; if $W$ is a graded $k$-module, $k\{W\}$ inherits a grading just as in \eqref{eq:gradedtensor}. Thus if $V$ is a graded FI-module, the free algebra $k\{V\}$ is a graded FI-algebra.

Given a graded sub-FI-module $V\subset A$, we say that $A$ is \emph{generated} by $V$ if $V_S$ generates $A_S$ as a $k$-algebra for all finite sets $S$. The inclusion $V\subset A$ induces a natural map $k\{V\}\to A$, and $A$ is generated by $V$ exactly when this map is a surjection $k\{V\}\onto A$.

Given $\Sigma\subset \coprod A_n^i$, we denote by $((\Sigma))\subset A$ the FI-ideal generated by $\Sigma$, i.e.\ the smallest FI-ideal of $A$ containing $\Sigma$.
If $V\subset A$ is a graded sub-FI-module, the FI-ideal $((V))\subset A$ is easy to describe: $((V))_S\subset A_S$ is simply the ideal in $A_S$ generated by $V_S$. 

\begin{theorem}[{\bf Algebras generated by finite type FI-modules}]
\label{thm:fisubalgebra}
Let $A$ be a graded FI-algebra generated by a graded sub-FI-module $V\subset A$ with  $V^0=0$. If $V$ is of finite type then $A$ is of finite type. If $V$ has slope $\leq m$ then $A$ has slope $\leq m$.
\end{theorem}
\begin{proof}
We first consider $k\{V\}$, which decomposes as a graded FI-module as \[k\{V\} \simeq \bigoplus_{j=0}^\infty (V^{\otimes j})^{\oplus P_j}\]
where $P_j$ is the set of parenthesizations of a $j$-fold product (so that $\abs{P_j}$ is the $j$-th Catalan number).

If $V$ is of finite type (resp.\ has slope $\leq m$), the same is true of each summand $V^{\otimes j}$ by Proposition~\ref{pr:gradedtensor}.
Moreover, since $V^i\neq 0$ only for $i\geq 1$, we have $(V^{\otimes j})^i\neq 0$ only for $i\geq j$. For fixed $i\in \N$ we thus have a finite sum $k\{V\}^i=\bigoplus_{j=0}^i\big((V^{\otimes j})^i\big)^{\oplus P_j}$. We conclude that the FI-module $k\{V \}^i$ is finitely generated (resp.\ of weight $\leq m\cdot i$) for each $i$. By assumption we have a surjection $k\{V\}\onto A$, so the same is true of $A^i$; this completes the proof.
\end{proof}

Theorem~\ref{thm:fisubalgebra} 
applies to each of the FI-algebras
\[
\begin{array}{l}
k\{V\} \coloneq \text{the free  algebra on $V$}\\
k\langle V\rangle \coloneq \text{the free associative algebra on $V$}\\
k[V]\coloneq \text{the free commutative algebra on $V$}\\
\L(V)\coloneq \text{the free Lie algebra on $V$}\\
\Gamma[V]\coloneq \text{the free graded-commutative $k$-algebra on $V$}
\end{array}
\]
as well as any quotient of these algebras by any FI-ideal. Many of the applications in the second half of this paper will follow directly from these results.

\para{Co-FI-algebras} Let $W$ be a graded co-FI-module over a field $k$. We say that $W$ is \emph{of finite type} (resp.\ \emph{has slope $\leq m$}) if its dual $W^*$ is of finite type (resp.\ has slope $\leq m$) as a graded FI-module. We say a graded co-FI-algebra $A$ is generated by $W\subset A$ if the natural map $k\{W\}\to A$ is surjective.

If $W$ is of finite type then any $U\subset W$ is of finite type since the injection $U\into W$ induces a surjection $W^*\onto U^*$. Similarly, when $\car(k)=0$, any quotient of $W$ has finite type by Theorem~\ref{th:noetherian}. The graded FI-module $k\{W\}^*\simeq k\{W^*\}$ is of finite type by Theorem~\ref{thm:fisubalgebra}, so we obtain the following proposition as a corollary.
\begin{prop}
\label{pr:cofisubalgebra} Let $A$ be a graded co-FI-algebra over a field of characteristic~0, generated by a graded co-FI-submodule $W\subset A$ with $W^0=0$. If $W$ is of finite type, $A$ is of finite type. If $V$ has slope $\leq m$, then $A$ has slope $\leq m$.
\end{prop}
\remove{
\begin{alphcorollary}
\label{cor:subalgebraslope}
If a graded FI-algebra (resp.\ co-FI-algebra) $A$ is generated by a sub-FI-module (resp.\ co-FI-module) $V$ concentrated in grading $1$ (i.e.\ $V^i=0$ for $i\neq 1$), then $A$ has slope $\leq \wt(V)$. 
\end{alphcorollary}
}
\pagebreak

\para{Another tensor product $V^{\tensor\bullet}$}
In Sections~\ref{section:FIalgebrasandcoinv} and \ref{section:configspaces} we will need to consider a different sort of tensor product. For any set $S$ and any $k$-module $W$, let $W^{\tensor S}\coloneq \bigotimes_{s\in S}W$. Given $F\colon W\to U$ and a bijection $f\colon S\to T$, let  $F^{\tensor f}\colon W^{\tensor S}\to U^{\tensor T}$ be the map that takes the factor labeled by $s\in S$ to the factor labeled by $f(s)\in T$ via $F$.
\begin{definition}
\label{def:tensorbullet}
Given an $\FIsharp$-module $V$ equipped with a splitting $M(0)\into V\onto M(0)$,  define the $\FIsharp$-module $V^{\tensor\bullet}$ by $(V^{\tensor\bullet})_S\coloneq (V_S)^{\otimes S}$. 
 Given $f\colon S\supset A\xrightarrow{\phi}B\subset T$, let $f_*\colon (V^{\tensor\bullet})_S\to (V^{\tensor\bullet})_T$ be the  composition
\begin{equation}
\label{eq:tensorbullet}
\bigotimes_{s\in S}V_S\onto  \bigotimes_{s\in A}V_S\xrightarrow{(f_*)^{\otimes \phi}}\bigotimes_{t\in B}V_T\into \bigotimes_{t\in T}V_T
\end{equation}
Here the first map applies the surjection $V_S\onto M(0)_S\simeq k$ to each factor with $s\not\in A$, followed by the canonical isomorphism $(V_S)^{\tensor A}\tensor k^{\tensor S-A}\simeq (V_S)^{\tensor A}$. Similarly the last map composes the isomorphism $(V_T)^{\tensor B}\simeq (V_T)^{\tensor B}\tensor k^{\tensor T-B}$ with the inclusions $k\simeq M(0)_T\into V_T$.

We can apply the construction \eqref{eq:tensorbullet} to FI-modules and co-FI-modules as well. Given an FI-module $V$ equipped with an injection $M(0)\into V$, let $V^{\tensor\bullet}$ be the FI-module with $(V^{\tensor\bullet})_S\coloneq (V_S)^{\otimes S}$, where $f\colon S\into T$ acts by \[f_*\colon
(V_S)^{\tensor S}\xrightarrow{(f_*)^{\otimes f}}(V_T)^{\otimes f(S)}\into (V_T)^{\otimes T}.\]

Given a co-FI-module $V$ equipped with an surjection $V\onto M(0)$, let $V^{\tensor\bullet}$ be the co-FI-module with $(V^{\tensor\bullet})_S\coloneq (V_S)^{\otimes S}$, where $f\colon S\into T$ acts by \[f_*\colon
(V_T)^{\tensor T}\onto (V_T)^{\tensor f(S)}\xrightarrow{(f_*)^{\otimes {f^{-1}}}}(V_S)^{\otimes S}.\]

If $V$ is a graded $\FIsharp$-module (or FI-module, or co-FI-module), the individual pieces $(V_S)^{\otimes S}$ are graded $k$-modules. In this case we consider $M(0)$ as concentrated in grading 0 (i.e.\ $M(0)^i=0$ for $i\neq 0$), and require that the maps $M(0)\into V$ and/or $V\onto M(0)$ preserve the grading. The grading is then preserved by the morphisms $f_*\colon (V_S)^{\otimes S}\to (V_T)^{\otimes T}$, and so $V^{\otimes \bullet}$ is a graded $\FIsharp$-module (or FI-module, or co-FI-module).

If $V$ is an $\FIsharp$-algebra (or FI-algebra, or co-FI-algebra), the individual pieces $(V_S)^{\otimes S}$ are $k$-algebras. In this case we consider $M(0)$ as an $\FIsharp$-algebra with multiplication $M(0)_S\otimes M(0)_S=k\otimes k\simeq k=M(0)_S$, and require that the maps $M(0)\into V$ and/or $V\onto M(0)$ are maps of $\FIsharp$-algebras (or FI-algebras, or co-FI-algebras). The morphisms $f_*\colon (V_S)^{\otimes S}\to (V_T)^{\otimes T}$ are then maps of $k$-algebras, and so $V^{\otimes \bullet}$ is an $\FIsharp$-algebra (or FI-algebra, or co-FI-algebra).
\end{definition}

\begin{remark}
If $U$ is a graded $k$-module with splitting $k\into U^0\onto k$, we can consider $U$ as a ``constant'' graded $\FIsharp$-module, sending every finite set $S$ to $U$ and every morphism to the identity. Definition~\ref{def:tensorbullet} thus defines a graded $\FIsharp$-module $U^{\otimes \bullet}$, which satisfies $(U^{\tensor\bullet})_n\simeq U^{\otimes n}$.
\end{remark}

\begin{proposition}
\label{pr:Vtensorbullet}
Let $V$ be a graded FI-module endowed with an isomorphism $V^0\simeq M(0)$. If $V$ is of finite type then the FI-module $V^{\tensor\bullet}$ is of finite type.
\end{proposition}
\begin{proof}
Consider a graded FI-module $M$ that is free in each grading, meaning that there exists an index set $L$, and for each $\ell\in L$ numbers $m_\ell\in \N$ and $i_\ell\in \N$, so that $M^i\simeq \bigoplus_{i_\ell=i}M(m_\ell)$. 
Every graded FI-module $V$ admits a surjection $M\onto V$ from such a graded FI-module $M$. If $V$ is of finite type, then we can take $M$ to be of finite type, meaning that $\{\ell \in L\,|\, i_\ell=j\}$ is finite for each $j$. If $V^0\simeq M(0)$, we can assume that there is a unique $\ell_0\in L$ with $i_{\ell_0}=0$, and it satisfies $m_{\ell_0}=0$. Since a surjection $M\onto V$ induces a surjection $M^{\tensor\bullet}\onto V^{\tensor \bullet}$, it suffices to prove that $M^{\tensor\bullet}$ is of finite type.  

A basis for $M^{\tensor\bullet}_S\simeq (M_S)^{\tensor S}$ is given by a choice of indices $\eta\colon S\to L$, and for each $s\in S$ an injection $g_s\colon \m_{\eta(s)}\into S$. For such a basis element, the multiset $\eta(S)$ can be written uniquely for some $j\leq \abs{S}$ as $\{\ell_1,\ldots,\ell_j\}\cup \{\ell_0,\ldots,\ell_0\}$ with $\ell_1,\ldots,\ell_j\in L-\{\ell_0\}$. Under any FI-map $f_*\colon M^{\tensor\bullet}_S\to M^{\tensor \bullet}_T$, such a basis element is taken to another basis element determining the same multiset $\uell=\{\ell_1,\ldots,\ell_j\}$. Therefore the graded FI-module $M^{\tensor\bullet}$ splits as a direct sum $M^{\tensor\bullet}=\bigoplus_{\uell}  M^{\tensor\bullet}_{\uell}$ indexed by finite multisets $\uell=\{\ell_1,\ldots,\ell_j\}$ of elements of $L-\{\ell_0\}$.

We now show that each summand $M^{\tensor\bullet}_{\uell}$ is finitely generated.
Fix a multiset $\uell=\{\ell_1,\ldots,\ell_j\}$ of elements of $L-\{\ell_0\}$. 
For any finite set $T$, a basis for $(M^{\tensor\bullet}_{\uell})_T$ is determined by
\begin{equation}
\label{eq:Mtensorbulletbasis}
\big\{\ (\eta\colon T\to L,g_t\colon \m_{\eta(t)}\into T)\,\big|\,\eta(T)=\{\ell_1,\ldots,\ell_j\}\cup \{\ell_0,\ldots,\ell_0\}\ \big\}
\end{equation}  If $\abs{T}>j+m_{\ell_1}+\cdots+m_{\ell_j}$, there must exist some $t\in T$ with $\eta(t)=\ell_0$ such that $t\not\in \im(g_{t'})$ for all $t'\in T$. Set $S\coloneq T-\{t\}$, so that $\im(g_s)\subset S$ for all $s\in S$, and let $f\colon S\into T$ be the inclusion. Under the induced map $f_*\colon (M^{\tensor\bullet}_{\uell})_S\to (M^{\tensor\bullet}_{\uell})_T$, the basis element $(\eta|_S,g_s)$ of $(M^{\tensor\bullet}_{\uell})_S$ is sent to the basis element $(\eta,g_t)$ of $(M^{\tensor\bullet}_{\uell})_T$. This shows that the FI-module $M^{\tensor\bullet}_{\uell}$ is generated in degree $\leq j+m_{\ell_1}+\cdots+m_{\ell_j}$. Since the basis \eqref{eq:Mtensorbulletbasis} is finite for each $T$, it follows that $M^{\tensor\bullet}_{\uell}$ is finitely generated.

The summand $M^{\tensor\bullet}_{\uell}$ contributes to $(M^{\tensor\bullet})^i$ only in grading $i=i_{\ell_1}+\cdots+i_{\ell_j}$. Since $\{\ell\in L\,|\,i_\ell=j\}$ is finite for each $j$, for fixed $i\in \N$ there are only finitely many multisets $\uell$ with $i=i_{\ell_1}+\cdots+i_{\ell_j}$. Therefore for each $i\in \N$ the FI-module $(M^{\tensor\bullet})^i$ is a finite direct sum of finitely generated FI-modules $M^{\tensor\bullet}_{\uell}$, so the graded FI-module $M^{\tensor\bullet}$ is of finite type, as desired.
\end{proof}

\section{Applications:  Coinvariant algebras and rank varieties}
\label{section:FIalgebrasandcoinv}
In this section we apply the theory of FI-modules
to explicit  FI-algebras arising in algebraic combinatorics and algebraic geometry, obtaining concrete results about important families of $S_n$-representations.

\subsection{Algebras with explicit presentations}
One way that FI-algebras and co-FI-algebras naturally arise is from explicit presentations by generators and relations. For example, we can consider the free commutative graded $\FIsharp$-algebra $k[M(1)]$ generated by $M(1)$, so that $k[M(1)]_n$ is the polynomial algebra $k[x_1,\ldots,x_n]$. We will consider $k[M(1)]$ as an $\FIsharp$-algebra, an FI-algebra, or a co-FI-algebra, depending on the application.

\begin{definition}[{\bf Symmetrically presented FI-algebra}]
Given a collection $\PP=\{P_j(x_1,\ldots,x_{m_j})\}$ of homogeneous polynomials, the \emph{symmetrically presented FI-algebra} $R(\PP)$ is the quotient $k[M(1)]/((\PP))$.
\end{definition}
By Theorem~\ref{thm:fisubalgebra}, $R(\PP)$ is a graded FI-algebra of finite type. Concretely, $R(\PP)_S$ is the quotient of the polynomial ring $k[x_s\,|\,s\in S]$ by the ideal generated by the polynomials $P_i(x_{f(1)},\ldots ,x_{f(m_i)})$
for all $i$ and and all injections $f\colon \mathbf{m_i}\into S$.

\begin{xample}[{\bf Nilpotent rings}]
For $d\geq 2$, let $\PP_d=\{x_1^d\}$. The symmetrically presented FI-algebra $R(\PP_d)$ over $\Q$ corresponds to the sequence of nilpotent rings
\[
R(\PP_d)_n\simeq \Q[x_1,\ldots ,x_n]/(x_1^d,\ldots ,x_n^d)
\]
on which $S_n$ acts by permutations.
Since $R(\PP_d)$ is of finite type, for each $i\geq 0$ the sequence of $S_n$-representations $R(\PP_d)_n^i$ is uniformly representation stable.  When $d=2$, this fact was first proved by Ashraf--Azam--Berceneau in \cite[Corollary 5.2]{AAB}.
\end{xample}
\pagebreak

\remove{
\begin{alphxample}[{\bf The Arnol'd algebra}]
\label{example:arnoldalgebra} Braid groups were first studied by Artin and Hurwitz over a century ago. In this section we focus on the \emph{pure braid group} $P_n$, which is the fundamental group of the configuration space of $n$ distinct complex numbers $(z_1,\ldots,z_n)$ with $z_i\in \C$ and $z_i\neq z_j$. This space is aspherical, so its cohomology coincides with $H^*(P_n;\Q)$.
These cohomology algebras $H^*(P_n;\Q)$ fit together into a graded $\FIsharp$-algebra $H^*(P_\bullet;\Q)$; this follows from a general result that we will prove below in Proposition~\ref{pr:CMhTop}. In this section we combine this $\FIsharp$-algebra structure with a presentation of $H^*(P_n;\Q)$ due to Arnol'd~\cite{Ar} to describe more precisely the cohomology of the pure braid groups.

From Arnold's description, a basis for $H^1(P_\bullet;\Q)_S$ is given by the symbols $\{w_{ij} \,|\, i,j\in S\}$ modulo the relations $w_{ij}=w_{ji}$ and $w_{ii}=0$, and a partial injection $f\colon S\supset A\xrightarrow{\phi}B\subset T$ acts by $f_* w_{ij}=w_{\phi(i),\phi(j)}$ when $i,j\in A$ and $f_*w_{ij}=0$ otherwise. We can immediately identify this as
$H^1(P_\bullet;\Q)\simeq \Sym^2 M(1)/M(1)\simeq  M(\y{2})$.

 Arnol'd~\cite{Ar} proved that the cohomology ring $H^*(P_n;\Q)$ is generated by $H^1(P_n;\Q)$. By Theorem~\ref{thm:fisubalgebra}, the graded $\FIsharp$-algebra $H^*(P_\bullet;\Q)$ is of finite type. Since $H^1(P_\bullet;\Q)$ is generated in degree $2$, Corollary~\ref{cor:subalgebraslope} implies that $H^*(P_\bullet;\Q)$ has slope 2 and that $H^i(P_\bullet;\Q)$ is generated in degree $\leq 2i$.
Corollary~\ref{co:fisharprange} thus implies that for each $i\geq 0$, the sequence $\{H^i(P_n;\Q)\}$ of $S_n$-representations is uniformly representation stable with stable range $n\geq 4i$.
This yields a new proof of one of the main theorems (Theorem 4.1) in \cite{CF}.

\para{Computing the character of $H^i(P_n;\Q)$} Consider $H^1(P_\bullet;\Q)$ as a graded FI-module $W$ with $W^1=H^1(P_\bullet;\Q)$ and $W^i=0$ for $i\neq 1$.
Thus $\Gamma[W]$ is a graded-commutative $\FIsharp$-algebra, which as an FI-algebra is of finite type by Theorem~\ref{thm:fisubalgebra}.
Let $I$ be the $\FIsharp$-ideal in $\Gamma[W]$ 
generated by the quadratic relation
\begin{equation}
\label{eq:arnoldideal}
w_{12}w_{23}+w_{23}w_{31}+w_{31}w_{12}\in \Gamma[W]_3.
\end{equation}
In this case, $I$ is also the FI-ideal generated by this relation \eqref{eq:arnoldideal}. It follows from Arnol'd~\cite{Ar}  that $H^\ast(P_\bullet;\Q)$ is isomorphic to the graded $\FIsharp$-algebra $\Gamma[W]/I$.
Thanks to the classification of $\FIsharp$-modules, this reduces the description of $H^i(P_n;\Q)$ to a finite computation for each $i\geq 0$.
For example, we have
\[H^2(P_n;\Q)\simeq \bwedge^2 H^1(P_n;\Q)/I^{2}.\] 

The quadratic part $I^{2}$ is just the FI-module generated by the relation \eqref{eq:arnoldideal}; since $S_3$ acts on this element of $\Gamma[W]_3$ by the sign representation, this is $I^2\simeq M(\y{1,1,1})$.
We computed in \eqref{eq:dlmn} that \[M(\y{2})\otimes M(\y{2})=M(\y{4}) \oplus M\big(\x\y{3,1}\x\big) \oplus M\big(\x\y{2,2}\x\big)\oplus M(\y{3}) \oplus M\big(\x\y{2,1}\x\big)^{\oplus 2} \oplus M\Big(\x\y{1,1,1}\x\Big) \oplus M(\y{2}).\] Exchanging the two factors gives an isomorphism from $M(\y{2})\otimes M(\y{2})$ to itself. As an isomorphism of $\FIsharp$-modules, it necessarily preserves each of the summands in the decomposition above. By examining the details of the computation of \eqref{eq:dlmn} we can check that this involution acts trivially on $M(\y{4})$, $M\big(\x\y{2,2}\x\big)$, $M(\y{3})$, and $M(\y{2})$; it acts by negation on $M\big(\x\y{3,1}\x\big)$ and $M\Big(\x\y{1,1,1}\x\Big)$; and it exchanges the two $M\big(\x\y{2,1}\x\big)$ summands. We conclude that  $\bwedge^2 M(\y{2}) = 
M\Big(\x\y{1,1,1}\x\Big) \oplus M\big(\x\y{2,1}\x\big) \oplus M\big(\x\y{3,1}\x\big)$.
This yields the description of $H^2(P_n;\Q)$ mentioned in the introduction:
\begin{align*}
H^2(P_\bullet;\Q)
&\simeq \bwedge^2 M\big(\y{2}\big) / M\Big(\x\y{1,1,1}\x\Big)\\
&\simeq M\Big(\x\y{1,1,1}\x\Big) \oplus M\big(\x\y{2,1}\x\big) \oplus M\big(\x\y{3,1}\x\big) / M\Big(\x\y{1,1,1}\x\Big)\\
&\simeq M\big(\x\y{2,1}\x\big) \oplus M\big(\x\y{3,1}\x\big)
\end{align*}
which specifies $H^2(P_n;\Q)$ for all $n$ simultaneously.
Based on computer calculations, John Wiltshire-Gordon (personal communication) has formulated a precise conjecture 
for the decomposition of
$H^i(P_n;\Q)$ as a sum of $\FIsharp$-modules $M(\lambda)$ for all $i\geq 0$.
\end{alphxample}
}

\para{Coinvariant algebras}
\label{section:coinvariants}
The \emph{symmetric polynomials} are the $S_n$-invariants in the ring of polynomials $\Q[x_1,\ldots,x_n]$, where $S_n$ acts by permuting the variables.
Let $J_n=\big(\Q[x_1,\ldots,x_n]^{S_n}_{> 0}\big)$ be the ideal generated by symmetric polynomials with vanishing constant term.  The classical \emph{coinvariant algebra} $R(n)$ is the quotient $R(n)\coloneq \Q[x_1,\ldots ,x_n]/J_n$. Chevalley \cite[Theorem B]{Chevalley}  proved that $R(n)$ is isomorphic as 
an $S_n$-representation to the regular representation $\Q[S_n]$.

We cannot analyze $R(n)$ using Theorem~\ref{thm:fisubalgebra}, because the ideals $J_n$ do not together form an FI-ideal.  For example, the inclusion $\n\into \npone$ takes the symmetric polynomial $x_1+\cdots + x_n\in  \Q[x_1,\ldots,x_n]^{S_n}$ to the non-symmetric polynomial $x_1+\cdots+x_n\not\in \Q[x_1,\ldots ,x_n,x_{n+1}]^{S_{n+1}}$. 
Fortunately, we can understand the coinvariant algebra by instead using the natural \emph{co-FI-module} structure on the algebra of polynomials. 

Given a co-FI-module $W$, let $W^{\inv}\subset W$ be the submodule consisting of the invariants $W^{\inv}_T\coloneq (W_T)^{S_T}$. Given $f\colon T\into T'$, for any $\sigma\colon T\into T$ there exists $\sigma'\colon T'\into T'$ such that $f\circ \sigma=\sigma'\circ f$. This shows that $f_*\colon W_{T'}\to W_T$ satisfies $f_*(W^{\inv}_{T'})\subset f_*(W^{\inv}_T)$, so $W^{\inv}$ is indeed a co-FI-module. If $W$ is a graded co-FI-module, let $W_{>0}\subset W$ be the part in positive grading, defined by $W^0_{>0}=0$ and $W^i_{>0}=W^i$ for $i>0$.

\begin{definition}
\label{def:generalcoinvariantalgebra}
Given a graded co-FI-algebra $A$, the \emph{coinvariant algebra} $\coinv(A)$ is the graded co-FI-algebra $\coinv(A)\coloneq A/((A^{\inv}_{>0}))$, which assigns to $T$   the algebra $\coinv(A)_T=A_T/((A_T)^{S_T}_{>0})$.
\end{definition}
\begin{prop}
\label{pr:coinvariantfinite}
Let $A$ be a graded co-FI-algebra over a field of characteristic 0. If $A$ is of finite type, $\coinv(A)$ is of finite type. If $A$ has slope $\leq m$, then $\coinv(A)$ has slope $\leq m$.
\end{prop}

\begin{xample}
Consider the graded co-FI-algebra $R\coloneq \coinv(\Q[M(1)])$, which is of finite type with slope $\leq 1$ by Proposition~\ref{pr:coinvariantfinite}. As $S_n$-representations, $R_n\simeq \Q[x_1,\ldots,x_n]/(\Q[x_1,\ldots,x_n])^{S_n}_{>0}\simeq R(n)$. In other words, we have combined the classical coinvariant algebras $R(n)$ into a single co-FI-algebra $R$. Theorem~\ref{theorem:intro:FIequiv} implies that for each $i\geq 0$, the graded pieces $R^i_n\simeq R(n)^i$ of the coinvariant algebra form a uniformly representation stable sequence of $S_n$-representations. A different proof of this result was given in \cite[Theorem 7.4]{CF}, using work of Stanley, Lusztig, and Kraskiewicz--Weyman.
\end{xample}

We can now prove Theorem~\ref{theorem:intro:polynomial:answer} from the introduction, regarding the characters of multivariate diagonal coinvariant algebras, using  Proposition~\ref{pr:coinvariantfinite}.

\begin{proof}[Proof of Theorem~\ref{theorem:intro:polynomial:answer}]
For fixed $r\geq 1$, let $R^{(r)}$ be the graded co-FI-algebra $R^{(r)}\coloneq \coinv(\Q[M(1)^{\oplus r}])$ from Definition~\ref{def:generalcoinvariantalgebra}, with $M(1)^{\oplus r}$ concentrated in grading 1. To the set $\n$ this associates $R^{(r)}_n=\coinv(\Q[M(1)_n^{\oplus r}])$, which is naturally isomorphic to the diagonal coinvariant algebra 
\[
R^{(r)}(n)\coloneq \Q[x_1^{(1)},\ldots ,x_n^{(1)},\ldots ,x_1^{(r)},\ldots ,x_n^{(r)}]/\big(\Q[x_1^{(1)},\ldots ,x_n^{(1)},\ldots ,x_1^{(r)},\ldots ,x_n^{(r)}]_{>0}^{S_n}\big)
\] defined in the introduction. By Proposition~\ref{pr:Mmuweight} we have $\wt(M(1)^{\oplus r})=1$, so Proposition~\ref{pr:cofisubalgebra} implies that $\Q[M(1)^{\oplus r}]$ is of finite type with slope 1. By Proposition~\ref{pr:coinvariantfinite}, $R^{(r)}$ is also of finite type with slope $1$. In other words, for each $i\geq 0$ the FI-module $(R^{(r),i})^*$ is finitely generated with weight $\leq i$. The multi-grading by $J=(j_1,\ldots,j_r)\in \N^r$ considered there is a refinement of our grading; setting $\abs{J}\coloneq j_1+\cdots+j_r$ we have a decomposition $(R^{(r)})^i=\bigoplus_{\abs{J}=i}R^{(r)}_J$ as co-FI-modules. Therefore for any $J\in \N^r$, the FI-module $(R^{(r)}_J)^*$ is finitely generated, and  $\wt((R^{(r)}_J)^*)\leq J$. By Theorem~\ref{thm:persinomial}, this implies that the characters $\chi_{R_J^{(r)}}=\chi_{(R_J^{(r)})^*}$ are eventually given by a polynomial in $X_1,\ldots,X_{\abs{J}}$ of degree $\leq J$, as claimed.
\end{proof}

\subsection{The Bhargava--Satriano algebra}
\label{section:BS}
In \cite{BS} Bhargava and Satriano develop a notion of \emph{Galois closure} for an arbitrary finite-dimensional algebra $R$ over a field $k$, which reduces to the usual notion when $R$ is a field extension (or even an \'{e}tale algebra).  The algebra $G(R/k)$ is a quotient of $R^{\tensor d}$, where $d = \dim_k R$, and the natural action of $S_d$ on $R^{\tensor d}$ descends to an action of $S_d$ on $G(R/k)$ by ``Galois automorphisms'', just as in the classical case.

In many ways, the most interesting case is the most degenerate: consider the ring \beq
R_n\coloneq k[x_1, \ldots, x_n] / (x_i x_j)_{i,j \in \{1,\dots, n\}}
\eeq
which has dimension $n+1$ over $k$ and which carries an action of $S_n$ by permutation of the variables.  We call the Galois closure $G(R_n/k)$ the \emph{Bhargava--Satriano algebra}. Bhargava--Satriano show  \cite[Theorem 7]{BS} that this ring is ``maximally degenerate'' among rings of the same dimension, in the sense that $\dim G(R_n/k)\geq \dim G(T_n/k)$ for any other algebra $T_n$ with $\dim T_n=\dim R_n$.

The ring $G(R_n/k)$ carries an action of $S_n \times S_{n+1}$, the first factor coming from the action of $S_n$ on $R_n$ by field automorphisms, the second factor coming from the Galois automorphisms.  Bhargava--Satriano compute the decomposition of $G(R_n/\Q)$ as a representation of the Galois group $S_{n+1}$, and use this to show that $\dim G(R_n/\Q) > (n+1)!$ for $n\geq 3$. The decomposition of $G(R_n/\Q)$ as an $S_n \times S_{n+1}$-representation is unknown.

\begin{proposition} Assume that $k$ is a field of characteristic~0.
There is a graded co-FI-algebra $BS$ of finite type satisfying $BS_n\simeq G(R_n/k)$ as graded $S_n$-representations, 
where the $S_n$-action on $G(R_n/k)$ is via the diagonal subgroup $S_n \subset S_n\times S_{n+1}$. In particular, the $S_n$-representations $G(R_n/k)^i$ are uniformly representation stable.
\end{proposition}
\begin{proof}
The Galois closure $G(R_n/k)$ is defined as the quotient of $R_n^{\otimes n+1}$ by certain relations, and
Bhargava--Satriano show \cite[\S11.2]{BS} in this case that the ideal $J_n$ of relations in $R_n^{\otimes n+1}$ is generated by the elements: \[\gamma_n(x_i)\coloneq x_i\otimes 1\otimes\cdots \otimes 1\ +\ 
1\otimes x_i\otimes\cdots \otimes 1\ \ +\ \cdots\ +\ \ 
1\otimes\cdots \otimes 1\otimes x_i\]

Let $R$ be the graded co-FI-algebra $R$ with $R^0\simeq M(0)$, $R^1\simeq M(1)^*$, and $R^i=0$ for $i>1$; the multiplication is uniquely determined, since $R^1\otimes R^1\to R^2=0$. To the set $\n$ this assigns the ring $R_n=k[x_1,\ldots,x_n]/(x_ix_j)$ defined above.

Recall the graded co-FI-algebra $R^{\tensor\bullet}$ from Definition~\ref{def:tensorbullet}, which has $(R^{\tensor\bullet})_n\simeq R_n^{\otimes n}$. 
The action of $S_n$ on $(R\otimes R^{\tensor \bullet})_n$ realizes the diagonal action of $S_n$ on $R_n^{\otimes n+1}$. We  define the co-FI-ideal $J\subset R\otimes R^{\tensor\bullet}$ as follows.
For any finite set $T$ and any $t\in T$, define $\gamma_T(x_t)\in (R\otimes R^{\tensor\bullet})_T$ by 
\[\gamma_T(x_t)\coloneq x_t\otimes (1\otimes \cdots\otimes 1)+ 1\otimes \big(\sum_{t'\in T} 1\otimes \cdots \otimes x_t\otimes \cdots \otimes 1\big).\]
Let $J_T\subset (R\otimes R^{\tensor\bullet})_T$ be the ideal generated by $\gamma_T(x_t)$ for all $t\in T$. For any $f\colon S\into T$, examination shows that the co-FI-module map $f_*\colon (R\otimes R^{\tensor\bullet})_T\to (R\otimes R^{\tensor\bullet})_S$ satisfies $f_*(\gamma_T(x_t))=\gamma_S(x_s)$ if $t=f(s)$, and $f_*(\gamma_T(x_t))=0$ if $t\not\in f(S)$. Therefore the ideals $J_T$ form a co-FI-ideal $J\subset R\otimes R^{\tensor\bullet}$.

We define the graded co-FI-algebra $BS$ to be the quotient $BS\coloneq R\otimes R^{\tensor\bullet}/J$, so that
\[BS_n\simeq R_n^{\otimes n+1}/J_n\simeq G(R_n/k).\]
Since $R$ is of finite type, Propositions~\ref{pr:Vtensorbullet} and \ref{pr:gradedtensor} imply that the graded co-FI-algebra $R\otimes R^{\tensor\bullet}$ is of finite type. Since $\car(k)=0$, its quotient $BS$ is a graded co-FI-algebra of finite type. The representation stability of $(BS^i)^*\simeq BS^i\simeq G(R_n/k)^i$ follows from Theorem~\ref{th:FIequiv}.
\end{proof}

The graded pieces $BS^i$ can be computed for small $i$ by hand.  For example,  the dual of $BS^1$ is isomorphic as an $\FI$-module to
\beq
(BS^1)^* \simeq M(1) \tensor M(1) = M(\y{2}) \oplus M(\y{1,1}) \oplus M(\y{1}).
\eeq
This raises the  question: does the co-FI-algebra $BS$ in fact have the structure of an $\FIsharp$-algebra?

\subsection{Polynomials on rank varieties}
\label{subsection:determinantal}
Let $k$ be a field of arbitrary characteristic.  For $n\geq 1$ let $k[x_{ij}\,|\,i,j\in \n]$ be the coordinates on the  space $\Mat_{n\times n}(k)$ of $n\times n$ matrices over $k$.   
Let $P=(p_1,\ldots ,p_m)$ be an $m$-tuple of polynomials $p_\ell\in k[T]$, and let $r=(r_1,\ldots ,r_m)$ 
be an $m$-tuple of positive integers.  With this data Eisenbud--Saltman~\cite{ES} define the \emph{rank variety} of matrices:
\[X_{P,r}(n)\coloneq \big\{A\in \Mat_{n\times n}(k)\,\big|\, \rank(p_\ell(A))\leq r_\ell, 1\leq \ell\leq m\big\}\]
The special case when $m=1$ and $p_1(T)=T$ is the fundamental example of the \emph{determinantal variety} of  matrices of rank $\leq r_1$. We consider $X_{P,r}(n)$ as an affine variety, defined by the polynomial 
equations 
\[\big\{\det\big(B(p_\ell(A))\big)=0\,\big|\, B\in {\cal B}_\ell,  1\leq \ell\leq m\big\}\]
where ${\cal B}_\ell$ denotes the set of $(r_\ell+1)\times (r_\ell+1)$ minors of the matrix $(x_{ij})$.   Let $J_n$ be the ideal of $k[\{x_{ij}\}]$ generated by the homogeneous polynomials $\det(B(p_\ell(A)))$.  The coordinate ring $\O(X_{P,r}(n))$ 
is isomorphic by definition to $k[\{x_{ij}\}]/J_n$, and inherits a grading from $k[\{x_{ij}\}]$. 

\begin{question}
\label{question:rankvar}
What is the dimension of the space of degree $d\geq 1$ polynomials on $X_{P,r}(n)$; that is, what is the dimension of the degree $d\geq 1$ component of  $\O(X_{P,r}(n))\simeq k[\{x_{ij}\}]/J_n$?
\end{question}

In the special case of determinantal varieties in characteristic~0, an answer to Question~\ref{question:rankvar} can be deduced from work of Lascoux and Weyman (see \cite[Corollary 6.1.5(d)]{Wey}).  While we cannot resolve Question~\ref{question:rankvar} completely, the theory of $\FIsharp$-modules imposes strong constraints on the answer.

\begin{theorem}[{\bf Polynomiality for rank varieties}]
\label{theorem:sections:polygrowth} 
Fix $P$, $r$, and $d$.
Then the dimension of the space $\O(X_{P,r}(n))^d$ of degree-$d$ polynomials on the rank variety $X_{P,r}(n)$ is polynomial in $n$ of degree at most $2d$ for all $n\geq 0$. If $\car k=0$ then the same is true of the characters $\chi_{\O(X_{P,r}(n))^d}$.
\end{theorem}
\begin{proof}
The matrix varieties $\Mat_{n\times n}$ fit together into an $\FIsharp$-scheme $\Mat$; the affine scheme $\Mat_{S\times S}$ has coordinate ring $\O(\Mat_{S\times S})=k[x_{s,s'}\,|\,s,s'\in S]$. Given $f\colon S\supset A\xrightarrow{\phi}\phi(A)\subset T$, the map $f_*\colon \Mat_{S\times S}\to \Mat_{T\times T}$ is given in coordinates by $x_{\phi(s),\phi(s')}\mapsto x_{s,s'}$ and  $x_{t,t'}\mapsto 0$ if $t\not\in \phi(A)$ or $t'\not\in \phi(A)$.
The rank varieties $X_{P,r}(S)\into \Mat_{S\times S}$ are preserved by these maps, and thus form an $\FIsharp$-subscheme $X_{P,r}\into \Mat$.

The coordinate rings $\O(\Mat)$ and $\O(X_{P,r})$ are naturally $\FIsharp^{\op}$-algebras, but via the isomorphism $\FIsharp\simeq \FIsharp^{\op}$ from Remark~\ref{remark:FIsharpremarks} we can consider them as $\FIsharp$-algebras. For example, from the above formula we see that $\O(\Mat)\simeq k[M(1)\otimes M(1)]$ as a graded $\FIsharp$-algebra. Its quotient  $\O(X_{P,r})$ is thus generated as a graded $\FIsharp$-algebra by the image of $M(1)\otimes M(1)$ in $\O(X_{P,r})^1$.

$M(1)\tensor M(1)$ is finitely generated in degree $2$, so in particular $\wt(M(1)\otimes M(1))=2$. Thus Theorem~\ref{thm:fisubalgebra} implies that $\O(X_{P,r})$ is a graded $\FIsharp$-algebra of finite type and slope $\leq 2$.
The theorem now follows from Theorems~\ref{thm:persinomial} and \ref{thm:polyFIsharp}. 
 \end{proof}
\begin{remark} The rank variety $X_{P,r}(n)$ of Theorem~\ref{theorem:sections:polygrowth} may be non-reduced as a scheme.  However, the same theorem holds for the underlying reduced variety. The proof amounts to replacing the ideal $J_n$ by its radical and applying the same argument. 
\end{remark}

Recent work of Draisma and Kuttler~\cite{DK} proves a bounded-generation result for ``border rank varieties'', which generalize the rank varieties discussed here to the case of tensors of rank higher than $2$.  The key ingredient of their theorem is a Noetherian-type property.  It would be interesting to understand whether their work can be phrased in the language of FI-modules. 

\section{Applications: cohomology of configuration spaces}
\label{section:configspaces}

In this section we apply the theory of FI-modules to the cohomology of configuration spaces of distinct points in a manifold $M$.
We obtain in Theorem~\ref{thm:confncharpoly} and Corollary~\ref{co:vakilwood} strengthenings of earlier results of Church \cite[Theorems 1 and 5]{Ch}.
When $M$ is an open manifold, we apply the theory of $\FIsharp$-modules to obtain a number of new theorems on the rational, integral, and mod-$p$ cohomology of the configuration spaces of $M$.

\subsection{FI-spaces and co-FI-spaces}
An \emph{FI-space} $X$ is a functor from $\FI$ to the category $\Top$ of topological spaces. These have been considered elsewhere in the topological literature under the names ``$\mathcal{I}$-spaces'' or ``$\Lambda$-spaces''; see 
Remark~\ref{remark:FIspaces} below.
Since homology is functorial, for any FI-space $X$ the homology $H_*(X;k)$ forms a graded FI-module.
Similarly, a \emph{co-FI-space} $X$ is a functor $X\colon \coFI\to\Top$. Since cohomology is contravariantly functorial, for any co-FI-space $X$ the cohomology $H^*(X;k)$ forms a graded FI-algebra.

\begin{definition}
\label{def:coFIMbullet}
Given a topological space $M$, we define the co-FI-space $M^\bullet$ as follows. For any finite set $S$, we have the topological space $M^S$ of functions from $S$ to $M$. An injection $f\colon T\into S$ induces a projection $f_*\colon M^S\onto M^T$ by restriction, given explicitly by $f_*(\varphi)\coloneq \varphi\circ f\colon T\to S\to M$.
\end{definition}

\begin{proposition}
\label{pr:cohoMbullet}
Let $M$ be a connected topological space, and let $k$ be a field. If $H^i(M;k)$ is finite-dimensional for all $i\geq 0$, the graded FI-algebra $H^*(M^\bullet;k)$ is of finite type.
\end{proposition}
\begin{proof}
The graded FI-algebra $H^*(M^\bullet;k)$ associates to a finite set $S$ the cohomology $H^*(M^S;k)$. 
Since $k$ is a field, the K\"{u}nneth theorem implies that $H^*(M^S;k)\simeq H^*(M;k)^{\otimes S}$. It thus seems that $H^*(M^\bullet;k)$ should be isomorphic to the FI-module $H^*(M;k)^{\otimes \bullet}$ defined in Definition~\ref{def:tensorbullet}. This is almost true; the only difference is that the K\"{u}nneth isomorphism $H^*(M^S;k)\simeq H^*(M;k)^{\otimes S}$ depends on an ordering of $S$, and when these factors are permuted, a sign is sometimes introduced depending on the grading.

We can nevertheless deduce the proposition from Proposition~\ref{pr:Vtensorbullet}, since it is easy to see that the proof of that proposition is not affected by this sign. To apply Proposition~\ref{pr:Vtensorbullet}, we consider the graded $k$-module $H^*(M;k)$ as a ``constant'' graded FI-module $V$. The assumption that $M$ is connected guarantees that $H^0(M;k)\simeq k$, which gives an isomorphism $V^0\simeq M(0)$. The remaining hypothesis of Proposition~\ref{pr:Vtensorbullet} is that $V$ is of finite type, or in other words that $H^i(M;k)$ is finite-dimensional for each $i\geq 0$. We conclude that $H^*(M^\bullet;k)$ is of finite type, as claimed.
\end{proof}

\begin{remark}
\label{remark:FIspaces}
Let $M$ be a connected well-based topological space (i.e.\ with specified basepoint $\ast\in M$). We can extend the co-FI-space structure of Definition~\ref{def:coFIMbullet} to consider $M^\bullet$ as an $\FIsharp$-space: given $f\colon S\supset A\xrightarrow{\phi}B\subset T$ and $\varphi\colon S\to M$, we define $f_*\colon M^S\to M^T$ by
\begin{equation}
\label{eq:MbulletFIsharp}
f_*(\varphi)(t)=\begin{cases}\varphi(\phi^{-1}(t))&\text{ if }t\in B\\\ast&\text{ if }t\in T-B\end{cases}
\end{equation}
\pagebreak

Considering $M^\bullet$ as an FI-space, given a map defined on $S$, an inclusion $f\colon S\into T$ simply extends it to $T$ by sending $T-f(S)$ to the basepoint $\ast$. We point out that the direct limit $\colim_{\FI} M^\bullet$ computes the infinite symmetric product $\Sym^\infty M\simeq \colim_{\FI} M^\bullet$, which  by the Dold--Thom theorem satisfies $\pi_i(\Sym^\infty M)\simeq \widetilde{H}_i(M;\Z)$.

In fact, $M^\bullet$ provides an example of an \emph{$\mathcal{I}$-monoid} in the sense of Sagave--Schlichtkrull~\cite{SagaveSch} (a commutative monoid in the category of FI-spaces). Sagave--Schlichtkrull prove that any infinite loop space can be modeled by an $\mathcal{I}$-monoid. For example, the Barratt--Priddy--Quillen theorem implies that $\Omega^\infty\Sigma^\infty M\simeq \hocolim_{\FI} M^\bullet$. Similarly, Quillen's $K$-theory space $K(R)$ (or rather the identity component) is realized by $\hocolim_{\FI} \BGL_\bullet(R)$ for the natural FI-space $\BGL_\bullet(R)$. We will not exploit this connection further in this paper; see \cite[Examples 1.3 and 1.5]{SagaveSch} for these examples and other applications of $\mathcal{I}$-monoids.
\end{remark}

\subsection{Configuration spaces as co-FI-spaces} 
\label{section:fgconfig}
We define the co-FI-space $\Conf(M)$ as a subspace of $M^\bullet$: let $\Conf_S(M)\subset M^S$ be the subspace of \emph{embeddings} $S\into M$:
\[\Conf_S(M)\coloneq \Emb(S,M)\quad\subset \quad M^S\coloneq \Map(S,M)\]
If $\varphi\colon T\into M$ and $f\colon S\into T$ are injective, then $f_*(\varphi)=\varphi\circ f\colon S\into T$ is injective. Thus $\Conf(M)$ is indeed a co-FI-space, and the inclusion $\Conf(M)\into M^\bullet$ is a map of co-FI-spaces.
Of course, the space $\Conf_n(M)=\Emb(\n,M)$ can be identified with the configuration space of  ordered $n$-tuples of distinct points on $M$:
\[\Conf_n(M)=\big\{(x_1,\ldots,x_n)\in M^n\ \big|\  x_i\neq x_j\big\}\]
The following theorem is proved (in different language) in \cite{Ch}. The assumption $\dim M\geq 2$ guarantees that $\Conf_S(M)$ is connected; the condition ${\dim_\Q(H^*(M;\Q))<\infty}$ is satisfied whenever $M$ is compact.
\begin{thm}[{\bf Rational cohomology of configuration spaces}]  
\label{theorem:coho:config2}
Let $M$ be a connected, oriented manifold of dimension at least 2 with $\dim_\Q(H^*(M;\Q))<\infty$.  Then the graded FI-algebra  $H^*(\Conf(M);\Q)$ is of finite type.
\end{thm}
\begin{proof}
We will consider the Leray spectral sequence of the inclusion of co-FI-spaces $\Conf(M)\into M^\bullet$. Recall that for any continuous map of topological spaces $f\colon X\to Y$, the Leray spectral sequence with $\Q$ coefficients is a cohomological spectral sequence converging to $H^*(X;\Q)$ with $E_2$ page $E_2^{p,q}\simeq H^p(Y;R^q f_*\underline{\Q})$. These entries are in general difficult to compute, with the exception of the edge terms $E_2^{p,0}\simeq H^p(Y;\Q)$. Recall further that the Leray spectral sequence is functorial: if $f'\colon X'\to Y'$ is another continuous map, and $g_X,g_Y$ are maps $g_X\colon X'\to X$ and $g_Y\colon Y'\to Y$ such that $f\circ g_X=g_Y\circ f'$, then $g$ induces a map of spectral sequences from the Leray spectral sequence for $f\colon X\to Y$ to the Leray spectral sequence for $f'\colon X\to Y$.

For any finite set $S$ we have a continuous map $i_S\colon \Conf_S(M)\into M^S$, so we have a Leray spectral sequence with $E_2^{p,0}\simeq H^p(M^S;\Q)$ converging to $H^*(\Conf_S(M);\Q)$. Moreover for any $f\colon T\into S$ we have $i_S\circ f_*=f_*\circ i_T$ (this is what it means for $i\colon \Conf(M)\into M^\bullet$ to be a map of co-FI-spaces), so we obtain maps $f_*$ between these spectral sequences by functoriality. In other words, all the Leray spectral sequences for $\Conf_S(M)\into M^S$ together form a cohomological spectral sequence of FI-modules with edge terms $E_2^{p,0}\simeq H^p(M^\bullet;\Q)$ converging to the graded FI-algebra $H^*(\Conf(M);\Q)$.

Each page $E_r$ forms a bigraded FI-algebra, which we can think of as a graded FI-algebra by $(E_r)^i\simeq \bigoplus_{0\leq p\leq i}E_r^{p,i-p}$. Since this sum is finite, there is no ambiguity in asking whether $E_r$ is of finite type: this means that $E_r^{p,q}$ is a finitely generated FI-module for all $p\geq 0$ and all $q\geq 0$. 

Totaro~\cite{To} provides an explicit description of the $E_2$ page, which in particular gives an isomorphism $E_2^{0,d-1}\simeq H^{d-1}(\Conf(\R^d);\Q)$. This FI-module can be identified with $M(\y{2})$ when $d$ is even, \remove{as we saw in Example~\ref{example:arnoldalgebra} for $d=2$, } and with $M(\y{1,1})$ when $d$ is odd; in particular, the FI-module $E_2^{0,d-1}$ is finitely generated. Moreover, Totaro shows that the graded FI-algebra $E_2$  is generated by $E_2^{0,d-1}$ together with $E_2^{p,0}$ for all $p\geq 0$. (For future reference, this implies that $E_r^{p,\ast}=0$ unless $\ast=q(d-1)$.)

$E_2^{*,0}\simeq H^*(M^\bullet;\Q)$ is of finite type by Proposition~\ref{pr:cohoMbullet}, and $E_2^{0,d-1}$ is finitely generated, so Totaro's generating set is a graded FI-module of finite type. Theorem~\ref{thm:fisubalgebra} then implies that the graded FI-algebra $E_2$ is of finite type. Since $E_\infty$ is a subquotient of $E_2$, Theorem~\ref{th:noetherian} implies that $E_\infty$ is of finite type. The FI-module $H^i(\Conf(M);\Q)$ has a finite filtration with graded quotients $E_\infty^{p,i-p}$, so $H^*(\Conf(M);\Q)$ is of finite type if and only if $E_\infty$ is; this concludes the proof.
\end{proof}

Combined with Theorem~\ref{th:FIequiv}, Theorem~\ref{theorem:coho:config2} implies the theorem of Church~\cite[Theorem 1]{Ch} that the sequence of $S_n$-representations $\{H^i(\Conf_n(M);\Q)\}$ is uniformly representation stable.  By applying this result to the trivial representation, Church extended the classical theorem on rational homological stability for the configuration spaces of \emph{unordered} points  in an open manifold to configuration spaces of unordered points in an arbitrary manifold.

\begin{remark}
The key technical innovation in \cite{Ch} was the notion of \emph{monotonicity}, which allows representation stability to be transmitted from the initial term of a spectral sequence to its $E_\infty$ term.  We showed in the proof of Theorem~\ref{th:FIequiv} that, for FI-modules, finite generation implies monotonicity.  Thus it is not surprising that we are able to imitate the argument of \cite{Ch} in the present paper.
\end{remark}

\subsection{Bounding the stability degree for cohomology of configuration spaces}
The results of Church~\cite{Ch} not only show that $H^i(\Conf_n(M);\Q)$ is representation stable, but they provide an explicit stable range. In this section we strengthen those results by computing explicit bounds on the stability degree for the FI-module $H^i(\Conf(M);\Q)$, which let us improve Church's stable range.

Recall from Definition~\ref{def:injectivitydegree} that we divided the notion of stability degree into \emph{injectivity degree} and \emph{surjectivity degree}. For succinctness, for an FI-module $V$ we write $V\stabT (A,B)$ if $\injD(V)\leq A$ and $\surjD(V)\leq B$; this implies  $\stabD(V)\leq \max(A,B)$.

\begin{thm}
\label{th:config}
Let $M$ be a connected, oriented manifold of dimension $d \geq 3$.   For any $i\geq 0$, the FI-module $H^i(\Conf(M);\Q)$ has weight $\leq i$ and $H^i(\Conf(M);\Q)\stabT(i+2-d,i)$; in particular $\stabD(H^i(\Conf(M);\Q))\leq i$. 
\end{thm}
We first record two elementary homological facts that we will use in the proof, which follow from the fact that $\Phi_a$ is exact over rings containing $\Q$. Both can be verified  by a simple diagram chase.
\begin{lemma}
\label{lemma:homologyfacts}
Assume that $k$ contains $\Q$. 
\begin{enumerate}[(i)]
\item Consider a complex of FI-modules $X\xrightarrow{f} Y\xrightarrow{g} Z$ with $g\circ f=0$. If $X\stabT(A,B)$, $Y\stabT(C,D)$, and $Z\stabT(E,F)$, then the homology $\ker g/\im f$ satisfies $\ker g/\im f\stabT\big(\max(B,C),\ \max(D,E)\big)$.
\item If $V$ is an FI-module with a finite filtration by sub-FI-modules $F^iV\subset V$, then $V\stabT(A,B)$ if and only if $F^iV/F^{i-1}V\stabT(A,B)$ for all $i$.
\end{enumerate}
\end{lemma}

\begin{proof}[Proof of Theorem~\ref{th:config}]
We continue with the notation from the proof of Theorem~\ref{theorem:coho:config2}. 
Let $D=d-1$.
We saw in that proof that $E_2^{p,\ast}=0$ unless $\ast=qD$.
Moreover, Church~\cite[\S3.3]{Ch} gives an explicit description of $E_2^{p,qD}$ as a direct sum $E_2^{p,qD}\simeq \bigoplus_{k=0}^{p+2q}M(W_k)$, where $W_k$ is a certain representation of $S_k$. By Proposition~\ref{pr:stabdegprojective}, this implies that $E_2^{p,qD}\stabT(0,p+2q)$ and $\wt(E_2^{p,qD})\leq p+2q$.

The argument proceeds by induction on the pages of the Leray spectral sequence. To be precise, the fact that  $E_2^{p,\ast}=0$ unless $\ast=qD$ implies that the only nontrivial differentials occur on pages $E_{kD+1}^{p,qD}$, with $E_{(k-1)D+2}\simeq \cdots\simeq E_{kD}\simeq E_{kD+1}$. We will prove the following claims $(a_k)$ and $(b_k)$ for all $k \geq 2$:
\begin{itemize}
\item[($a_k$)] $\injD(E_{kD+1}^{p,qD})\leq p+2q+1-D$ for all $p\geq 0$ and all $q\geq 0$
\item[($b_k$)] $\surjD(E_{kD+1}^{p,qD})\leq p+2q +\min(k-2,q-1)(D-2)$ for all $p\geq 0$ and all $q > 0$. 
\end{itemize}

We begin with the base case $k=2$. We have isomorphisms $E_2\simeq E_{D+1}$ and $E_{D+2}\simeq E_{2D+1}$. The only nontrivial intervening differential is on the $E_{D+1}$ page, where $E_{2D+1}$ is computed as the cohomology of the complex
\beq
E_{D+1}^{p-D-1,(q+1)D} \ra E_{D+1}^{p,qD} \ra E_{D+1}^{p+D+1,(q-1)D}.
\eeq
From $E_2\simeq E_{D+1}$ we know that $E_{D+1}^{p,qD}\stabT(0,p+2q)$, so 
\beq
E_{D+1}^{p-D-1,(q+1)D} \stabT (0, p + 2q + 1 - D),\quad E_{D+1}^{p,qD}\stabT(0,p+2q),\quad E_{D+1}^{p+D+1,(q-1)D}\stabT(0,p+2q+ D - 1).
\eeq
Applying Lemma~\ref{lemma:homologyfacts}(i) shows that $E_{2D+1}^{pq}\stabT(p+2q+1-D, p+2q)$, verifying $(a_2)$ and $(b_2)$.

Suppose that $(a_k)$ and $(b_k)$ hold.
We can write $E_{(k+1)D+1}^{p,qD}$ as the cohomology of a three-term complex:
\begin{equation}
\label{eq:kcomplex}
E_{kD+1}^{p-kD-1,(q+k)D} \ra E_{kD+1}^{p,qD} \ra E_{kD+1}^{p+kD+1,(q-k)D} 
\end{equation}
For readability, let $X\to Y\to Z$ denote the FI-modules making up this complex \eqref{eq:kcomplex}.
By $(a_k)$ we have $\injD(Y
)\leq p+2q+1-D$, and by $(b_k)$ we have
\beq
\surjD(X
)\leq p-kD -1 + 2(q+k) + (k-2)(D-2) = p+2q + 3 -2D\leq p+2q+1-D,
\eeq
where the last inequality holds because $D\geq 2$. We conclude that $\injD(E_{(k+1)D+1}^{p,qD})\leq p+2q+1-D$ by Lemma~\ref{lemma:homologyfacts}(i), verifying $(a_{k+1})$. 

Assume that $k>q>0$. In this case $Z=0$, since $E_{kD+1}^{p+kD+1,(q-k)D}$ lies outside the first quadrant, so $\surjD(E_{(k+1)D+1}^{p,qD})\leq \surjD(Y)$. By $(b_k$) we have $\surjD(Y)\leq p+2q +\min(k-2,q-1)(D-2)\leq p+2q +\min(k+1-2,q-1)(D-2)$, verifying $(b_{k+1})$ when $k>q>0$.

Finally, assume that $k\leq q$. By $(a_k)$, $\injD(Z)\leq 
p + kD + 1 + 2(q-k) + 1 - D = p+2q+ (k-1)(D-2)$.
 Similarly, $(b_k)$ implies $\surjD(Y)\leq p+2q +\min(k-2,q-1)(D-2)=p+2q+(k-2)(D-2)$. 
 By  Lemma~\ref{lemma:homologyfacts}(i), $\surjD(E_{(k+1)D+1}^{p,qD})\leq p+2q+(k-1)(D-2)$, verifying $(b_{k+1})$ when $k\leq q$.

We now deduce the theorem from the claims $(a_k)$ and $(b_k)$. The FI-module $H^i(\Conf(M);\Q)$ has a finite filtration with graded quotients those $E_\infty^{p,qD}$ with $p+qD=i$. Thus by Lemma~\ref{lemma:homologyfacts}(ii), it suffices to show that $E_\infty^{p,qD}\stabT(i+2-d,i)$ and $\wt(E_\infty^{p,qD})\leq i$ whenever $p+Dq=i$.

Since $E_\infty^{p,qD}$ is a subquotient of $E_2^{p,qD}$, we automatically have $\wt(E_\infty^{p,qD})\leq p+2q\leq i$.
As $k\to \infty$, the claims $(a_k)$ imply $\injD(E_{\infty}^{p,qD})\leq p+2q+1-D\leq i+2-d$. Similarly, when $q>0$ the claims $(b_k)$ imply $\surjD(E_\infty^{p,qD})\leq p+Dq + 2 - D\leq i$, so in this case we have $E_\infty^{p,qD}\stabT(i+2-d,i)$ as desired. When $q=0$, the FI-module $E_{\infty}^{p,0}$ is a quotient of $E_2^{p,0}$, and we showed at the beginning of the proof that $\surjD(E_2^{p,0})\leq p$. We therefore have $\surjD(E_\infty^{p,0})\leq p=i$, with no need for induction, verifying that $E_\infty^{p,0}\stabT(i+2-d,i)$ as well.
\end{proof}

Combining Theorem~\ref{th:config}  with Theorem~\ref{thm:persinomial} proves Theorem~\ref{thm:confncharpoly} from the introduction.

\begin{remark}
If $M$ is a connected oriented manifold of dimension 2 with $\dim_\Q(H^*(M;\Q))<\infty$, the weight of the FI-module $H^i(\Conf M;\Q)$ is bounded by that of $\bigoplus_i E^{p,i-p}_2$, just as in the proof of Theorem~\ref{th:config}.  This weight is $p + 2(i-p)$, which is maximized when $p=0$.  So in this case $H^i(\Conf M;\Q)$ is a finitely generated FI-module of weight at most $2i$, whence the characters and Betti numbers of $H^i(\Conf_n(M);\Q)$ are eventually polynomial of degree $\leq 2i$.
\end{remark}

\para{A ``classical'' application}  We can apply Theorem~\ref{th:config} to prove a cohomological stability result, in the usual sense, for unordered configuration spaces with some population of colored points. This result is inspired by Vakil--Wood~\cite{VakilWood}, who proved a motivic analogue when $M$ is an algebraic variety.

Given a partition $\mu=(\mu_1,\ldots,\mu_k)$, we denote by $B_{n,\mu}(M)$ the configuration space of sets of $n$ distinct  unordered points on $M$, with $\mu_i$ of the points labeled with color $i$ and $n-\abs{\mu}$ of the points left uncolored. 
The following corollary is a direct improvement on \cite[Theorem 5]{Ch}, where the bound $n\geq \max(2i,2\abs{\mu})$ was obtained. It is the bound on stability degree in Theorem~\ref{th:config} that allows us to improve the stable range from $\max(2i,2\abs{\mu})$ to $i+\abs{\mu}$.
\begin{corollary} 
\label{co:vakilwood}
 Let $M$ be a connected, oriented manifold of dimension $\geq 3$.  Then  for all $n \geq i + \abs{\mu}$,
\[
H^i(B_{n,\mu}(M);\Q) \simeq H^i(B_{n+1,\mu}(M);\Q).
\]
\end{corollary}

\begin{proof}
The invariant subspace $H^i(\Conf_n(M);\Q)^{S_{n-\abs{\mu}}}$ carries an action of $S_{\abs{\mu}}$, and 
the dimension of $H^i(B_{n,\mu}(M);\Q)$ is the dimension of the space fixed by the action of
\beq
S_\mu \coloneq  S_{\mu_1} \times \ldots \times S_{\mu_k} \subset S_{\abs{\mu}}.
\eeq

The statement that $H^i(\Conf_n(M);\Q)$ has stability degree $\leq i$ means by definition that the isomorphism class of  $H^i(\Conf_n(M);\Q)^{S_{n-\abs{\mu}}}\simeq H^i(\Conf_n(M);\Q)_{S_{n-\abs{\mu}}}$ as an $S_{\abs{\mu}}$-representation is constant for $n\geq i+\abs{\mu}$. The dimension of the $S_{\abs{\mu}}$-invariant subspace is thus constant for $n\geq i+\abs{\mu}$ as well.
\end{proof}

\subsection{Configurations on manifolds with boundary as homotopy \texorpdfstring{$\FIsharp$-space}{FI\#-space}}
\label{ss:manifoldswithboundary}

We saw in Remark~\ref{remark:FIspaces} that $M^\bullet$ could be extended from a co-FI-space to an $\FIsharp$-space. This is not possible for $\Conf(M)$. However, when $M$ is the interior of a compact manifold with nonempty boundary, we can in fact realize the configuration space $\Conf(M)$ as an $\FIsharp$-space \emph{up to homotopy}. 

\begin{definition}
An \emph{homotopy $\FIsharp$-space} is a functor $X$ from $\FIsharp$ to $\hTop$, the category of topological spaces and \emph{homotopy classes} of continuous maps. Concretely, this means that for each $n$ we have a space $X_n$, for each partial injection $f\in \Hom_{\FIsharp}(\m,\n)$ we have a map $f_*\colon X_m\to X_n$ (or rather a homotopy class of maps), and the corresponding diagrams all commute up to homotopy.  Since a homotopy class of maps induces a well-defined map on homology and cohomology, both the homology groups $H_i(X;k)$ and cohomology groups $H^i(X;k)$ form an $\FIsharp$-module.
\end{definition}

\begin{prop}
\label{pr:CMhTop}
Let $M$ be the interior of a connected, oriented, compact manifold $\overline{M}$ of dimension $d\geq 2$ with nonempty boundary $\partial\overline{M}$.  We can extend the FI-space $\Conf(M)\colon \coFI\to \Top$ to a \emph{homotopy $\FIsharp$-space} $\Conf(M)\colon \FIsharp\to \hTop$.
\end{prop}\pagebreak

\begin{proof} We will use a variant of the $\FIsharp$-space structure on $M^\bullet$ defined in \eqref{eq:MbulletFIsharp}, modified so as to preserve the injectivity of our configurations $S\into M$.
The fact that $\partial\overline{M}\neq\emptyset$ lets us define, for any inclusion of finite sets $B\subseteq T$, a map (defined up to homotopy)
\[\Psi_B^T\colon \Conf_B(M)\to\Conf_T(M)\]
which ``adds points at infinity'', as follows.  Fix a collar neighborhood $R$ of one component of  $\partial\overline{M}$ (so that $R$ is connected), and fix a homeomorphism $\Phi\colon M\xrightarrow{\simeq} M\setminus \overline{R}$ isotopic to the identity $\id\colon M\to M$.

If $T=B$ we set $\Psi_B^T=\id$. Otherwise, fix a configuration $q_B^T\colon (T-B)\into R$ in $\Conf_{T-B}(R)$. Then any embedding $\varphi\colon B\into M$ in $\Conf_B(M)$ can be extended to a function $\Psi_B^T(\varphi)\colon T\into M$ by defining:
\[\Psi_B^T(\varphi)(t)=\begin{cases} \Phi(\varphi(t))&t\in B\\q_B^T(t)&t\in T- B\end{cases}\] Since $\Conf_{T-B}(R)$ is connected, different choices of $q_B^T\in \Conf_{T-B}(R)$ induce homotopic maps, so  the map $\Psi_B^T$ is well-defined up to homotopy.

For a morphism $f\colon S\supset \phi^{-1}(B)\xrightarrow{\phi}B\subset T$, let $f_*\colon \Conf_S(M)\to \Conf_T(M)$ be the composition 
\[f_*\colon \Conf_S(M)\xrightarrow{-\circ\phi^{-1}}\Conf_B(M)\xrightarrow{\Psi_B^T}\Conf_T(M).\]
Let $g\colon T\supset C\xrightarrow{\psi} \psi(C)\subset U$ be another morphism, and let us compare the two maps $g_*\circ f_*$ and $(g\circ f)_*\colon \Conf_S(M)\to \Conf_U(M)$. The restriction $(g\circ f)_*(\varphi)|_{U-\psi(B\cap C)}$ is equal to $q_B^T$ for all $\varphi\in \Conf_S(M)$; the restriction $g_*\circ f_*(\varphi)|_{U-\psi(B\cap C)}$ is similarly constant (though not the same). On the remaining elements $\psi(B\cap C)$, if $u=\psi(\phi^{-1}(s))$ we have $(g\circ f)_*(\varphi)(u)=\Phi(\varphi(s))$ and $g_*\circ f_*(\varphi)(u)= \Phi(\Phi(\varphi(s)))$. Using the isotopy from $\Phi$ to the identity, we can thus construct a homotopy from $g_*\circ f_*$ to $(g\circ f)_*$. This completes the proof that $\Conf(M)$ is a homotopy $\FIsharp$-space.
\end{proof}

It follows from Proposition~\ref{pr:CMhTop} that the FI-module $H^*(\Conf(M);k)$ is in fact an $\FIsharp$-module in this case, for any ring $k$. In particular, all of the maps  $H^*(\Conf_n(M);k)\to H^*(\Conf_{n+m}(M);k)$ are injective since every inclusion $f\colon S\into T$ has a left inverse in $\FIsharp$; similarly, all of the maps $H^*(\Conf_{n+m}(M);k)\to H^*(\Conf_n(M);k)$ are surjective.

\begin{theorem}
\label{thm:ConfMFIsharpfg}
Let $M$ be a connected, oriented manifold of dimension $d\geq 2$ which is the interior of a compact manifold with nonempty boundary.  Let 
$k$ be a Noetherian ring of finite Krull dimension. For each $i\geq 0$, the $\FIsharp$-module $H^i(\Conf(M);k)$ is finitely generated in dimension $\leq i$ if $d\geq 3$, and in dimension $\leq 2i$ if $d=2$.
\end{theorem}
\begin{proof}
We have already proved the theorem in the case $k=\Q$; the content lies in the case when $k$ has positive characteristic, or is not a field.
We recall from  the proof of Theorem~\ref{th:config} that each entry $E_2^{pq}$  of the Leray spectral sequence $\Conf(M) \into M^\bullet$ is a finitely generated $\FIsharp$-module. 
However, we cannot immediately conclude that $E_\infty^{pq}$ or $H^i(\Conf(M);k)$ is finitely generated, even though this spectral sequence converges to $H^*(\Conf(M);k)$. The problem is that $\Conf(M)$ is only a \emph{homotopy} $\FIsharp$-space, and the Leray spectral sequence is not homotopy invariant. As a result it is not a spectral sequence of $\FIsharp$-modules.

Our hypotheses on $k$ will allow us to circumvent this issue. We first explain the proof in the cases when $k$ is either $\Z$ or a field, which were mentioned in the introduction, and afterwards explain how to get the general case. 
First assume that $k$ is a field. From the proof of Theorem~\ref{th:config}, each entry $E_2^{p,q(d-1)}$ of the Leray spectral sequence is a finitely generated $\FIsharp$-module generated in degree $\leq p+2q$. When $d\geq 3$, it follows that $\dim_k (E_2^{pq})_n=O(n^{p+q})$. As a vector space $(E_\infty^{p,q})_n$ is a subquotient of $(E_2^{pq})_n$, so $\dim_k (E_\infty^{pq})_n=O(n^{p+q})$ as well. Finally, $H^i(\Conf_n(M);k)$ has a finite filtration with quotients $(E_\infty^{p,i-p})_n$, so we conclude that $\dim_k H^i(\Conf_n(M);k) = O(n^i)$. By Theorem~\ref{thm:polyFIsharp}, this implies that the $\FIsharp$-module $H^i(\Conf(M);k)$ is finitely generated in degree $\leq i$, as the theorem claims. When $d=2$ we have $\dim_k(E_2^{pq})_n=O(n^{p+2q})$, so $\dim_k H^i(\Conf_n(M);k) = O(n^{2i})$.

The same argument works when $k=\Z$, using the fact that the number of generators of a $\Z$-module decreases when passing to submodules. Of course, this fact is not true for a general ring; for example, $\C[x,y]$ has rank 1 as a module over itself, but the ideal $(x,y)$ cannot be generated by fewer than 2 elements. However, Forster \cite[Satz~1]{Forster} proved the following theorem for a Noetherian ring $k$ of Krull dimension $d$: if a $k$-module $A$ satisfies 
 \[\qquad\quad\dim_\F(A\otimes_k \F)\leq N\qquad\quad\text{for all quotient fields } \F\coloneq k/\mathfrak{m},\]
then $A$ is generated by at most $N+d$ elements.

We want to apply Forster's theorem to $H^i(\Conf_n(M);k)$. The discussion above shows for $d\geq 3$ that  $\dim_\F(E_2^{pq})_n\otimes_k \F = O(n^{p+q})$. Since $\F$ is a field, this bound passes to subquotients, so $\dim_\F H^i(\Conf_n(M);k)\otimes \F=O(n^{i})$ for all quotient fields $\F$. (Despite the use of big-$O$ notation, we in fact have an explicit upper bound on this dimension which is independent of $\F$, coming from the number of generators for $(E_2^{pq})_n$ in the case $k=\Z$.) Applying Forster's theorem, we conclude that the $k$-module $H^i(\Conf_n(M);k)$ is generated by $O(n^{i}+d)=O(n^{i})$ elements. The  equivalence of (iv) and (i) in Theorem~\ref{thm:polyFIsharp} then implies that the $\FIsharp$-module $H^i(\Conf(M);k)$ is finitely generated in degree $\leq i$, as claimed. The proof for $d=2$ is identical.
\end{proof}

Theorem~\ref{thm:intro:openbetti} follows immediately from Theorem~\ref{thm:ConfMFIsharpfg} by applying the classification of $\FIsharp$-modules from Theorem~\ref{th:charFIsharp}, just as in the proof of Theorem~\ref{thm:polyFIsharp}.

\section{Applications:  cohomology of moduli spaces}
\label{section:moduli}

\subsection{Stability for \texorpdfstring{$\cM_{g,n}$}{M(g,n)} and its tautological ring}
\label{subsection:mgn}
Let $S_g$ be a closed, oriented surface of genus $g\geq 2$.  For each $n\geq 1$ let $(y_1,\ldots ,y_n)$ be an ordered $n$-tuple of $n$ distinct points on $S_g$.  Let $\cM_{g,n}$ denote the moduli space of $n$-pointed Riemann surfaces $(X;y_1,\ldots ,y_n)$ 
homeomorphic to $(S_g;y_1,\ldots ,y_n)$.   The space $\cM_{g,n}$ has the structure of a complex orbifold, i.e.\ the quotient of a complex manifold by a finite group of holomorphic automorphisms.  

The map ``forget the $n^{\rm th}$ point'' yields a fibration 
$p_{n}\colon \cM_{g,n}\to \cM_{g,n-1}$ whose fiber is an $(n-1)$-punctured surface of genus $g$.  For each $k \geq 0$, the map $p_n$ induces a homomorphism 
\[p_n^\ast\colon H^k(\cM_{g,n-1};\Q)\to H^k(\cM_{g,n};\Q).\]

For each $i=1,\ldots,n$, we let $L_i\to \cM_{g,n}$ be the complex line bundle over $\cM_{g,n}$ 
whose fiber at $(X;y_1,\ldots ,y_n)$ is the cotangent space to $X$ at $y_i$.  
Let $\psi_i=c_1(L_i)\in H^2(\cM_{g,n};\Q)$ be the first Chern class of the line bundle $L_i$.  
Integration along the fiber yields a Gysin homomorphism 
\[(p_n)_!:H^j(\cM_{g,n};\Q)\to H^{j-2}(\cM_{g,n-1};\Q).\]
For each $j\geq 0$ define $\kappa_j\coloneq (p_{n+1})_!(\psi_{n+1}^{j+1})\in H^{2j}(\cM_{g,n};\Q)$.

The \emph{tautological ring} $\cR(\cM_{g,n})$ is defined to be the subring of $H^*(\cM_{g,n};\Q)$ generated by 
\[\{\,\psi_i\,|\,1\leq i\leq n\}\ \cup\ \{\,\kappa_j\,|\,j\geq 0\}.\]
This ring has been intensively studied by algebraic geometers 
(see e.g.\ \cite{Va}).
The usual grading on $\mathcal{R}(\mathcal{M}_{g,n})$ is half the cohomological grading, so that $\psi_i$ has grading 1 and $\kappa_j$ has grading $j$.

$S_n$ acts on $\cM_{g,n}$ by permuting the marked points. The induced action on $H^*(\cM_{g,n};\Q)$ satisfies
\[ \sigma\cdot \kappa_j=\kappa_j \ \ \text{and}\ \ \ \sigma\cdot\psi_i=\psi_{\sigma(i)}\]
for each $j\geq 1$ and each $1\leq i\leq n$, so $\cR(\cM_{g,n})$ is a subrepresentation of $H^*(\cM_{g,n};\Q)$. 
The grading-$j$ components $\cR^j(\cM_{g,n})$ can 
be quite complicated as $S_n$-representations, and for $g\geq 3$ they are poorly understood.
However, we have the following strong constraint once $n$ is sufficiently large.

\begin{theorem}
For each $g\geq 2$ the tautological ring $\cR(\cM_{g,\bullet})$ is a graded FI-algebra of finite type.  Thus for each $j\geq 1$ the characters $\chi_{\cR^j(\cM_{g,n})}$ are eventually polynomial of degree $\leq j$.
In particular, $\dim\cR^j(\cM_{g,n})$ is eventually polynomial in $n$ of degree $\leq j$.
\end{theorem}

\begin{proof}  
Let $\cM_{g,S}$ be the moduli space of genus $g$ Riemann surfaces $X$ endowed with an injection $S\into X$. These spaces form a co-FI-space $\cM_{g,\bullet}$  just as in \S\ref{section:fgconfig}, so the  cohomology $H^*(\cM_{g,\bullet};\Q)$ is a graded FI-algebra.

For each $s\in S$ we have a line bundle $L_s\to \cM_{g,S}$ with first Chern class $\psi_s\coloneq c_1(L_i)\in H^2(\cM_{g,S};\Q)$. Similarly, we consider the map $p_\star\colon \cM_{g,S\disjoint\{\star\}}\to \cM_{g,S}$, and for each $j\geq 0$ we define $\kappa_j\coloneq (p_\star)_!(\psi_\star^{j+1})\in H^{2j}(\cM_{g,S};\Q)$. These classes define a map of graded FI-modules $t\colon V\to H^*(\cM_{g,\bullet};\Q)$, where $V$ is the graded FI-module with $V^{2i+1}=0$ and \[V^2\simeq M(1)\oplus M(0), \quad V^{2i} = \Q \kappa_i  \simeq M(0) \text{ for }i>1.\]
By definition, $\cR(\cM_{g,\bullet})$ is the sub-FI-algebra of $H^*(\cM_{g,\bullet};\Q)$ generated by the image $t(V)$. By construction, $V$ is a graded FI-module of finite type with slope $\frac{1}{2}$, so the same is true of $t(V)$.   Theorem~\ref{thm:fisubalgebra} implies that the sub-FI-algebra of  $H^*(\cM_{g,\bullet};\Q)$ generated by $t(V)$ is a graded FI-module of finite type which has slope $\leq \frac{1}{2}$. Recalling the difference in the grading, we conclude that the FI-module $\cR^j(\cM_{g,\bullet})\subset H^{2j}(\cM_{g,\bullet};\Q)$ is finitely generated in degree $\leq i$. The desired conclusion follows from Theorem~\ref{thm:persinomial}.
\end{proof}

\begin{remark}
Jimenez Rolland~\cite{J1} proved the related theorem that for each fixed $g\geq 2$ and $i\geq 0$, the sequence $\{H^i(\cM_{g,n};\Q)\}$ is a uniformly representation stable sequence of $S_n$-representations.   Since $H^i(\cM_{g,\bullet};\Q)$ is an FI-module, and since $H^\ast(\cM_{g,n};\Q)$ is finite dimensional, Theorem~\ref{th:FIequiv} together with Jimenez Rolland's theorem shows that 
$\{H^\ast(\cM_{g,\bullet};\Q)\}$ is a graded FI-module of finite type.  This result, with bounds on the stability degree, etc., has been worked out by Jimenez Rolland~\cite{J2}  using the theorems in the present paper.
\end{remark}

\subsection{Albanese cohomology of Torelli groups}
\label{subsection:albanese}

The Torelli subgroups $\I_g^1$ and $\IAn$ of the mapping class groups and automorphism groups of free groups, respectively, are of interest in low-dimensional topology and combinatorial group theory (see below for definitions). However, very little  is known about $H^*(\I_g^1;\Q)$ or $H^*(\IAn;\Q)$. In this section we apply the theory of FI-modules to the subalgebra generated by first cohomology.

\begin{definition}[{\bf Albanese cohomology}]
Let $\Gamma$ be a finitely generated group, let $\Gamma^{\rm ab}$ be its abelianization, and let $\psi\colon\Gamma\to \Gamma^{\rm ab}$ be the natural quotient map.    Since $H^1(\Gamma^{\ab};\Q)\simeq H^1(\Gamma;\Q)$ and $H^*(\Gamma^{\rm ab};\Q)\simeq \bwedge^*H^1(\Gamma^{\rm ab};\Q)$, the map $\psi$ induces a homomorphism 
\[\psi^*\colon H^*(\Gamma^{\rm ab};\Q)\simeq \bwedge^*H^1(\Gamma;\Q)\to H^*(\Gamma;\Q).\] We define the \emph{Albanese cohomology} $H_{\rm Alb}^*(\Gamma;\Q)$ of $\Gamma$ to be the image of this map: \[H_{\rm Alb}^*(\Gamma;\Q)\coloneq \psi^*(H^*(\Gamma^{\rm ab};\Q))\subset H^*(\Gamma;\Q)\]
\end{definition}
We use the name ``Albanese cohomology'' in order to keep in mind the frequently encountered case where $\Gamma$ is the fundamental group of a compact Kahler manifold,  in which case $H_{\rm Alb}^*$ is the part of the cohomology coming from the associated Albanese variety.  Clearly $H^*_{\rm Alb}(\Gamma;\Q)$ can also be described as the subalgebra of $H^*(\Gamma;\Q)$ generated by $H^1(\Gamma;\Q)$.

\para{Albanese cohomology of the Torelli group}
Let $S_g^1$ be a compact, oriented genus $g\geq 2$ surface with one boundary component.  The mapping class group $\Mod(S_g^1)$ is the group of path components of the group $\Homeo^+(S_g^1,\partial S_g^1)$ of orientation-preserving homeomorphisms of $S_g^1$ fixing the boundary pointwise.  The action of $\Mod(S_g^1)$ on $H_1(S_g^1;\Z)$ preserves algebraic intersection number, which is a symplectic form on $H_1(S_g^1;\Z)$.  The \emph{Torelli group} $\I_g^1$ is defined to be the kernel of this action, and we have a well-known (see e.g.\ \cite{FM}) exact sequence, 
where $\Sp_{2g}\Z$ is the integral symplectic group:
\[1\to\I_g^1\to\Mod(S_g^1)\to\Sp_{2g}\Z\to 1\]

Very little is known about the cohomology $H^*(\I_g^1;\Q)$ of the Torelli group, or even about the subalgebra $H^*_{\rm Alb}(\I_g^1;\Q)$. For $g=2$, it follows from Mess~\cite{Mess} that $H^1_{\rm Alb}(\I_2^1;\Q)=H^1(\I_2^1;\Q)\simeq \Q^{\infty}$. For $g\geq 3$, Johnson 
proved that $\I_g^1$ is finitely generated, and gave an isomorphism
\begin{equation}
\label{eq:Johnsonhom}
H^1_{\rm Alb}(\I_g^1;\Q)=H^1(\I_g^1;\Q)\simeq \bwedge^3\Q^{2g}
\end{equation}
as $\Sp_{2g}\Z$-modules.   Hain~\cite[\S10]{Ha} computed $H^2_{\rm Alb}(\I_g^1;\Q)$ as an $\Sp_{2g}\Z$-module, and Sakasai~\cite{Sa} did the same for $H^3_{\rm Alb}(\I_g^1;\Q)$.  
We found many examples of nontrivial classes in $H^*_{\rm Alb}(\I_g^{1};\Q)$ in \cite{CF2}; these classes give a lower bound of order $\sim g^{i+2}$ for the dimension of $H^i_{\rm Alb}(\I^1_g;\Q)$.  
Beyond these coarse bounds, nothing is known about the precise dimension of 
$H^i_{\rm Alb}(\I_g^{1};\Q)$.

\begin{theorem}
\label{theorem:albanese1}
For each $i\geq 0$ there exists a polynomial $P_i(T)$ of degree at most $3i$ such that $\dim H^i_{\rm Alb}(\I^1_g;\Q)=P_i(g)$ for $g\gg i$.
\end{theorem}

Although it follows from Johnson's theorem that $\dim H^i_{\rm Alb}(\I^1_g;\Q)$ grows no faster than $O(g^{3i})$, to say that this dimension coincides exactly with a polynomial for large $g$ is much stronger.  
We emphasize that the dimension of $H^i_{\rm Alb}(\I_g^1;\Q)$ is unknown when $i \geq 3$; in particular we do not know what the polynomials produced by Theorem~\ref{theorem:albanese1} are.

In order to prove Theorem~\ref{theorem:albanese1} we will prove that $H^*_{\rm Alb}(\I_\bullet^1;\Q)$ is a graded FI-module of finite type.  One novelty here is that the FI-module structure on $H^*_{\rm Alb}(\I_\bullet^1;\Q)$ is in some sense not natural, but its existence nevertheless allows us to prove Theorem~\ref{theorem:albanese1}. 

We say an \emph{FI-group up to conjugacy} is a functor $\Gamma$ from $\FI$ to the category of groups and homomorphisms-up-to-conjugacy; in other words, for each finite set $S$ we have a group $\Gamma_S$, and for each $f\colon S\into T$ we have maps $f_*\colon \Gamma_S\to \Gamma_T$ so that the relevant diagrams commute up to conjugacy. Since conjugation acts trivially on group homology and cohomology, if $\Gamma$ is an FI-group up to conjugacy, then $H_*(\Gamma;k)$ is a graded FI-module, and $H^*(\Gamma;k)$ is a graded co-FI-algebra.

\begin{proof}[Proof of Theorem~\ref{theorem:albanese1}]
We define the FI-group up to conjugacy $\I_\bullet^1$ as follows.
Fix for each $g$ an ordered symplectic basis $\{a_1,b_1,\ldots,a_g,b_g\}$ for $H_1(S_g^1;\Z)$. Any injection $f\colon \m\into \n$ determines an injection $f_*\colon H_1(S_m^1;\Z)\into H_1(S_n^1;\Z)$ by $f_*(a_i)=a_{f(i)}$ and $f_*(b_i)=b_{f(i)}$. Any embedding $S_m^1\into S_n^1$ determines an inclusion $H_1(S_m^1;\Z)\into H_1(S_n^1;\Z)$ as a symplectic summand.  Such an embedding also determines an injection $\I_m^1\into \I_n^1$, given by 
extending each map to be the identity in the complement.  Moreover, although there are many embeddings $S_m^1\into S_n^1$ inducing a given map on homology, Johnson \cite[Theorem 1A]{JohnsonCR} proved that the resulting injections $\I_m^1\into \I_n^1$ are all conjugate in $\I_n^1$. As a consequence, these maps $f_*\colon \I_m^1\to I_n^1$ make $\I_\bullet^1$ into an FI-group up to conjugacy.

Of course, the co-FI-module structure on $H^*(\I_\bullet^1;\Q)$  is not unique or canonical; it depends on the initial choice of bases 
for $H_1(S_g^1;\Z)$.  But all that matters for us is the existence of a co-FI-module structure on $H^*(\I_\bullet^1;\Q)$ such that $H_1(\I_\bullet^1;\Q)=H^1(\I_\bullet^1;\Q)^*$ is a finitely generated FI-module.

Johnson's identification in \eqref{eq:Johnsonhom} comes from a natural isomorphism $H_1(\I_g^1;\Q)\simeq \bwedge^3(H_1(\Sigma_g^1;\Q))$. With respect to the FI-module structure above, we have $H_1(\Sigma_\bullet^1;\Q)\simeq M(1)\oplus M(1)$, so we obtain an isomorphism $H_1(\I_\bullet^1;\Q)\simeq  \bwedge^3(M(1)\oplus M(1))$. In particular, $H_1(\I_\bullet^1;\Q)$ is a finitely generated FI-module of weight $3$ by  Proposition~\ref{pr:schurfunctor}.  Proposition~\ref{pr:cofisubalgebra} implies that $H^*_{\rm Alb}(\I_\bullet^1;\Q)$ is a co-FI-algebra of finite type and slope $\leq 3$, and the desired conclusion follows from Theorem~\ref{thm:persinomial}.
\end{proof}

\para{Albanese cohomology of $\IAn$}
Let $F_n$ denote the free group of rank $n\geq 0$.  The Torelli subgroup $\IAn$ of $\Aut(F_n)$ is defined to be the subgroup consisting of those automorphisms that act trivially on $H_1(F_n;\Z)\simeq \Z^n$, giving the exact sequence \[1\to \IAn\to \Aut(F_n)\to\GL_n\Z\to 1.\] Magnus proved in 1934 that $\IAn$ is finitely generated for all $n\geq 0$.
The conjugation action of $\Aut(F_n)$ on $\IAn$ induces an action of $\GL_n\Z$ on the homology groups $H_i(\IAn;\Q)$.  Farb, Cohen--Pakianathan and Kawazumi (see e.g.\ \cite{Ka}) independently proved that for all $n\geq 0$, there exists a natural isomorphism
\begin{equation}
\label{eq:H1IAn}
H_1(\IAn;\Z)\simeq \bwedge^2 F_n^{\ab}\otimes_\Z (F_n^{\ab})^\ast
\end{equation}
as $\GL_n\Z$-representations.
 As with the Torelli group, basically nothing is known about $H^*_{\rm Alb}(\IAn;\Q)$;  other than $H^1_{\rm Alb}=H^1$; the only progress is Pettet's computation \cite{Pe} of $H^2_{\rm Alb}(\IAn;\Q)$.
  
\begin{theorem}
\label{theorem:albanese2}
For each $i\geq 0$ there exists a polynomial $P_i(T)$ of degree at most $3i$ such that $\dim H^i_{\rm Alb}(\IAn;\Q)=P_i(n)$ for $n\gg i$.
\end{theorem}

\begin{proof}
The proof proceeds along much the same lines as that of Theorem~\ref{theorem:albanese1}. One simplifying factor is that in this case we can actually define an FI-group $\IA_\bullet$ as follows. To $\n$ we associate the group $\IA_n$ defined above, and to an injection $f\colon \m\to \n$ we associate the following homomorphism $f_*\colon \IA_m\into\IAn$. Let $f_*\colon F_m\into F_n$ denote the inclusion induced by $x_i\mapsto x_{f(i)}$. Given $\varphi\in \IA_m$, we define $f_*\varphi\in \IAn$ to be the automorphism sending $x_{f(i)}\mapsto f_*(\varphi(x_i))$ and $x_j\mapsto x_j$ for those $j\not\in f(\m)$. It is easy to check that $(f\circ g)_*=f_*\circ g_*$, so this defines an FI-group $\IA_\bullet$. In particular, $H^*(\IA_\bullet;\Q)$ is a graded co-FI-algebra.

The isomorphism \eqref{eq:H1IAn} is provided by the Johnson homomorphism, which is compatible with the maps $f_*$, so we have  an isomorphism of FI-modules  
\[
H_1(\IA_\bullet;\Q)\simeq \bwedge^2 M(1) \otimes M(1) \simeq M(\y{2})\oplus M\big(\x\y{1,1}\x\big)\oplus M\big(\x\y{2,1}\x\big) \oplus M\Big(\x\y{1,1,1}\x\Big).
\]

The FI-module $H_1(\IA_\bullet;\Q)\simeq H^1(\IA_\bullet;\Q)^*$ is finitely generated of weight $3$ (we could also deduce this from Propositions~\ref{pr:tensor} and \ref{pr:schurfunctor}). By Proposition~\ref{pr:cofisubalgebra}  the graded co-FI-module $H^*_{\rm Alb}(\IA_\bullet;\Q)$ is of finite type and slope $\leq 3$. Theorem~\ref{thm:persinomial} now implies the theorem, and moreover that  the character of $H^i_{\rm Alb}(\IAn;\Q)$ as an $S_n$-representation is eventually given by a character polynomial.
\end{proof}

\subsection{Graded Lie algebras associated to the lower central series}
\label{section:gradedlielcc}

Recall that the \emph{lower central series}
\[\Gamma=\Gamma_1 >\Gamma_2 > \cdots\]
of a group $\Gamma$ is defined inductively by $\Gamma_1\coloneq
\Gamma$ and $\Gamma_{j+1}\coloneq [\Gamma,\Gamma_j]$. For simplicity, assume that $k$ is a field. The
\emph{associated graded Lie algebra} of $\Gamma$, denoted $\gr(\Gamma)$, is the
Lie algebra over $k$ defined by \[\gr(\Gamma)\coloneq \bigoplus_{j=1}^\infty \gr(\Gamma)_j=
\bigoplus_{j=1}^\infty (\Gamma_j/\Gamma_{j+1})\otimes k\] where the
Lie bracket is induced by the group commutator.  

When $\Gamma$ is finitely generated the graded vector space $\gr(\Gamma)$ is of finite type. 
The natural action of the automorphism group $\Aut(\Gamma)$ on $\Gamma$ preserves each $\Gamma_j$, and so acts on $\gr(\Gamma)$.  It is easy to see that this action factors through an action of $\Aut(\Gamma/[\Gamma,\Gamma])$. In particular, the action of inner automorphisms on $\gr(\Gamma)$ is trivial, so the group $\Out(\Gamma)$ of outer automorphisms of $\Gamma$ naturally acts on $\gr(\Gamma)$.    The action on $\gr(\Gamma)_1$ in grading 1 is nothing more than the representation of $\Aut(\Gamma/[\Gamma,\Gamma])=\Aut(H_1(\Gamma;\Z))$ on $H_1(\Gamma;k)\simeq H_1(\Gamma;\Z)\otimes k$.  

\para{The fundamental group of an FI-space}
The fundamental group $\pi_1$ is not a functor from topological spaces to groups, since it depends on a choice of basepoint; a map $Y\to Z$ only determines a homomorphism $\pi_1(Y)\to \pi_1(Z)$ up to conjugation. As a result, even though the configuration space $\Conf_\bullet(M)$ from \S\ref{section:fgconfig} is a co-FI-space, there is no co-FI-group $\pi_1(\Conf_\bullet(M))$ (only a co-FI-group up to conjugacy).

However, since conjugation acts trivially on $\gr(\pi_1(Z))$, there is a functor from $\Top$  to the category of graded Lie algebras which sends a space $Z$ to the graded Lie algebra $\gr(\pi_1(Z))$. 
Thus if $X$ is an FI-space $X$ (or even just a homotopy FI-space), the associated graded Lie algebra $\gr(\pi_1(X))$ is a graded FI-algebra. Similarly we do have a graded co-FI-algebra $\gr(\pi_1(\Conf_\bullet(M)))$ for any $M$.  When $M$ is the interior of a manifold with boundary as in \S\ref{ss:manifoldswithboundary}, Proposition~\ref{pr:CMhTop} implies that $\gr(\pi_1(\Conf_\bullet(M)))$ is  a graded $\FIsharp$-algebra.
In general, if $\Gamma$ is an FI-group up to conjugacy then  $\gr(\Gamma)$ is a graded FI-module.

\begin{theorem}[{\bf Finite generation of $\gr(\Gamma)$}]
\label{theorem:assoc:graded}
If $\Gamma$ is  an FI-group up to conjugacy (e.g.\ the fundamental group of a homotopy FI-space) and the FI-module $H_1(\Gamma;k)$ is finitely generated, the graded FI-module $\gr(\Gamma)$ is of finite type.
\end{theorem}

\begin{proof}
For any group $\Gamma$, the terms $\Gamma_i$ of the lower central series are by definition generated by iterated commutators of elements of $\Gamma$. This shows that $\gr(\Gamma)$ is always generated as a 
Lie algebra by $\gr(\Gamma)_1\simeq H_1(\Gamma;k)$.  Thus if $H_1(\Gamma;k)$ is finitely generated, Theorem~\ref{thm:fisubalgebra} implies the  graded FI-algebra $\gr(\Gamma)$ is of finite type.
\end{proof}

\begin{xample}[{\bf Free groups}]
Let $\Gamma$ be the $\FIsharp$-group which assigns to $S$ the free group $F_S=\langle x_s\,|\, s\in S\rangle$, and to a morphism $f\colon S\supset A\xrightarrow{\phi}B\subset T$ assigns  the map $F_S\to F_T$ given by $\phi(x_s)=x_{\phi(s)}$ if $i\in A$, and $\phi(x_s)=1$ if $i\in S-A$. It is easy to compute that $\gr(\Gamma)_1\simeq M(1)$. Since $\gr(\Gamma)_1$ is finitely generated, Theorem~\ref{theorem:assoc:graded} implies that $\gr(\Gamma)$ is of finite type.  It is known that $\gr(F_n)$ is isomorphic to the free Lie algebra on $n$ variables. Applying Theorem~\ref{th:FIequiv}, we conclude that the graded pieces of the free Lie algebra on $n$ variables are representation stable as representations of $S_n$. A variant of this result, for $\GL_n\C$-representations rather than $S_n$-representations, was proved in \cite[Corollary 5.7]{CF}.
\end{xample}

\begin{xample}[{\bf Pure braid groups}]
\label{example:pbgraded}
We proved in Proposition~\ref{pr:CMhTop} that the configuration space $\Conf_n(\R^2)$ of ordered $n$-tuples of distinct points in the plane is a homotopy $\FIsharp$-space.  The pure braid group $P_n$ on $n$ strands is the fundamental group $P_n=\pi_1(\Conf_n(\R^2))$, so this shows that $P_\bullet$ is an $\FIsharp$-group up to conjugacy.
\remove{We proved in Example~\ref{example:arnoldalgebra} that $H_1(P_\bullet;\Q)\simeq M(\y{2})$, }
\antiremove{One can compute directly that $H_1(P_\bullet;\Q)\simeq M(\y{2})$, }
 so Theorem~\ref{theorem:assoc:graded} implies that $\gr(P_\bullet)$ is a graded $\FIsharp$-module of finite type. 

Hain~\cite{Ha} proved that there is an isomorphism of $S_n$-representations $\gr(P_n)\simeq \gr(\fp_n)$, where $\fp_n$ is the \emph{Malcev Lie algebra} of $P_n$. Applying Theorem~\ref{th:FIequiv} to Example~\ref{example:pbgraded} implies the following theorem, which confirms Conjecture 5.15 of \cite{CF}.

\begin{theorem}[{\bf Representation stability for $\fp_n$}]
\label{theorem:malcevpn}
For each fixed $i\geq 1$, the sequence $\{\fp_n^i\}$ of grading-$i$ pieces of $\fp_n$ is 
a uniformly representation stable sequence of $S_n$--representations.
\end{theorem}

Drinfeld--Kohno (see \cite{Ko}) actually found an explicit presentation of $\fp_n$; for another approach to Theorem~\ref{theorem:malcevpn}, we could apply Theorem~\ref{thm:fisubalgebra} directly to their presentation.

\begin{remark}
The pure braid groups $P_n$ are examples of \emph{pseudo-nilpotent} groups, so that 
$H^*(\fp_n;\Q)\simeq H^*(P_n;\Q)$ for all $n$.  It was already proved in \cite{CF} that the sequence $\{H^i(P_n;\Q)\}$ is uniformly representation stable for each $i \geq 0$, so one is tempted to derive Theorem~\ref{theorem:malcevpn} directly from  \cite[Theorem 5.3]{CF}, which states the equivalence of uniform representation stability for a Lie algebra and for  its (co)homology.  However, that theorem was only proved in \cite{CF} in the context of stability for $\SL_n\C$-representations and $\GL_n\C$-representations. Indeed the ``strong stability'' hypothesis assumed in that theorem almost never holds for $S_n$-representations (and it does not hold here).
\end{remark}
\end{xample}

\begin{xample}[{\bf The Torelli group}]
In \S\ref{subsection:albanese} we encountered the Torelli group $\I_g^1$.   The action of 
the mapping class group $\Mod(S_g^1)$ on $\I_g^1$ by conjugation induces a well-defined action of $\Sp_{2g}\Z$ on $\gr(\I_g^1)$, taking $k=\Q$.    A finite presentation for $\gr(\I_g^1)$ as a Lie algebra has been given by Habegger--Sorger~\cite{HS}, extending the fundamental computation of Hain~\cite{Ha} in the  case of closed surfaces.  Hain also worked out the first few graded terms of this Lie algebra explicitly as $\Sp_{2g}\Z$-representations. Getzler--Garoufalidis (personal communication) have recently given more detailed computations in this direction. However, exact computations in arbitrary degrees seem out of reach. The situation for $\gr(\IA_n)$ as a $\GL_n\Z$-representation is the same.  Even so, we have the following theorems.

\begin{theorem}
For each $i\geq 0$ the dimension $\dim(\gr(\I_g^1)^i)$ is polynomial in $g$ for $g\gg i$.\end{theorem}
\begin{theorem}
For each $i\geq 0$ the dimension $\dim(\gr(\IAn)^i)$ is polynomial in $n$ for $n\gg i$.
\end{theorem}

\begin{proof}
We showed in the proof  of Theorems~\ref{theorem:albanese1} and \ref{theorem:albanese2} that $\I_\bullet^1$ is an FI-group up to conjugacy and $\IA_\bullet$ is an FI-group. Moreover, we saw that the FI-modules $H_1(\I_\bullet^1;\Q)\simeq  \bwedge^3(M(1)\oplus M(1))$ and $H_1(\IA_\bullet;\Q)\simeq \bwedge^2 M(1) \otimes M(1)$ are finitely generated. Theorem~\ref{theorem:assoc:graded} now implies that the graded FI-algebras $\gr(\I_\bullet^1)$ and $\gr(\IA_\bullet)$ are of finite type, and the claim follows from Theorem~\ref{thm:persinomial}.
\end{proof}

\end{xample}

\begin{footnotesize}
\noindent
\begin{tabular*}{\linewidth}[t]{@{}p{\widthof{Department of Mathematics}+0.5in}@{}p{\widthof{Department of Mathematics}+0.5in}@{}p{\linewidth - \widthof{Department of Mathematics} - \widthof{Department of Mathematics}-1in}@{}}
{\raggedright
Thomas Church\\
Department of Mathematics\\
Stanford University\\
450 Serra Mall\\
Stanford, CA 94305\\
\myemail{church@math.stanford.edu}}
&
{\raggedright
Jordan Ellenberg\\
Department of Mathematics\\
University of Wisconsin\\
480 Lincoln Drive\\
Madison, WI 53706\\
\myemail{ellenber@math.wisc.edu}}
&
{\raggedright
Benson Farb\\
Department of Mathematics\\
University of Chicago\\
5734 University Ave.\\
Chicago, IL 60637\\
\myemail{farb@math.uchicago.edu}}
\end{tabular*}\hfill
\end{footnotesize}

\end{document}